%% file: article.tex
\documentclass{article}
\usepackage{hyperref}
\hypersetup{hypertexnames = false, bookmarksdepth = 2, bookmarksopen = true, colorlinks, linkcolor = black, citecolor = black, urlcolor = black, pdfstartview={XYZ null null 1}}

\usepackage{amsfonts}
\usepackage[fleqn, leqno]{amsmath}
\usepackage{amsthm}

\usepackage[noabbrev,capitalise]{cleveref}

\usepackage{filecontents}

\usepackage[backend=biber, maxbibnames=99, datamodel=mrnumber]{biblatex}

\usepackage{adjustbox}
\usepackage{booktabs}
\usepackage{diagbox}
\usepackage{dynkin-diagrams}
\usepackage{enumitem}
\usepackage{float}
\usepackage{mathtools}
\usepackage{multirow}
\usepackage{parskip}
\usepackage{scalerel}
\usepackage{thmtools}
\usepackage{tikz-cd}
\usepackage{xparse}
\usepackage{xspace}

\usepackage[utf8]{inputenc}
\usepackage[T1]{fontenc}
\usepackage{libertine}
\usepackage[libertine]{newtxmath}
\usepackage[scaled=0.83]{beramono}
\usepackage{eucal}
\usepackage{microtype}
\DeclareFieldFormat{mrnumber}{%
  MR\addcolon\space
  \ifhyperref
    {\href{http://www.ams.org/mathscinet-getitem?mr=1#1}{\nolinkurl{#1}}}
    {\nolinkurl{#1}}}

\renewbibmacro*{doi+eprint+url}{%
  \iftoggle{bbx:doi}
    {\printfield{doi}}
    {}%
  \newunit\newblock
  \printfield{mrnumber}%
  \newunit\newblock
  \iftoggle{bbx:eprint}
    {\usebibmacro{eprint}}
    {}%
  \newunit\newblock
  \iftoggle{bbx:url}
    {\usebibmacro{url+urldate}}
    {}}

\declaretheoremstyle[
  spaceabove = 3pt,
  spacebelow = 3pt,
]{lecture}
\theoremstyle{lecture}
\newtheorem{theorem}{Theorem}

\newtheorem{definition}[theorem]{Definition}
\newtheorem{example}[theorem]{Example}

\newtheorem{lemma}[theorem]{Lemma}
\newtheorem{proposition}[theorem]{Proposition}

\newtheorem{remark}[theorem]{Remark}

\newtheorem{alphatheorem}{Theorem}
\newtheorem{alphaconjecture}[alphatheorem]{Conjecture}
\newtheorem{alphacorollary}[alphatheorem]{Corollary}
\newtheorem{alphaproposition}[alphatheorem]{Proposition}

\crefname{alphatheorem}{Theorem}{Theorems}
\crefname{alphaconjecture}{Conjecture}{Conjectures}
\crefname{alphaproposition}{Proposition}{Propositions}

\makeatletter
\def\gitfootnote{\gdef\@thefnmark{}\@footnotetext}
\makeatother

\makeatletter
\def\paragraph{\@startsection{paragraph}{4}{\z@}{3.25ex \@plus 1ex \@minus .2ex}{-1em}{\normalfont\normalsize\bfseries}}
\makeatother

\crefformat{enumi}{#2\textup{(#1)}#3}

\mathchardef\mhyphen="2D
\newcommand\dash{\nobreakdash-\hspace{0pt}}

\let\oldbigwedge\bigwedge
\renewcommand\bigwedge{\oldbigwedge\nolimits}

\newcommand\mfb{\ensuremath{\mathfrak{b}}}
\newcommand\mfg{\ensuremath{\mathfrak{g}}}
\newcommand\mfl{\ensuremath{\mathfrak{l}}}
\newcommand\mfp{\ensuremath{\mathfrak{p}}}

\newcommand\mft{\ensuremath{\mathfrak{t}}}
\newcommand\mfn{\ensuremath{\mathfrak{n}}}

\newcommand\fundamental[1]{\ensuremath{\omega_{#1}}}
\newcommand\simple[1]{\ensuremath{\alpha_{#1}}}

\newcommand\roots{\ensuremath{\mathrm{R}}}
\newcommand\simpleroots{\ensuremath{\mathrm{S}}}

\DeclareMathOperator\characters{\mathrm{X}}
\DeclareMathOperator\cocharacters{\mathrm{Y}}
\newcommand\weightlattice[1]{\ensuremath{\characters(T)}_{#1}}
\newcommand\dominantweights[1]{\ensuremath{\characters(T)_{#1}^+}}

\newcommand\bundle[1]{\ensuremath{\mathcal{E}^{#1}}}
\newcommand\representation[2]{\ensuremath{\mathrm{V}_{#2}^{#1}}} 

\newcommand\weyl{\ensuremath{\mathrm{W}}}
\newcommand\cosets{\ensuremath{\,\prescript{\mathfrak{l}}{}{\mathrm{W}}}}

\newcommand\ab{\ensuremath{\mathrm{ab}}}
\newcommand\CE{\ensuremath{\mathrm{CE}}}
\newcommand\LLL{\ensuremath{\mathbf{L}}}
\newcommand\reg{\ensuremath{\mathrm{reg}}}
\newcommand\semisimple{\ensuremath{\mathrm{ss}}}
\newcommand\tangent{\ensuremath{\mathrm{T}}}

\DeclareMathOperator\Bl{Bl}
\DeclareMathOperator\coh{coh}
\DeclareMathOperator\coker{coker}
\DeclareMathOperator\Ext{Ext}
\DeclareMathOperator\FF{F}
\DeclareMathOperator\GG{G}
\DeclareMathOperator\GL{GL}
\DeclareMathOperator\HH{H}
\DeclareMathOperator\HHHH{HH}
\DeclareMathOperator\Hom{Hom}
\DeclareMathOperator\gr{gr}
\DeclareMathOperator\Gr{Gr}

\DeclareMathOperator\fanoindex{i}

\DeclareMathOperator\Kostant{K}
\DeclareMathOperator\Lie{Lie}
\DeclareMathOperator\LGr{LGr}
\DeclareMathOperator\OGr{OGr}
\DeclareMathOperator\Pic{Pic}
\DeclareMathOperator\rep{rep}
\DeclareMathOperator\rk{rk}
\DeclareMathOperator\SGr{SGr}

\DeclareMathOperator\SP{Sp}
\DeclareMathOperator\Sym{Sym}
\DeclareMathOperator\td{td}

\newcommand\upperbound{\ensuremath{M}}

\newcommand\groundfield{\ensuremath{\mathbf{k}}}

\numberwithin{equation}{section}

\title{Hochschild cohomology of generalised Grassmannians}
\author{Pieter Belmans \and Maxim Smirnov}

\begin{document}

\maketitle


\begin{abstract}
  We compute the Hochschild--Kostant--Rosenberg decomposition of the Hoch\-schild cohomology of generalised Grassmannians, i.e.,~partial flag varieties associated to maximal parabolic subgroups in a simple algebraic group, in terms of representation-theoretic data. We explain how the decomposition is concentrated in global sections for the (co)minuscule and (co)adjoint generalised Grassmannians, and conjecture that for (almost) all other cases the same vanishing of the higher cohomology does not hold. Our methods give an explicit partial description of the Gerstenhaber algebra structure for the Hochschild cohomology of cominuscule and adjoint generalised Grassmannians. We also consider the case of adjoint partial flag varieties in type~A, which are associated to certain submaximal parabolic subgroups.
\end{abstract}

%
%

\section{Introduction}
For partial flag varieties, or compact homogeneous spaces, it has been an important question to compute topological and geometric invariants of~$G/P$ in terms of representation-theoretic data. This has a long and rich history, starting with the works of Borel and Hirzebruch. Examples of these invariants are singular cohomology \cite{MR0102800} or quantum cohomology \cite{MR2077592}, and equivariant variations thereupon \cite{MR1649623,MR2359822}.

In this paper we consider another important algebro-geometric invariant: \emph{Hochschild cohomology}, denoted~$\HHHH^\bullet(G/P)=\bigoplus_{i\geq 0}\HHHH^i(G/P)$. It controls the (generalised) deformation theory \cite{MR2183254,MR2238922}, it is related to Poisson geometry \cite{MR3765972},
and it comes equipped with a rich algebraic structure. It has not been considered before for partial flag varieties, and the results in this paper suggest many interesting features arising in this setting.

One can explicitly compute~$\HHHH^i(X)$ of a smooth variety~$X$ (at least as a vector space, and in characteristic~0), using the \emph{Hochschild--Kostant--Rosenberg decomposition} \cite{MR2141853}, which expresses it as the direct sum
\begin{equation}
  \label{equation:HKR}
  \HHHH^i(X)\cong\bigoplus_{p+q=i}\HH^q(X,\bigwedge^p\tangent_X)
\end{equation}
in terms of polyvector fields.
Whilst formally similar to the Hodge decomposition of the cohomology of~$X$, the right-hand side is more complicated to compute, as for instance there are no symmetries induced by Serre duality or Hodge symmetry.


In this paper we discuss the case where~$X$ is of the form~$G/P$, where~$G$ is a simple algebraic group defined over an algebraically closed field of characteristic~0, and~$P$ is either
\begin{itemize}
  \item a \emph{maximal} parabolic subgroup, in which case we call~$G/P$ a \emph{generalised Grassmannian}, including when~$G$ is of exceptional type;
  \item the submaximal parabolic subgroup in type~$\mathrm{A}_n$ corresponding to the \emph{adjoint} case, in which case we have that~$G/P=\operatorname{Fl}(1,n,n+1)$.
\end{itemize}

The starting point for this paper is the question whether these varieties are what we will call \emph{Hochschild global}, i.e.,~whether
\begin{equation}
  \label{equation:hochschild-global}
  \HH^q(G/P,\bigwedge^p\tangent_{G/P})=0\qquad\forall q\geq 1,
\end{equation}
so that the Hochschild--Kostant--Rosenberg decomposition \eqref{equation:HKR} is concentrated in global sections. Amongst experts there was the expectation that this would indeed be the case for partial flag varieties.

Upon replacing exterior powers of the tangent bundle by \emph{symmetric} powers of the tangent bundle, it can be shown using Grauert--Riemenschneider vanishing and the Leray spectral sequence (see \cite[\S A2]{MR0522032}) that
\begin{equation}
  \label{equation:symmetric-vanishing}
  \HH^q(G/P,\Sym^p\tangent_{G/P})=0\qquad\forall q\geq 1,
\end{equation}
as~$\Sym^\bullet\tangent_{G/P}$ are the functions on the total space of the cotangent bundle. So \eqref{equation:hochschild-global} can be seen as an odd version of the vanishing in \eqref{equation:symmetric-vanishing}.

Instead of Hochschild cohomology, one could also study \emph{Hochschild homology}. Here the Hochschild--Kostant--Rosenberg decomposition takes on the form
\begin{equation}
  \label{equation:HKR-homology}
  \HHHH_i(X)\cong\bigoplus_{p-q=i}\HH^q(X,\Omega_X^p),
\end{equation}
which involves more familiar invariants when working over the complex numbers: the pieces of the Hodge decomposition. The description of these is amenable to topological methods as in \cite[\S24]{MR0110105}. For Hochschild cohomology such topological methods are not available and new tools are needed.

\paragraph{Algebraic structures}
Hochschild cohomology comes equipped with a rich structure, namely that of a \emph{Gerstenhaber algebra}. This combines a graded-commutative cup product with a graded Lie algebra structure of degree~$-1$, the Gerstenhaber bracket, which are related via the Poisson identity. Two important features of the Gerstenhaber bracket are that
\begin{enumerate}
  \item in the setting of $G/P$ the degree-1 component $\HHHH^1(G/P)$ is a Lie subalgebra
    given by $\Lie G=\mathfrak{g}$ (outside a few exceptional cases, see \cref{remark:exceptional-isomorphisms-exotic,lemma:exceptional-tangent-bundle}),
    which equips all the~$\HHHH^i(G/P)$ with the structure of a~$\mathfrak{g}$\dash representation;
  \item the self-bracket~$[\alpha,\alpha]\in\HHHH^3(G/P)$ for a class~$\alpha\in\HHHH^2(G/P)=\HH^0(G/P,\bigwedge^2\tangent_{G/P})$ can be identified with the Schouten self-bracket\footnote{This identification needs to use Kontsevich's refined Hochschild--Kostant--Rosenberg isomorphism giving an isomorphism of Gerstenhaber algebras between Hochschild cohomology and the cohomology of exterior powers of the tangent bundle. But by the vanishing~$\HH^i(G/P,\tangent_{G/P})=0$ for all~$i\geq 1$ \cite[Theorem~VII]{MR0182022}, so that~$G/P$ is (locally) rigid as a variety, we obtain an identification of the two brackets for classes of degree~2. See \cref{subsection:hochschild-cohomology} for more details.}, whose vanishing is precisely the condition that a bivector~$\alpha\in\HH^0(G/P,\bigwedge^2\tangent_{G/P})$ gives a Poisson structure.
\end{enumerate}
The first feature gives us a convenient method to describe the vector spaces~$\HHHH^i(G/P)$ as representations of~$\mathfrak{g}$. The second feature on the other hand highlights that one of the natural next steps in the description of the Gerstenhaber algebra structure (namely the self-bracket of two classes in degree~2) is very complicated, as this description is only known for the generalised Grassmannians~$\mathbb{P}^3$ and~$Q^3$ \cite{MR3066408,MR3765972}, with a classification in higher dimensions being wide open, see also \cref{remark:gerstenhaber}.

\paragraph{Vanishing}
The first result we describe is a positive answer to the vanishing question suggested above in an important class of examples. The notions of (co)minuscule and (co)adjoint are recalled in \cref{subsection:partial-flag-varieties}, in particular the following theorem concerns the varieties listed in \cref{table:cominuscule,table:adjoint,table:coadjoint}.

\begin{alphatheorem}[Vanishing]
  \label{theorem:vanishing}
  Let~$G/P$ be either a generalised Grassmannian which is (co)minu\-scule or (co)adjoint, or an adjoint partial flag variety in type~$\mathrm{A}_n$. Then
  \begin{equation}
    \HH^q(G/P,\bigwedge^p\tangent_{G/P})=0\qquad\forall q\geq 1.
  \end{equation}
\end{alphatheorem}
For generalised Grassmannians this vanishing result can in fact be deduced from the vanishing results in \cite{MR1016881}, but we will give a streamlined proof below. The vanishing in the adjoint case in type~$\mathrm{A}_n$ is new and a representation-theoretic proof is given. Alternatively the vanishing can be deduced using the description of the adjoint partial flag variety in type~$\mathrm{A}_n$ as~$\mathbb{P}(\tangent_{\mathbb{P}^n})$.

The complication (outside the cominuscule case) is that~$\bigwedge^i\tangent_{G/P}$ is an equivariant vector bundle, but it is not completely reducible. Therefore one cannot immediately apply the Borel--Weil--Bott theorem. This can be dealt with by using an appropriate filtration on the exterior powers of the tangent bundle so that the associated graded is completely reducible, and one has a spectral sequence \eqref{equation:konno} computing the cohomology we are interested in.

\paragraph{An explicit description}
In almost all cases covered by \cref{theorem:vanishing} we can actually describe the Hochschild cohomology~$\HHHH^i(G/P)$ as a representation of the Lie algebra~$\HHHH^1(G/P)\cong\mathfrak{g}$. By \cref{remark:no-minuscule} we can and will ignore the minuscule case. In the cominuscule case we can state the following theorem giving this description.

\begin{alphatheorem}[Cominuscule decomposition]
  \label{theorem:cominuscule-decomposition}
  Let~$G/P$ be a cominuscule generalised Grassmannian, where~$P$ is associated to the cominuscule weight~$\fundamental{k}$. Then
  \begin{equation}
    \label{equation:cominuscule-decomposition}
    \begin{aligned}
      \HHHH^i(G/P)
      &\cong\HH^0(G/P,\bigwedge^i\tangent_{G/P}) \\
      &\cong\bigoplus_{\mathclap{\substack{w\in\cosets \\ \ell(w)=\dim G/P-i}}}\representation{w\cdot0+\fanoindex_{G/P}\fundamental{k}}{\mathfrak{g}}
    \end{aligned}
  \end{equation}
  as representations of~$\HHHH^1(G/P)\cong\mathfrak{g}$, where~$\mathfrak{l}$ is the Lie algebra of the Levi quotient of~$P$.
\end{alphatheorem}

Here~$\fanoindex_{G/P}$ denotes the \emph{index} of~$G/P$, i.e.,~the maximal integer~$r$ such that we can write the anticanonical line bundle~$\smash{\omega_{G/P}^\vee}$ as~$\mathcal{L}^{\otimes r}$ for an ample line bundle~$\mathcal{L}$, and~$\cosets$ are the minimal length coset representatives of the Weyl group~$\weyl_{\mathfrak{l}}$ in~$\weyl_{\mathfrak{g}}$. This decomposition is obtained from Kostant's theorem on the Lie algebra cohomology for the nilradical of a parabolic subalgebra \cite[Corollary~8.2]{MR0142696}.

Outside this particularly nice situation the description of the Hochschild cohomology becomes more complicated, and no existing results can be applied. The adjoint case is closest in complexity to the cominuscule case, in which case we obtain the following description. For notational ease, we will write
\begin{equation}
  \label{equation:kostant}
  \Kostant(G,P,i,j)
  \coloneqq
  \bigoplus_{\mathclap{\substack{w\in\cosets \\ \ell(w)=\dim G/P-i}}}\representation{w\cdot0+(\fanoindex_{G/P}+j)\fundamental{k}}{\mathfrak{g}}
\end{equation}
as a~$\mathfrak{g}$\dash representation, where the right-hand side describes the result of a suitable modification of Kostant's description of the Lie algebra cohomology of the nilradical of~$\Lie P$. For our application we want to restrict this sum to those weights~$w\cdot0+(\fanoindex_{G/P}+j)\fundamental{k}$ which are regular, for which we will use the notation
\begin{equation}
  \label{equation:restricted-kostant}
  \Kostant(G,P,i,j)^\reg.
\end{equation}

\begin{alphatheorem}[Adjoint decomposition]
  \label{theorem:adjoint-decomposition}
  Let~$G/P$ be an adjoint generalised Grassmannian, or the adjoint partial flag variety of type~$\mathrm{A}_n$. Then~$\dim G/P=2r+1$ for some~$r$, and
  \begin{equation}
    \label{equation:adjoint-decomposition}
    \begin{aligned}
      &\HHHH^i(G/P)
      \cong\HH^0(G/P,\bigwedge^i\tangent_{G/P}) \\
      &\qquad\cong
      \begin{cases}
        \left( \bigoplus_{p=0}^{\lfloor\frac{i}{2}\rfloor}\Kostant^\reg(G,P,i-2p,p) \right)
        \oplus
        \left( \bigoplus_{p=0}^{\lfloor\frac{i-1}{2}\rfloor}\Kostant^\reg(G,P,i-2p-1,p+1) \right)
        & i\leq r \\
        \left( \bigoplus_{p=0}^{\lfloor\frac{2r-i}{2}\rfloor}\Kostant^\reg(G,P,i+1+2p,-p-1) \right)
        \oplus
        \left( \bigoplus_{p=0}^{\lfloor\frac{2r-i+1}{2}\rfloor}\Kostant^\reg(G,P,i+2p,-p-2) \right)
        & i\geq r
      \end{cases}
    \end{aligned}
  \end{equation}
  as representations of~$\HHHH^1(G/P)\cong\mathfrak{g}$.
\end{alphatheorem}

All these results only give a small piece of the whole Gerstenhaber algebra structure, see \cref{remark:gerstenhaber} for more details.

\paragraph{Non-vanishing}
In addition to the vanishing results and explicit descriptions obtained above we want to highlight the following surprising phenomenon, showing that the expected vanishing does not always hold, even for maximal parabolic subgroups, contrary to experts' belief.

\begin{alphaproposition}
  \label{proposition:non-vanishing}
  For all~$n\geq 4$ we have that
  \begin{equation}
    \HH^1(\SGr(3,2n),\bigwedge^2\tangent_{\SGr(3,2n)})\cong\representation{\fundamental{4}}{\mathfrak{sp}_{2n}},
  \end{equation}
  as representations of~$\HHHH^1(\SGr(3,2n))\cong\mathfrak{sp}_{2n}$.
\end{alphaproposition}
Here~$\SGr(3,2n)$ is the symplectic Grassmannian parametrising~3\dash dimensional isotropic subspaces of a~$2n$\dash dimensional symplectic vector space, associated to the maximal parabolic subgroup~$\mathrm{P}_3$ of a simple group of type~$\mathrm{C}_n$.

In particular, not every generalised Grassmannian is Hochschild global in the sense of \eqref{equation:hochschild-global}.
That this can happen for \emph{full} flag varieties became clear after computer calculations by Knutson and Schedler for the flag variety~$G/B$ in type~A (see \cite[Remark~2.2]{MR4057490}) and the generalisation of these computer calculations to all Dynkin types (but still for the full flag variety) by the authors. But any systematic description is out of reach in this setting.

\paragraph{Bott vanishing}
Bott vanishing is a strong vanishing property for the sheaf cohomology of certain vector bundles on~$\mathbb{P}^n$,
subsequently generalised to certain other settings.
We refer to \cref{subsection:bott-vanishing} for more context.
In particular, it is expected to fail for all~$G/P$ which are not projective space.
\Cref{proposition:non-vanishing} gives the following corollary.
\begin{alphacorollary}
  \label{corollary:no-bott-vanishing-for-sgr-3-2n}
  Bott vanishing fails for~$\SGr(3,2n)$.
\end{alphacorollary}
This gives the first instance of the failure of Bott vanishing for generalised Grassmannians in the non-cominuscule case.
For generalised Grassmannians this was only known in the cominuscule case \cite[\S4.3]{MR1464183}
using the method used for \cref{theorem:cominuscule-decomposition}.

\paragraph{Conjectural non-vanishing}
The methods to prove the non-vanishing in \cref{proposition:non-vanishing} can be implemented in computer algebra, and computations up to rank~10 for maximal parabolic subgroups show that the vanishing result in \cref{theorem:vanishing} is in fact (very close to) an if and only if in these cases. Let us phrase this optimistically as the following conjecture, with an important caveat being discussed in \cref{remark:caveat}.

\begin{alphaconjecture}
  \label{conjecture:non-vanishing}
  Let~$G/P$ be a generalised Grassmannian which is \emph{not} (co)minuscule or (co)adjoint. Then
  \begin{equation}
    \HH^q(G/P,\bigwedge^p\tangent_{G/P})\neq0
  \end{equation}
  for some~$p\geq 2$ and~$q\geq 1$.
\end{alphaconjecture}

\begin{remark}
  \label{remark:caveat}
  The conjecture is phrased optimistically: there is one family of generalised Grassmannians where our computational methods do not give a definitive answer. These are the orthogonal Grassmannians~$\OGr(n-1,2n+1)$ for~$n\geq 4$ associated to a simple algebraic group of type~$\mathrm{B}_n$ and the maximal parabolic subgroup~$\mathrm{P}_{n-1}$. As explained in \cref{subsection:open-case-Bn-Pn-1} our methods are inconclusive because it is possible that all higher cohomology gets cancelled in the spectral sequence we use to analyse the Hochschild--Kostant--Rosenberg decomposition.
\end{remark}

For all other cases up to rank~10 (except $\mathrm{E}_8$) the computer calculations precisely tell us that all generalised Grassmannians in \cref{conjecture:non-vanishing} are \emph{not} Hochschild global in the sense of \eqref{equation:hochschild-global}.
Hence there is ample computational evidence for the conjecture.

The conjecture is only phrased for generalised Grassmannians. For~$G$ of rank up to~3, and also in type~$\mathrm{A}_4$, it has been computationally confirmed in \cite{MR4187255} that~$G/P$ is Hochschild global in the sense of \eqref{equation:hochschild-global} for \emph{all} possible parabolic subgroups~$P$. For~$G/B$ in type~$\mathrm{A}_4$ this was also confirmed in \cite[Example~3.3]{1801.08261v1}. This is consistent with the computations due to Knutson--Schedler and ourselves for full flag varieties in arbitrary type, where the non-vanishing in type A starts for rank~$\geq 5$, and e.g.~in other types occurs for~$G/B$ in type~$\mathrm{D}_5$ or~$\mathrm{F}_4$.

\paragraph{Related works}
Some related computations appear in \cite{MR4187255,1801.08261v1}.
The methods in op.~cit.~are different from ours, using the Bernstein--Gelfand--Gelfand resolution in relative Lie algebra cohomology to compute multiplicities of representations in the sheaf cohomology of polyvector fields, using \cite[Proposition~2.8]{MR3872326}.

We expect the interaction between the methods from op.~cit. and this paper will prove useful in understanding the precise conjecture for generalised Grassmannians, and the general picture for arbitrary partial flag varieties.

In \cite{2104.07626v2} the Hochschild cohomology of Fano 3-folds is computed, also using representation-theoretic methods.

\paragraph{Structure of the paper}
We start with a lengthy introduction in \cref{section:preliminaries}, in order to make the computations accessible to algebraic geometers without a representation-theoretic background. In \cref{section:vanishing} we give a self-contained proof of \cref{theorem:vanishing}. This result can be deduced from \cite{MR1016881}, but we will reprove it to set up the notation and machinery for later arguments.

In \cref{section:description} we will then prove \cref{theorem:cominuscule-decomposition,theorem:adjoint-decomposition}. We will illustrate both descriptions in some examples.

In \cref{section:counterexample} we show that \emph{not} every generalised Grassmannian is Hochschild global in sense of \eqref{equation:hochschild-global} by explicitly studying the first example where this is the case.
We moreover discuss the phenomenon discussed in \cref{remark:caveat}, and the link with Bott vanishing.

\paragraph{Acknowledgements}
We would like to thank Michel Brion, Friedrich Knop, Alexander Kuznetsov, Catharina Stroppel and Michel Van den Bergh for interesting discussions.

We want to especially thank Travis Schedler and Allen Knutson for sharing their computations for full flag varieties in type~A, which suggested that the situation was even more interesting than a priori expected.

We want to thank Nicolas Hemelsoet for interesting discussions regarding \cref{example:OGr-3-9}, and sharing his computations.

The first author acknowledges the support of the FWO (Research Foundation---Flanders). We want to thank the Max Planck Institute for Mathematics for the pleasant working conditions during the start of this project, and its high performance computing infrastructure.

\section{Preliminaries}
\label{section:preliminaries}

\subsection{Setup and notation}
\label{subsection:new-setup-and-notation}

For the purposes of our paper it is enough to work with homogenous spaces $G/P$,
where $G$ is assumed to be a connected simply-connected simple algebraic group
over $\groundfield$.
Most of the time the parabolic subgroup $P$ is assumed to be maximal.
Below we recall the relevant notation, mostly following \cite[Section II.1]{MR2015057}.

\paragraph{Roots and coroots}
Let $G$ be a connected simply-connected simple algebraic group over $\groundfield$.
Let $T \subset G$ be a maximal torus and let $\characters(T)$ be its group of characters.
The group $G$ acts on its Lie algebra $\mfg = \Lie(G)$ via the adjoint action and
we obtain a decomposition into root spaces
\begin{equation}
  \mfg = \mft \oplus \bigoplus_{\alpha \in \roots} \mfg_{\alpha},
\end{equation}
where $\mft = \Lie (T)$ and $\roots=\roots(G) \subset \characters(T)$ are the \emph{roots} of $G$.

Let $\cocharacters(T)$ be the group of cocharacters of $T$. We denote by
\begin{equation}\label{eq:pairing-weights-coweights}
  \langle-,-\rangle \colon \characters(T) \times \cocharacters(T) \to \mathbb{Z}
\end{equation}
the natural \emph{perfect pairing} that gives rise to an isomorphism of abelian groups
\begin{equation}
  \cocharacters(T) \cong \Hom_{\mathbb{Z}}(\characters(T), \mathbb{Z}).
\end{equation}
For each root $\alpha \in \roots$ there is a uniquely defined \emph{coroot} $\alpha^\vee \in \cocharacters(T)$,
and the set of roots $\roots$ together with the map $\alpha \mapsto \alpha^\vee$
defines a root system in $\characters(T)_\mathbb{R}$ in the sense of \cite[Ch.~VI, \S1, no.~1]{MR0240238}.
For each $\alpha \in \roots$ we denote by $s_\alpha$ the corresponding reflection on $\characters(T)$
\begin{equation}
  s_{\alpha}(\lambda) = \lambda - \langle \lambda , \alpha^\vee \rangle \alpha,
\end{equation}
and we extend it to $\characters(T)_\mathbb{R}$ (resp. $\characters(T)_\mathbb{Q}$) by extending
$\alpha^\vee \in \cocharacters(T) \cong \characters(T)^\vee$ to $\characters(T)_\mathbb{R}$ (resp. $\characters(T)_\mathbb{Q}$).

The reflections $s_\alpha$ for $\alpha \in\roots$ generate the \emph{Weyl group} of $G$
\begin{equation}
  \weyl_G = \langle s_\alpha \mid \alpha \in \roots \rangle \cong \mathrm{N}_{G}(T)/T.
\end{equation}
The Weyl group acts linearly on $\characters(T)$ and $\cocharacters(T)$ and leaves pairing \eqref{eq:pairing-weights-coweights}
invariant.

\paragraph{Weights and coweights}
Let $\roots^+ \subset \roots$ be a subset of \emph{positive roots} and $\simpleroots \subset \roots^+$
be the \emph{simple roots}. We denote by $\roots^- = -\roots^+$ the \emph{negative roots}.
We define an order $\leq$ on $\characters(T)$ by setting
\begin{equation}
  \lambda \leq \mu \quad \iff \quad \mu - \lambda \in \sum_{\alpha \in \simpleroots} \mathbb{Z}_{\geq 0} \, \alpha.
\end{equation}
Since $G$ is semisimple, the simple roots $\simpleroots$ form a basis of $\characters(T)_{\mathbb{Q}}$
and the corresponding simple coroots $\simpleroots^\vee = \{ \alpha^\vee \mid \alpha \in \simpleroots \}\subset\cocharacters(T)$
form a basis of $\cocharacters(T)_{\mathbb{Q}}$.

We define the \emph{fundamental weights} $(\omega_\alpha)_{\alpha \in \simpleroots} \in \characters(T)_{\mathbb{Q}}$ by
\begin{equation}
  \langle \omega_\alpha, \beta^\vee \rangle = \delta_{\alpha,\beta} \qquad \text{for} \quad \alpha, \beta \in \simpleroots,
\end{equation}
and the \emph{fundamental coweights} $(\omega_\alpha^\vee)_{\alpha \in \simpleroots} \in \cocharacters(T)_{\mathbb{Q}}$ by
\begin{equation}
  \langle \alpha, \omega_\beta^\vee \rangle = \delta_{\alpha,\beta} \qquad \text{for} \quad \alpha, \beta \in \simpleroots.
\end{equation}
A priori the fundamental weights $\omega_{\alpha}$ live in $\characters(T)_{\mathbb{Q}}$.
However, since we assume $G$ to be simply-connected, they live in $\characters(T)$
and form a basis of it. Consequently, the simple coroots form the dual basis of $\cocharacters(T)$.

A weight $\lambda = \sum_{\alpha \in \simpleroots} \ell_{\alpha} \omega_{\alpha} \in \characters(T)$ is
called \emph{$G$-dominant} (or simply \emph{dominant}) if
\begin{equation}
  \langle \lambda, \alpha^\vee \rangle \geq 0 \quad \text{for all} \quad \alpha \in \simpleroots
  \iff
  \ell_{\alpha} \geq 0 \quad \text{for all} \quad \alpha \in \simpleroots.
\end{equation}
Fundamental weights form a cone $\characters(T)^+_{G} \subset \characters(T)$ called the \emph{dominant cone}.

A weight $\lambda = \sum_{\alpha \in \simpleroots} \ell_{\alpha} \omega_{\alpha} \in \characters(T)^+_{G}$
is called \emph{strictly dominant} if $\ell_{\alpha} > 0$ for all $\alpha \in \simpleroots$.

For $\lambda\in\weightlattice{}$ we denote by $\lambda^\vee$ the unique coweight
defined by the identity
\begin{equation}
  \langle \lambda , \alpha^\vee \rangle = \langle \alpha , \lambda^\vee \rangle \quad \text{for all} \quad \alpha \in \simpleroots.
\end{equation}
Thus, for $\lambda = \sum_{\alpha \in \simpleroots} \ell_{\alpha} \omega_{\alpha}$ we have
$\lambda^\vee = \sum_{\alpha \in \simpleroots} \ell_{\alpha} \omega_{\alpha}^\vee$.
In particular, we have $(\omega_{\alpha})^\vee = \omega_{\alpha}^\vee$.

We define the weight
\begin{equation}
  \rho = \sum_{\alpha \in \simpleroots} \omega_{\alpha} = \frac{1}{2} \sum_{\alpha \in \roots^+} \alpha
\end{equation}
and the \emph{dot-action} of $\weyl_G$ on $\characters(T)$ by the formula
\begin{equation}
  w \cdot \lambda = w(\lambda + \rho) - \rho.
\end{equation}

Since $G$ is simple, its root system is irreducible, and therefore
there exists (up to a non-zero factor)
a unique $\weyl_G$-invariant scalar product
on $\characters(T)_{\mathbb{R}}$,
(see \cite[Ch.~VI, \S1, no.~2]{MR0240238})
\begin{equation}\label{eq:scalar-product-weights}
  \left( - , - \right) \colon \characters(T)_{\mathbb{R}} \times \characters(T)_{\mathbb{R}} \to \mathbb{R}.
\end{equation}
We choose the standard scaling as in \cite{MR0240238}.
Now using \eqref{eq:scalar-product-weights} we can identify $\characters(T)_{\mathbb{R}}$
and $\cocharacters(T)_{\mathbb{R}}$. Thus, for us both roots and coroots will live in the same
space $\characters(T)_{\mathbb{R}}$ and we have
\begin{equation}
  \alpha^\vee = \frac{2}{(\alpha, \alpha)} \alpha \qquad \text{for} \quad \alpha \in \roots.
\end{equation}

\paragraph{Parabolic subgroups}
We denote by $B^+$ and $B$ the Borel subgroups of $G$ corresponding to the positive
and negative roots respectively. We have
\begin{equation}
  B^+ \cap B = T.
\end{equation}
We want to stress that $B$ \emph{corresponds to the negative roots}.

For subset $I \subset \simpleroots$ one defines the \emph{standard parabolic subgroup} $P$
containing $B$ such that
\begin{equation}
  P = L U = L \ltimes U
\end{equation}
where $L$ is the \emph{Levi factor} of $P$
and $U$ is the unipotent radical of $P$,
and the subset of simple roots of the (reductive) group~$L$ is precisely~$I$.
The group $L$ is a reductive subgroup of $G$ containing $T$ and its roots with
respect to $T$ are
\begin{equation}
  \roots_L = \roots \cap \mathbb{Z}I.
\end{equation}
We also introduce the notation
\begin{equation}
  \simpleroots_L = \simpleroots \cap \roots_L = I \quad \text{and} \quad \roots_L^\pm = \roots_L \cap \roots^\pm.
\end{equation}
For the Weyl group of $L$ we have
\begin{equation}
  \weyl_L = \langle s_\alpha \mid \alpha \in \roots_L \rangle,
\end{equation}
and it is generated by  the simple reflections $s_{\alpha}$ with $\alpha \in I$.

\paragraph{Associated Lie algebras}
We denote Lie algebras of the aforementioned algebraic groups by
\begin{equation}
  \mfg = \Lie(G), \qquad \mft = \Lie(T), \qquad \mfb = \Lie(B),
\end{equation}
\begin{equation}
  \mfp = \Lie(P), \qquad \mfl = \Lie(L), \qquad \mfn = \Lie(U).
\end{equation}

We have the following decompositions
\begin{equation}
  \mfg = \mft \oplus \bigoplus_{\alpha \in \roots} \mfg_{\alpha} \, , \qquad
  \mfb = \mft \oplus \bigoplus_{\alpha \in \roots^-} \mfg_{\alpha} \, , \qquad
\end{equation}

\begin{equation}
  \mfp = \mfl \oplus \mfn \, , \qquad
\end{equation}

\begin{equation}
  \mfl = \mft \oplus \bigoplus_{\alpha \in \roots_L} \mfg_{\alpha} \, , \qquad
  \mfp = \mft \oplus \bigoplus_{\alpha \in \roots^- \cup \roots_L} \mfg_{\alpha} \, , \qquad
  \mfn = \bigoplus_{\alpha \in \roots^- \setminus \roots_L^-} \mfg_{\alpha} \, .
\end{equation}

\paragraph{Varieties}
In \cref{subsection:partial-flag-varieties} we will introduce the partial flag varieties of interest,
for which we use the following notation:
\begin{itemize}
  \item $Q^n$, the $n$\dash dimensional smooth quadric hypersurface in~$\mathbb{P}^{n+1}$;
  \item $\Gr(d,n)$, the Grassmannian of~$d$\dash subspaces in an~$n$\dash dimensional vector space;
  \item $\OGr(d,n)$, the \emph{orthogonal Grasmannian} of isotropic~$d$\dash dimensional subspaces in an~$n$\dash dimensional vector space equipped with a nondegenerate symmetric bilinear form, which is an isotropic Grassmannian in type~B (resp.~D) depending on the parity of~$n$;
  \item $\SGr(d,2n)$, the \emph{symplectic Grassmannian} of isotropic~$d$\dash dimensional subspaces in a~$2n$\dash dimensional vector space equipped with a nondegenerate skew-symmetric bilinear form, which is an isotropic Grassmannian in type~C.
\end{itemize}

\paragraph{Representations and equivariant vector bundles}
The representation theory of simple Lie algebras and algebraic groups
allows us to describe irreducible representations using highest weights;
and as the categories of representations are equivalent we will interchangeably use~$\mathfrak{g}$ and~$G$.
We will denote
\begin{itemize}
  \item $\representation{\lambda}{\mfg}$ (resp.~$\representation{\lambda}{G}$), the irreducible~$\mfg$\dash representation (resp.~$G$\dash representation) associated to the~$\mfg$\dash dominant highest weight~$\lambda\in\dominantweights{\mathfrak{g}}$;
  \item $\representation{\lambda}{\mfl}$ (resp.~$\representation{\lambda}{L}$), the irreducible~$\mfl$\dash representation (resp.~$L$\dash representation) associated to the~$\mfl$\dash dominant highest weight~$\lambda\in\dominantweights{\mathfrak{l}}$;
  \item $\bundle{\lambda}$, the~$G$\dash equivariant vector bundle on~$G/P$ associated to~$\representation{\lambda}{\mfl}$.
\end{itemize}

\paragraph{Notation for tables}
In some cases we will give a description of the associated graded of~$\bigwedge^p\tangent_{G/P}$ in the sense of \cref{definition:konno-filtration}, see \cref{table:wedge-A3-P2,table:wedge-A3-P13,table:wedge-B3-P2}. Each row is an irreducible summand, and the columns are to be interpreted as:
\begin{description}[leftmargin=8em,align=right,style=nextline]
  \item[weight] the weight of the (irreducible) vector bundle, as a coefficient vector for the fundamental weights
  \item[rank] the rank of the vector bundle
  \item[degree] the degree in which its cohomology lives according to the Borel--Weil--Bott theorem, or empty if the weight is not regular in the setting of Borel--Weil--Bott
  \item[representation] if the cohomology is nonzero, the highest weight of the representation obtained from the Borel--Weil--Bott theorem
  \item[dimension] the dimension of this representation, if nonzero
  \item[sum of roots] the weight of the vector bundle, as a coefficient vector for the simple roots
\end{description}
The coefficient vectors are given with respect to the fundamental weights (resp.~the simple roots)
instead of an explicit description of the weight,
using the labeling of the vertices from Bourbaki \cite{MR0240238}
and recalled in \cref{table:bourbaki-labelling}.

\subsection{Partial flag varieties}
\label{subsection:partial-flag-varieties}
Here we fix notation and terminology related to partial flag varieties $G/P$.

\paragraph{Classification of partial flag varieties}
To isolate certain well-behaved families of partial flag varieties,
we need to to talk about their explicit geometric realisation,
although the proofs will not use this.
We will use the Bourbaki convention for labelling the simple roots,
which is recalled in \cref{table:bourbaki-labelling}.

\begin{table}[ht!]
  \centering
  \begin{tabular}{cc}
    \toprule
    type & labelling \\
    \midrule
    $\mathrm{A}_n$ & \begin{tikzpicture} \dynkin[edgeLength=.6cm]{A}{} \dynkinLabelRoot*{1}{1} \dynkinLabelRoot*{2}{2} \dynkinLabelRoot*{3}{n-1} \dynkinLabelRoot*{4}{n} \end{tikzpicture} \\
    $\mathrm{B}_n$ & \begin{tikzpicture} \dynkin[edgeLength=.6cm]{B}{} \dynkinLabelRoot*{1}{1} \dynkinLabelRoot*{2}{2} \dynkinLabelRoot*{3}{n-2} \dynkinLabelRoot*{4}{n-1} \dynkinLabelRoot*{5}{n} \end{tikzpicture} \\
    $\mathrm{C}_n$ & \begin{tikzpicture} \dynkin[edgeLength=.6cm]{C}{} \dynkinLabelRoot*{1}{1} \dynkinLabelRoot*{2}{2} \dynkinLabelRoot*{3}{n-2} \dynkinLabelRoot*{4}{n-1} \dynkinLabelRoot*{5}{n} \end{tikzpicture} \\
    $\mathrm{D}_n$ & \begin{tikzpicture} \dynkin[edgeLength=.6cm]{D}{} \dynkinLabelRoot*{1}{1} \dynkinLabelRoot*{2}{2} \dynkinLabelRoot*{3}{n-3} \dynkinLabelRoot{4}{n-2} \dynkinLabelRoot{5}{n-1} \dynkinLabelRoot{6}{n} \end{tikzpicture} \\
    $\mathrm{E}_6$ & \dynkin[label,edgeLength=.6cm]{E}{6} \\
    $\mathrm{E}_7$ & \dynkin[label,edgeLength=.6cm]{E}{7} \\
    $\mathrm{E}_8$ & \dynkin[label,edgeLength=.6cm]{E}{8} \\
    $\mathrm{F}_4$ & \dynkin[label,edgeLength=.6cm]{F}{4} \\
    $\mathrm{G}_2$ & \dynkin[label,edgeLength=.6cm,reverseArrows]{G}{2} \\
    \bottomrule
  \end{tabular}
  \caption{Bourbaki labelling for simple roots}
  \label{table:bourbaki-labelling}
\end{table}


To a $G$-dominant weight $\lambda\in\dominantweights{G}$
we associate the unique closed $G$-orbit in~$\mathbb{P}((\representation{\lambda}{G})^\vee)$.
This is the orbit of the (line spanned by) the lowest weight vector $v_{-\lambda}$
of weight $-\lambda$ of the representation $(\representation{\lambda}{G})^\vee$
and its stabiliser is the standard parabolic subgroup $P$ associated to the subset $I \subset \simpleroots$ defined by
\begin{equation}
  I\coloneqq\{\alpha\in\simpleroots\mid(\alpha,\lambda)=0\} \subset \simpleroots.
\end{equation}
This gives an explicit realisation of~$G/P$.

We will specify a partial flag variety by describing the (sum of) simple roots which are \emph{not} included in the parabolic subgroup,
e.g.~$(\mathrm{A}_n,\alpha_1)$ corresponds to~$\mathbb{P}^{n+1}$.
For a maximal parabolic subgroup there thus is a single simple root.
In general we can describe~$P$ by crossing out these simple roots in the Dynkin diagram.
Hence~$(\mathrm{A}_n,\alpha_1)$ is described by
\begin{equation}
  \dynkin[parabolic=1]{A}{}.
\end{equation}
When~$P$ is a \emph{maximal} parabolic subgroup
we will say that the partial flag variety~$G/P$ is a \emph{generalised Grassmannian}.

The two following remarks explain why we can \emph{and will} ignore certain descriptions of generalised Grassmannians.
\begin{remark}
  \label{remark:exceptional-isomorphisms}
  Exceptional isomorphisms of Lie algebras in low rank and symmetries of the Dynkin diagrams account for the following isomorphisms of generalised Grassmannians:
  \begin{itemize}
    \item $(\mathrm{A}_n,\alpha_i)=(\mathrm{A}_n,\alpha_{n+1-i})$, as~$\Gr(i,n+1)\cong\Gr(n-i,n+1)$;
    \item $(\mathrm{A}_3,\alpha_2)=(\mathrm{D}_3,\alpha_1)$, which are isomorphic to~$Q^4$;
    \item $(\mathrm{D}_n,\alpha_{n-1})=(\mathrm{D}_n,\alpha_n)$, the \emph{$n(n-1)/2$\dash dimensional spinor variety},
      which is one of the connected components of the space of maximal isotropic subspaces for a nondegenerate symmetric bilinear form in a~$2n$\dash dimensional vector space;
    \item $(\mathrm{D}_4,\alpha_1)=(\mathrm{D}_4,\alpha_3)=(\mathrm{D}_4,\alpha_4)$, which are isomorphic to~$Q^6$;
    \item $(\mathrm{E}_6,\alpha_1)=(\mathrm{E}_6,\alpha_6)$, the \emph{Cayley plane};
    \item $(\mathrm{E}_6,\alpha_3)=(\mathrm{E}_6,\alpha_5)$.
  \end{itemize}
\end{remark}

\begin{remark}
  \label{remark:exceptional-isomorphisms-exotic}
  On the other hand one also has the following exotic isomorphisms which are not related to an exceptional isomorphism of the associated simple Lie algebras or an obvious symmetry of the Dynkin diagram:
  \begin{enumerate}
    \item $(\mathrm{B}_{n-1},\alpha_{n-1})=(\mathrm{D}_n,\alpha_n)$, giving an alternative description of the spinor varieties;
    \item $(\mathrm{C}_n,\alpha_1)=(\mathrm{A}_{2n-1},\alpha_1)$, isomorphic to~$\mathbb{P}^{2n-1}$;
    \item $(\mathrm{G}_2,\alpha_1)=(\mathrm{B}_3,\alpha_1)$, isomorphic to~$Q^5$.
  \end{enumerate}
\end{remark}

This second class of exotic isomorphisms explains the caveat in the following lemma \cite[\S2]{MR0435092}.
\begin{lemma}
  \label{lemma:exceptional-tangent-bundle}
  Let~$G/P$ be a partial flag variety. Then
  \begin{equation}
    \HH^0(G/P,\tangent_{G/P})\cong\mathfrak{g},
  \end{equation}
  unless~$P$ is a maximal parabolic subgroup and~$G/P$ is of type~$(\mathrm{B}_n,\alpha_n)$, $(\mathrm{C}_n,\alpha_1)$ or~$(\mathrm{G}_2,\alpha_1)$.
\end{lemma}
In the cases which are ruled out in the statement of the lemma, we notice that we obtain a Lie subalgebra of the Lie algebra associated to the automorphism group of~$G/P$. We will exclude these cases without further mention from our analysis.

\paragraph{Cominuscule and (co)adjoint partial flag varieties}
We will now introduce the terminology used to distinguish several special classes
of partial flag varieties. These go by different names in the literature, but we
will be using the terminology from \cite{MR2821244} and mention other terminology
as we go along. Here we will use the explicit geometric realisation of~$G/P$ as
the unique closed orbit
in~$\mathbb{P}((\representation{\lambda}{G})^\vee)$.

\begin{definition}
  \label{definition:cominuscule-adjoint-weights}
  Let~$\lambda\in\dominantweights{G}$ be a dominant weight of~$G$. We will say that~$\lambda$ is
  \begin{enumerate}
    \item \emph{minuscule} if 
      \begin{equation}\label{equation:minuscule-definition}
        \left( \lambda,\alpha^\vee \right) \leq 1 \quad \forall \alpha \in\roots^+
      \end{equation}
    \item \emph{cominuscule} if 
      \begin{equation}\label{equation:cominuscule-definition}
        \left( \alpha,\lambda^\vee \right) \leq 1 \quad \forall \alpha \in\roots^+
      \end{equation}
    \item \emph{adjoint} if~$\lambda$ is the highest weight of the adjoint representation\footnote{Since~$G$ is assumed to be simple, its adjoint representation is irreducible.} of~$G$, i.e.,~$\lambda=\Theta$ is the highest (long) root of~$G$;
    \item \emph{coadjoint} if~$\lambda$ is the highest short root $\theta$.
  \end{enumerate}
\end{definition}

Note that if $G$ is simply-laced, then the notions of minuscule and cominuscule
(resp.~adjoint and coadjoint) coincide.


\begin{definition}
  Let~$P$ be the standard parabolic subgroup of~$G$ associated to a weight~$\lambda$.
  Then we say that the partial flag variety~$G/P$ is \emph{minuscule} (resp.~\emph{cominuscule}, \emph{adjoint}, \emph{coadjoint}) if~$\lambda$ is.
  In such a case we also call the parabolic $P$ \emph{minuscule} (resp.~\emph{cominuscule}, \emph{adjoint}, \emph{coadjoint}).
\end{definition}


\begin{remark}
  \label{remark:no-minuscule}
  For the purposes of our analysis we can ignore the minuscule case. The only generalised Grassmannians which are minuscule but not cominuscule are associated to~$(\mathrm{B}_n,\alpha_n)$ and~$(\mathrm{C}_n,\alpha_1)$ respectively. But by \cref{remark:exceptional-isomorphisms-exotic} these are isomorphic to the generalised Grassmannians associated to~$(\mathrm{D}_{n+1},\alpha_{n+1})$ and~$(\mathrm{A}_{2n},\alpha_1)$ respectively, which are cominuscule as can be seen in \cref{table:cominuscule}, and we will use the latter realisations for our analysis.
\end{remark}

In \cref{table:cominuscule} we have collected the cominuscule generalised Grassmannians, and their relevant properties.

\begin{table}[ht!]
  \begin{adjustbox}{center}
    \begin{tabular}{cccccc}
      \toprule
      type                                                    & variety                                & diagram                                                & dimension  & index    & Cartan label \\
      \midrule
      $(\mathrm{A}_n,\simple{1})$                             & $\mathbb{P}^n$                         & \dynkin[parabolic=1]{A}{}                              & $n$        & $n+1$    & AIII \\
      $(\mathrm{A}_n,\simple{2})$                             & $\Gr(2,n+1)$                           & \dynkin[parabolic=2]{A}{}                              & $2(n-1)$   & $n+1$    & AIII \\
      $(\mathrm{A}_n,\simple{i})$                             & $\vdots$                               & $\vdots$                                               & $\vdots$   & $\vdots$ \\
      $(\mathrm{A}_n,\simple{n-1})$                           & $\Gr(n-2,n+1)\cong\Gr(2,n+1)$          & \dynkin[parabolic=4]{A}{}                              & $2(n-1)$   & $n+1$    & AIII \\
      $(\mathrm{A}_n,\simple{n})$                             & $\mathbb{P}^{n,\vee}\cong\mathbb{P}^n$ & \dynkin[parabolic=8]{A}{}                              & $n$        & $n+1$    & AIII \\
      $(\mathrm{B}_n,\simple{1})$                             & $Q^{2n-1}$                             & \dynkin[parabolic=1]{B}{}                              & $2n-1$     & $2n-1$   & BDI \\
      $(\mathrm{C}_n,\simple{n})$                             & $\SGr(n,2n)=\LGr(2n)$                  & \dynkin[parabolic=16]{C}{}                             & $n(n+1)/2$ & $n+1$    & CI \\
      $(\mathrm{D}_n,\simple{1})$                             & $Q^{2n-2}$                             & \dynkin[parabolic=1]{D}{}                              & $2n-2$     & $2n-2$   & BDI \\
      $(\mathrm{D}_n,\simple{n-1})=(\mathrm{D}_n,\simple{n})$ & $\OGr(n-1,2n)$                         & \dynkin[parabolic=16]{D}{} \dynkin[parabolic=32]{D}{}  & $n(n-1)/2$ & $2n-2$   & DIII \\
      $(\mathrm{E}_6,\simple{1})=(\mathrm{E}_6,\simple{6})$   & Cayley plane                           & \dynkin[parabolic=1]{E}{6} \dynkin[parabolic=32]{E}{6} & 16         & 12       & EIII \\
      $(\mathrm{E}_7,\simple{7})$                             & Freudenthal variety                    & \dynkin[parabolic=64]{E}{7}                            & 27         & 17       & EVII \\
      \bottomrule
    \end{tabular}
  \end{adjustbox}
  \caption{Cominuscule partial flag varieties}
  \label{table:cominuscule}
\end{table}

\begin{remark}
  Over the complex numbers cominuscule generalised Grassmannians are also known as \emph{compact hermitian symmetric spaces}, and they are often referred to as such in the literature. We have included the Cartan labelling for them in \cref{table:cominuscule}.
\end{remark}

For each Dynkin type and rank there is a unique \emph{adjoint} partial flag variety. In \cref{table:adjoint} we have collected the adjoint partial flag varieties (excluding type~C, see below), and their relevant properties.

Two special cases for us are
\begin{itemize}
  \item in type~A, where it is not a generalised Grassmannian, as the associated parabolic subgroup is submaximal such that~$\rk\Pic G/P=2$: in this case it is isomorphic to~$\mathbb{P}(\tangent_{\mathbb{P}^n})$, the relative Proj of~$\Sym^\bullet\tangent_{\mathbb{P}^n}^\vee$;
  \item in type~C, where the highest root is~$2\fundamental{1}$, but by \cref{remark:exceptional-isomorphisms-exotic} these generalised Grassmannians are isomorphic to~$\mathbb{P}^{2n-1}$ (and the adjoint realisation is the second Veronese embedding) and will be omitted from the analysis.
\end{itemize}

\begin{table}[ht!]
  \centering
  \begin{tabular}{ccccc}
    \toprule
    type                                 & variety                               & diagram                                  & dimension & index \\
    \midrule
    $(\mathrm{A}_n,\simple{1}+\simple{n})$ & $\mathbb{P}(\tangent_{\mathbb{P}^n})$ & \dynkin[parabolic=9]{A}{}                & $2n-1$    & $n$ \\
    $(\mathrm{B}_n,\simple{2})$           & $\OGr(2,2n+1)$                        & \dynkin[parabolic=2]{B}{}                & $4n-5$    & $2n-2$ \\
    $(\mathrm{D}_n,\simple{2})$           & $\OGr(2,2n)$                          & \dynkin[parabolic=2]{D}{}                & $4n-7$    & $2n-3$ \\
    $(\mathrm{E}_6,\simple{2})$           &                                       & \dynkin[parabolic=2]{E}{6}               & $21$      & $11$ \\
    $(\mathrm{E}_7,\simple{1})$           &                                       & \dynkin[parabolic=1]{E}{7}               & $33$      & $17$ \\
    $(\mathrm{E}_8,\simple{8})$           &                                       & \dynkin[parabolic=128]{E}{8}             & $57$      & $29$ \\
    $(\mathrm{F}_4,\simple{1})$           &                                       & \dynkin[parabolic=1]{F}{4}               & $15$      & $8$ \\
    $(\mathrm{G}_2,\simple{2})$           & $\operatorname{G_2Gr}(2,7)$           & \dynkin[parabolic=2,reverseArrows]{G}{2} & $5$       & $3$ \\
    \bottomrule
  \end{tabular}
  \caption{Adjoint partial flag varieties}
  \label{table:adjoint}
\end{table}

Finally, similar to \cref{remark:no-minuscule} we need to consider the non-simply-laced case and classify the coadjoint but not adjoint partial flag varieties. These cannot be omitted from the analysis. In \cref{table:coadjoint} we have collected the remaining coadjoint partial flag varieties, and their relevant properties.

\begin{table}[ht!]
  \centering
  \begin{tabular}{cccccc}
    \toprule
    type           & variety      & diagram                    & dimension & index \\
    \midrule
    $(\mathrm{C}_n,\simple{2})$ & $\SGr(2,2n)$ & \dynkin[parabolic=2]{C}{}  & $4n-5$    & $n+1$ \\
    $(\mathrm{F}_4,\simple{4})$ &              & \dynkin[parabolic=8]{F}{4} & $15$      & $11$ \\
    \bottomrule
  \end{tabular}
  \caption{Coadjoint but not adjoint partial flag varieties}
  \label{table:coadjoint}
\end{table}

For more on the geometry of generalised Grassmannians one is referred to \cite{grassmannian-info}.

\paragraph{Equivariant vector bundles and Borel--Weil--Bott}
For a partial flag variety there exists an equivalence
\begin{equation}
  \label{equation:equivalence-of-categories}
  \coh^G G/P\cong\rep P
\end{equation}
of monoidal abelian categories between the category of $G$-equivariant vector bundles
on~$G/P$ and the category of finite-dimensional representations of~$P$ \cite{MR1074782,MR1621314};
under this equivalence a $G$-equivariant vector bundle $E$ is sent to its fiber
$E_{[P]}$ at the point $[P] \in G/P$.
As $P$ is not reductive, the category~$\rep P$ is not semisimple and its representation
theory is hard to understand. An interesting full subcategory of~$\rep P$ is given
by $\rep^\semisimple P$, with objects the completely reducible representations,
which is a semisimple category.
We have
\begin{equation}
  \label{equation:cohG-GP}
  \rep^\semisimple P\cong\rep L\subseteq\rep P\cong\coh^G G/P,
\end{equation}
where $L$ is the Levi factor~$L \subset P$ (see \cref{subsection:new-setup-and-notation}).

Given an $L$-dominant weight~$\lambda\in \weightlattice{L}^+$ we get an irreducible
$L$-representation~$\representation{\lambda}{L}$ with the highest weight $\lambda$,
which we can extend to a representation of~$P$ by letting the unipotent radical $U$
act trivially, and hence a $G$-equivariant vector bundle~$\bundle{\lambda}$ on~$G/P$.

Unfortunately, the inclusion in \eqref{equation:cohG-GP} is strict, and not all equivariant vector bundles we are interested in arise as representations of~$L$, see \cref{subsection:tangent-bundle}. But for those which are associated to completely reducible representations, there is a strong tool to compute their sheaf cohomology: the Borel--Weil--Bott theorem.

Recall that a weight $\mu \in \weightlattice{G}$ is called \emph{$G$-regular} (or \emph{regular}),
if it does not lie on a wall of a Weyl chamber of $G$. Equivalently, a weight $\mu$
is regular if and only if
\begin{equation}\label{equation:regular-weight}
 \left( \mu , \alpha \right) \neq 0 \qquad \forall \alpha \in \roots.
\end{equation}
Otherwise the weight is called \emph{$G$-singular} (or \emph{singular}).

\begin{theorem}[Borel--Weil--Bott]
  \label{theorem:bwb}
  Let~$\bundle{\lambda}$ be the~$G$\dash equivariant vector bundle on~$G/P$ given by the irreducible~$L$\dash representation with highest weight~$\lambda\in\dominantweights{L}$. Then one of the following holds:
  \begin{enumerate}
    \item if~$\lambda + \rho$ is~$G$\dash singular, then
      \begin{equation}
        \HH^i(G/P, \bundle{\lambda}) = 0
      \end{equation}
      for all~$i$;
    \item if~$\lambda + \rho$ is~$G$\dash regular, then there exists a unique~$w\in\weyl_G$ such that~$w(\lambda+\rho)$ is $G$-dominant, and then
      \begin{equation}
        \HH^i(G/P, \bundle{\lambda})
        \cong
        \begin{cases}
          \representation{w(\lambda + \rho)-\rho}{G} & i = \ell(w) \\
          0 & i\neq\ell(w),
        \end{cases}
      \end{equation}
      as~$G$\dash representations, where~$\ell(w)$ denotes the length of the element~$w\in\weyl_G$.
  \end{enumerate}
\end{theorem}

In most cases we cannot apply the Borel--Weil--Bott theorem on the nose to compute the sheaf cohomology of~$\bigwedge^p\tangent_{G/P}$.
These are the vector bundles we will be interested in when studying Hochschild cohomology,
but in general these equivariant vector bundles are not completely reducible.
This fact is the main reason for the existence of this paper.

\begin{remark}
  When describing Hochschild cohomology it is more natural to emphasise the structure as a representation of the Lie algebra, i.e.,~we will rather write~$\representation{w(\lambda + \rho)-\rho}{\mathfrak{g}}$.
\end{remark}

\paragraph{On dominant weights}
The following lemma summarises some standard properties.

\begin{lemma}\label{lemma:new-lemma-on-weights}
  Let~$G$ and~$P$ be as before.
  \begin{enumerate}
    \item\label{enumerate:weights-1}
      Let $\lambda$ be an $L$-dominant weight.
      If we have $\left( \lambda , \alpha \right) \geq 0$ for $\alpha \in \simpleroots \setminus \simpleroots_L$,
      then $\left( \lambda , \alpha \right) \geq 0$ for any $\alpha \in \roots^+$.
    \item\label{enumerate:weights-2}
      Let $\lambda$ be an $L$-dominant weight. Then
      \begin{enumerate}
        \item $\left( \lambda + \rho , \alpha \right) > 0$ for any $\alpha \in \roots^+_L$;
        \item If $\left( \lambda + \rho , \alpha \right) \neq 0$ for any $\alpha \in \roots^+ \setminus \roots^+_L$,
        then the weight $\lambda + \rho$ is regular.
      \end{enumerate}
    \item\label{enumerate:weights-3}
      Let $\lambda$ be a strictly $G$-dominant weight and $\alpha \in \roots$ a root.
      Then we have
      \begin{enumerate}
        \item $\alpha$ is positive $\iff$ $\left( \lambda , \alpha \right) > 0$;
        \item $\alpha$ is negative $\iff$ $\left( \lambda , \alpha \right) < 0$.
      \end{enumerate}
    \item\label{enumerate:weights-4}
      For $w \in \weyl_G$ we have
      \begin{equation}
        \ell(w) = \#\roots(w),
      \end{equation}
      where
      \begin{equation}
        \roots(w) = \big\{ \alpha \in \roots^+ \mid w(\alpha) \in \roots^- \big\}.
      \end{equation}
    \item\label{enumerate:weights-5}
      Let $\lambda$ be an $L$-dominant weight, such that $\lambda + \rho$ is regular,
      and let $w \in \weyl_G$ be the unique Weyl group element, such that $w(\lambda + \rho)$
      is strictly $G$-dominant.
      Then we have
      \begin{equation}
        \ell(w) = \# \big\{ \alpha \in \roots^+ \setminus \roots^+_L \mid \left( \lambda + \rho , \alpha \right) < 0 \big\}.
      \end{equation}
  \end{enumerate}
\end{lemma}

\begin{proof}
  \cref{enumerate:weights-1}, \cref{enumerate:weights-2} and~\cref{enumerate:weights-3}
  follow immediately from the definitions given \cref{subsection:new-setup-and-notation}.
  For \cref{enumerate:weights-4} we refer to \cite[Lemma 10.3A]{MR0499562}.

  Let us prove \cref{enumerate:weights-5}.
  Let $\mu$ be the unique $G$-dominant weight
  such that $w(\lambda + \rho) = \mu + \rho$. By \cref{enumerate:weights-3} and \cref{enumerate:weights-4} and $\weyl_G$-invariance of the scalar product
  we have
  \begin{equation}
    \roots(w) = \big\{ \alpha \in \roots^+ \mid \left( \mu + \rho , w(\alpha) \right) < 0 \big\} =
    \big\{ \alpha \in \roots^+ \mid \left( \lambda + \rho , \alpha \right) < 0 \big\}.
  \end{equation}
  Applying \cref{enumerate:weights-2} we get
  \begin{equation}
    \roots(w) = \big\{ \alpha \in \roots^+ \setminus \roots^+_L \mid \left( \lambda + \rho , \alpha^\vee \right) < 0 \big\}.
  \end{equation}
  Now the desired equality follows from \cref{enumerate:weights-4}.
\end{proof}

\subsection{Lie algebra cohomology}
\label{subsection:lie-algebra-cohomology}
In the description of the Hochschild--Kostant--Rosenberg decomposition for cominuscule and adjoint varieties we will need some results on Lie algebra cohomology. Let us briefly introduce the required notions.

\begin{definition}
  Let~$\mathfrak{g}$ be a Lie algebra, and~$V$ a representation of~$\mathfrak{g}$. Then the \emph{$i$th Lie algebra cohomology of~$\mathfrak{g}$ with values in~$V$} is
  \begin{equation}
    \HH_\CE^i(\mathfrak{g},V)
    \coloneqq
    \Ext_{\mathrm{U}\mathfrak{g}}^i(\groundfield,V)
  \end{equation}
  where~$\mathrm{U}\mathfrak{g}$ is the universal enveloping algebra of~$\mathfrak{g}$.
\end{definition}
Although confusion between sheaf cohomology and Lie algebra cohomology is unlikely, we will always denote Lie algebra cohomology as~$\HH_\CE^\bullet(\mathfrak{g},V)$.

To compute it one often uses an explicit projective resolution of the trivial~$\mathrm{U}\mathfrak{g}$\dash module~$k$ which gives rise to the \emph{Chevalley--Eilenberg complex}~$\Hom_\groundfield(\bigwedge^\bullet\mathfrak{g},V)$, with differential
\begin{equation}
  \begin{gathered}
    \mathrm{d}(f)(g_1\wedge\ldots\wedge g_{n+1})\coloneqq
    \sum_i(-1)^{i+1}g_if(g_1\wedge\ldots\wedge\hat{g_i}\wedge\ldots\wedge g_{n+1}) \\
    +
    \sum_{i<j}(-1)^{i+j}f([g_i,g_j]\wedge g_1\wedge\ldots\wedge\hat{g}_i\wedge\ldots\wedge\hat{g_j}\wedge\ldots\wedge g_{n+1})
  \end{gathered}
\end{equation}
for~$f\in\Hom_\groundfield(\bigwedge^n\mathfrak{g},V)$.

The straightforward observation to make from this definition is that if~$\mathfrak{g}$ is abelian and the action of~$\mathfrak{g}$ on~$V$ is trivial, then the differentials in this complex vanish, and~$\HH_\CE^i(\mathfrak{g},V)\cong\bigwedge^i\mathfrak{g}\otimes_\groundfield V$.

We will compute Lie algebra cohomology for the nilpotent radical of parabolic subalgebras in (semi)simple Lie algebras, and Lie algebras constructed out of it. For this we will use a result of Kostant \cite[Corollary~8.2]{MR0142696}. The setting we work in is that of a simple Lie algebra~$\mathfrak{g}$, with parabolic subalgebra~$\mathfrak{p}$ and unipotent radical~$\mathfrak{n}$,
\begin{equation}
  \mathfrak{n}\subseteq\mathfrak{p}\subseteq\mathfrak{g}.
\end{equation}
We have the Levi decomposition~$\mathfrak{p}=\mathfrak{l}\oplus\mathfrak{n}$, where~$\mathfrak{l}$ is the Levi subalgebra. If~$V$ is a~$\mathfrak{p}$\dash representation, then~$\HH_\CE^\bullet(\mathfrak{n},V)$ has the structure of a~$\mathfrak{l}$\dash representation.

We need one more piece of notation.
If~$I$ denotes the set of simple roots added to the Borel subalgebra to obtain~$\mathfrak{p}$,
then we denote
\begin{equation}
  \cosets
  \coloneqq
  \left\{ w\in\weyl_{\mathfrak{g}}\mid \forall\alpha\in I\colon\ell(s_\alpha w)=\ell(w)+1 \right\}
  =
  \left\{ w\in\weyl_{\mathfrak{g}}\mid w^{-1}(\dominantweights{\mathfrak{l}})\subseteq\dominantweights{\mathfrak{g}} \right\}
\end{equation}
the set of minimal length (right) coset representatives of the Weyl group~$\weyl_{\mathfrak{l}}$ in~$\weyl_{\mathfrak{g}}$.

The statement of Kostant's theorem, computing the~$\mathfrak{l}$\dash structure of the Lie algebra cohomology of~$\mathfrak{n}$ with values in~$V=\groundfield$ the trivial representation -- there exists a more general version with coefficients, but we will not need it -- reads as follows.

\begin{theorem}[Kostant]
  \label{theorem:kostant}
  There exists an isomorphism
  \begin{equation}
    \HH_\CE^i(\mathfrak{n},\groundfield)
    \cong
    \bigoplus_{\mathclap{\substack{w\in\cosets \\ \ell(w)=i}}}\representation{w\cdot 0}{\mathfrak{l}}
  \end{equation}
  of~$\mathfrak{l}$\dash modules.
\end{theorem}
This result is particularly interesting when~$\mathfrak{n}$ has an easy structure. In this paper we consider the cases where~$\mathfrak{n}$ is abelian (see \cref{lemma:tangent-cominuscule-new}~\cref{enumerate:cominuscule-abelian}) or Heisenberg (see \cref{lemma:adjoint-heisenberg-nilradical}). This allows us to obtain the descriptions in \cref{subsection:hochschild-cohomology-cominuscule,subsection:hochschild-cohomology-adjoint}.

\begin{example}
  We can conveniently visualise the result of Kostant's theorem, by giving the \emph{parabolic Bruhat graph} of~$\cosets$. An introductory reference (where it is called the \emph{Hasse diagram}) is \cite[\S3.2]{MR2532439}. We will only need that~$\cosets$ can be interpreted as a subset of~$\weyl_{\mathfrak{g}}$, which has the Bruhat order, which we will write from left to right. We then restrict this Bruhat order to~$\cosets$ to obtain the parabolic Bruhat graph, and we label an edge by the Weyl group element which sends the source to the target, and this will be a simple reflection.

  For example, let us consider the parabolic subalgebra of~$\mathfrak{sl}_4$ (of type~$\mathrm{A}_3$) corresponding to
  \begin{equation}
    \dynkin[parabolic=2]{A}{3}.
  \end{equation}
  Computing~$\cosets$ in this case gives rise to the parabolic Bruhat graph in \cref{figure:parabolic-bruhat-A3-P2}. We will revisit this example in \cref{example:A3-P2}.

  \begin{figure}[ht!]
    \centering
    \begin{tikzpicture}
      \foreach \x/\y [count=\k] in {0/0, 1/0, 2/0.5, 2/-0.5, 3/0, 4/0} {
        \coordinate (v\k) at (\x,\y);
      }

      \foreach \s/\t/\a in
      {
        1/2/2,
        2/3/1, 2/4/3,
        3/5/3, 4/5/1,
        5/6/2%
      } {
        \draw (v\s) -- (v\t) node [midway, fill=white] {\footnotesize \a};
      }

      \foreach \k in {1,...,6} {
        \draw[fill=black] (v\k) circle (1.5pt);
      }
    \end{tikzpicture}
    \caption{Parabolic Bruhat graph for the (cominuscule) parabolic $\dynkin[parabolic=2]{A}{3}$}
    \label{figure:parabolic-bruhat-A3-P2}
  \end{figure}

  We can also consider the parabolic subalgebra of~$\mathfrak{sl}_4$ given by
  \begin{equation}
    \dynkin[parabolic=5]{A}{3}.
  \end{equation}
  This is associated to an \emph{adjoint} parabolic subalgebra. The parabolic Bruhat graph is given in \cref{figure:parabolic-bruhat-A3-P13}. We will revisit this example in \cref{example:A3-P13}.

  \begin{figure}[ht!]
    \centering
    \begin{tikzpicture}
      \foreach \x/\y [count=\k] in {0/0, 1/0.5, 1/-0.5, 2/1, 2/0, 2/-1, 3/1, 3/0, 3/-1, 4/0.5, 4/-0.5, 5/0} {
        \coordinate (v\k) at (\x,\y);
      }

      \foreach \s/\t/\a in
      {
        1/2/1, 1/3/3,
        2/4/2, 2/5/3, 3/5/1, 3/6/2,
        7/10/2, 8/10/1, 8/11/3, 9/11/2,
        10/12/3, 11/12/1%
      } {
        \draw (v\s) -- (v\t) node [midway, fill=white] {\footnotesize \a};
      }

      \foreach \s/\t in {4/8, 4/9, 5/7, 5/8, 5/9, 6/7, 6/8} {
        \draw (v\s) -- (v\t);
      }

      \foreach \k in {1,...,12} {
        \draw[fill=black] (v\k) circle (1.5pt);
      }
    \end{tikzpicture}
    \caption{Parabolic Bruhat graph for the (adjoint) parabolic $\dynkin[parabolic=5]{A}{3}$}
    \label{figure:parabolic-bruhat-A3-P13}
  \end{figure}
\end{example}

\subsection{Hochschild cohomology}
\label{subsection:hochschild-cohomology}
We now introduce the invariant we are trying to compute in this paper. Originally Hochschild cohomology was introduced for associative algebras, where it governs the deformation theory (as an associative algebra). Later its definition has been generalised to algebraic geometry, and more generally abelian and suitably enhanced triangulated categories.

\begin{definition}
  Let~$X$ be a smooth projective variety. Then the \emph{$i$th Hochschild cohomology} of~$X$ is
  \begin{equation}
    \HHHH^i(X)
    \coloneqq
    \Ext_{X\times X}^i(\Delta_*\mathcal{O}_X,\Delta_*\mathcal{O}_X)
  \end{equation}
  where~$\Delta\colon X\hookrightarrow X\times X$ is the diagonal embedding.
\end{definition}

To compute Hochschild cohomology one needs a convenient resolution of~$\Delta_*\mathcal{O}_X$. It turns out that~$\LLL\Delta^*\circ\Delta_*(\mathcal{O}_X)$ is quasi-isomorphic to~$\bigoplus_{i=0}^{\dim X}\Omega_X^i[-i]$. In the affine setting this result is due to Hochschild--Kostant--Rosenberg \cite{MR0142598}, and various generalisations to the (quasi)projective setting (without an attempt to be exhaustive) are due to e.g.~Gerstenhaber--Schack \cite{MR0981619}, Swan \cite{MR1390671}, Markarian \cite{MR2472137} and Yekutieli \cite{MR1940241}.
\begin{theorem}[Hochschild--Kostant--Rosenberg]
  \label{theorem:HKR}
  Let~$X$ be a smooth projective variety. Then
  \begin{equation}
    \HHHH^i(X)\cong\bigoplus_{p+q=i}\HH^q(X,\bigwedge^p\tangent_X),
  \end{equation}
  as vector spaces.
\end{theorem}
In the study of partial flag varieties we wish to reduce to the case of~$G$ a simple, and not just semisimple, algebraic group. This is done using the following two lemmas. By the K\"unneth formula and the isomorphism~$\tangent_{X\times Y}\cong\pi_X^*\tangent_X\oplus\pi_Y^*\tangent_Y$ we obtain the first lemma.
\begin{lemma}[K\"unneth formula for Hochschild cohomology]
  \label{lemma:kunneth}
  Let~$X$ and~$Y$ be smooth projective varieties. Then
  \begin{equation}
    \begin{aligned}
      \HHHH^i(X\times Y)
      &\cong\bigoplus_{p+q=i}\HH^q(X\times Y,\bigwedge^p\tangent_{X\times Y}) \\
      &\cong\bigoplus_{p+q=i}\bigoplus_{n+m=p}\HH^q(X\times Y,\bigwedge^n\tangent_X\boxtimes\bigwedge^m\tangent_Y) \\
      &\cong\bigoplus_{p+q=i}\bigoplus_{n+m=p}\bigoplus_{a+b=q}\HH^a(X,\bigwedge^n\tangent_X)\otimes_\groundfield\HH^b(Y,\bigwedge^m\tangent_Y) \\
    \end{aligned}
  \end{equation}
\end{lemma}

If a partial flag variety~$G/P$ is associated to a semisimple but not simple algebraic group~$G$, then we always have an isomorphism
\begin{equation}
  G/P\cong G_1/P_1\times\ldots\times G_r/P_r
\end{equation}
where~$G_i$ is simple and~$P_i$ is a parabolic subgroup of~$G_i$. Hence in order to answer questions about the (non-)vanishing of components in the Hochschild--Kostant--Rosenberg decomposition, by \cref{lemma:kunneth} it suffices to consider only \emph{simple} algebraic groups~$G$.

\paragraph{Algebraic structure}
Hochschild cohomology comes in complete generality equipped with the extra structure of a Gerstenhaber algebra. This is a graded-commutative algebra together with a Lie bracket of degree~$-1$ which are compatible in a way which is not relevant for this paper.

Because the Lie bracket has degree~$-1$, we get that every~$\HHHH^1(X)$ has the structure of a Lie algebra, and every~$\HHHH^i(X)$ is a representation for this Lie algebra.

On the level of polyvector fields we have that~$\HH^1(X,\mathcal{O}_X)\oplus\HH^0(X,\tangent_X)$ also has the structure of a Lie algebra: $\HH^1(X,\mathcal{O}_X)$ is the tangent space at the Picard variety and is an abelian Lie algebra, whilst~$\HH^0(X,\tangent_X)$ is the Lie algebra of the automorphism group of~$X$. In the case of~$X=G/P$ there is only the contribution of the automorphism group, and by \cref{lemma:exceptional-tangent-bundle} we have that~$\HH^0(G/P,\tangent_{G/P})\cong\mathfrak{g}$ in all relevant cases.

The Hochschild--Kostant--Rosenberg decomposition from \cref{theorem:HKR} is only on the level of vector spaces. But by twisting the isomorphism by the square root of the Todd class~$\sqrt{\td X}\in\bigoplus_{i=0}^{\dim X}\HH^i(X,\Omega_X^i)$ it is possible to upgrade it to an isomorphism of Gerstenhaber algebras \cite{MR2646112,MR2950766}.
\begin{remark}
  If~$X$ is Hochschild global in the sense of \eqref{equation:hochschild-global}, then the twist by the square root of the Todd class is necessarily trivial, as in each component it is the cup product
  \begin{equation}
    \HH^i(X,\Omega_X^i)\otimes_\groundfield\HH^q(X,\bigwedge^p\tangent_X)\to\HH^{q+i}(X,\bigwedge^{p-i}\tangent_X),
  \end{equation}
  with~$(\sqrt{\td_X})_i$, which for~$i=0$ is the identity. In this case the vector space isomorphism is in fact a Gerstenhaber algebra isomorphism, and this paper provides new instances where this is the case.
\end{remark}

We obtain the following
\begin{lemma}
  Let~$G/P$ be a partial flag variety. Then each~$\HHHH^i(X)$ has the structure of a~$\mathfrak{g}$\dash representation.
\end{lemma}
In the setting of \cref{remark:exceptional-isomorphisms-exotic} we opt to use the largest possible Lie algebra, to get the most economical description.

\begin{remark}
  \label{remark:gerstenhaber}
  What we describe in this paper is only a portion of the full Gerstenhaber structure on~$\HHHH^\bullet(G/P)$, namely the part involving the Gerstenhaber bracket of classes in degrees~1 and~$i$. Other interesting questions, aside from the (non-)vanishing of the components, are e.g.
  \begin{enumerate}
    \item whether~$\HHHH^\bullet(G/P)$ is generated as an algebra by~$\HHHH^1(G/P)$ (the answer can be checked to be yes for~$\mathbb{P}^n$);
    \item what the other Gerstenhaber brackets are, in particular those involving~$\HHHH^2(G/P)$, to get a classification of the Poisson structures on~$G/P$.
  \end{enumerate}
  As already mentioned in the introduction, the classification of Poisson structures is seemingly a very hard problem, and the answer (in case it is non-trivial) is only known for the generalised Grassmannians~$\mathbb{P}^3$ and~$Q^3$ \cite{MR3366864,MR3066408}.
 Observe that on~$G/P$ there is always the (non-zero) standard Poisson structure, see e.g.~\cite{MR2529913}.
\end{remark}

\section{Vanishing for cominuscule and (co)adjoint varieties}
\label{section:vanishing}
In this section we prove \cref{theorem:vanishing}.
It is possible to deduce this result for generalised Grassmannians
from the vanishing result \cite[Theorem~4.2.3(i)]{MR1016881} due to Konno.
In the notation of loc.~cit.~we need that the difference between the index~$k(G/P)$ (denoted~$\fanoindex_{G/P}$ in this work)
and~$k'(G/P)$ as defined in equation (4.2.1) of op.~cit.~is at most~1.
One can check that this is the case if and only if~$G/P$ is of (co)minuscule or (co)adjoint type,
following the description in \cref{subsection:partial-flag-varieties}.

We give an alternative proof, relying more on the geometry of the varieties involved
and which also covers the adjoint variety in type~A.
Along the way we also set up the machinery used in the proofs of \cref{theorem:cominuscule-decomposition,theorem:adjoint-decomposition}.

\subsection{Tangent bundle of a partial flag variety}
\label{subsection:tangent-bundle}
We will discuss some preliminary facts on the tangent bundle~$\tangent_{G/P}$ of~$G/P$ and its exterior powers.

%
%

\begin{lemma}\label{lemma:tangent-general-new}
  Let $G$ and $P$ be as before.
  \begin{enumerate}
    \item\label{enumerate:tangent-1} The tangent bundle~$\tangent_{G/P}$ is~$G$\dash equivariant and corresponds
      via \eqref{equation:cohG-GP} to the quotient~$\mathfrak{g}/\mathfrak{p}$ endowed
      with the adjoint action of $P$.

    \item\label{enumerate:tangent-2} The weights of $\mathfrak{g}/\mathfrak{p}$ are the \emph{non-parabolic positive roots}
      $\roots^+ \setminus \roots_L^+$. These are those positive roots of $G$, whose
      decomposition in terms of simple roots necessarily involves~$\alpha_k$ with
      a positive coefficient.

    \item\label{enumerate:tangent-3} There is an isomorphism of~$P$-representations~$\mathfrak{n}^\vee\cong\mathfrak{g}/\mathfrak{p}$.

    \item\label{enumerate:tangent-4} The exterior powers~$\bigwedge^p \tangent_{G/P}$ are~$G$\dash equivariant
      and correspond to the representations~$\bigwedge^p \mathfrak{g}/\mathfrak{p}$.

    \item\label{enumerate:tangent-5} Weights of $\bigwedge^p \mathfrak{g}/\mathfrak{p}$ are sums of $p$ distinct
      non-parabolic positive roots $\roots^+ \setminus \roots_L^+$. In particular, for
      any weight $\beta$ appearing in $\bigwedge^p \mathfrak{g}/\mathfrak{p}$ we have
      \begin{equation}
        \left( \beta , \fundamental{k}^{\vee} \right) \geq p.
      \end{equation}
  \end{enumerate}
\end{lemma}

\begin{proof}
  \cref{enumerate:tangent-1}, \cref{enumerate:tangent-2} and \cref{enumerate:tangent-3} can be found in \cite[\S3]{MR1038279}.
  \cref{enumerate:tangent-4} follows from the equivalence \eqref{equation:equivalence-of-categories}
  and the identification of~$P$\dash representations with~$\mathfrak{p}$\dash representations.
  \cref{enumerate:tangent-5} follows from \cref{enumerate:tangent-2}.
\end{proof}

\paragraph{A filtration for exterior powers of the tangent bundle}
Often the tangent bundle $\tangent_{G/P}$ and its exterior powers $\bigwedge^p \tangent_{G/P}$
are not completely reducible, see \cref{lemma:tangent-cominuscule-new}.
Hence, one cannot directly apply Borel--Weil--Bott to compute cohomology of $\bigwedge^p \tangent_{G/P}$.

However, one can try to bypass this obstacle by considering a filtration of $\bigwedge^p \tangent_{G/P}$,
or equivalently of $\bigwedge^p \mathfrak{g}/\mathfrak{p}$, whose associated graded is completely reducible.
One possibility is using a composition series in the setting of the Jordan--H\"older theorem, but following
Konno \cite[\S3]{MR0867067} we define the following filtration, which is shorter but nevertheless gives
a completely reducible associated graded.

\begin{definition}
  \label{definition:konno-filtration}
  Fix~$p\geq 0$.
  We define a filtration
  \begin{equation}
    \label{equation:konno-filtration}
    \bigwedge^p \mathfrak{g}/\mathfrak{p} = \FF^0 \left( \bigwedge^p \mathfrak{g}/\mathfrak{p} \right) \supseteq \FF^1 \left( \bigwedge^p \mathfrak{g}/\mathfrak{p} \right) \supseteq \ldots \\
  \end{equation}
  where
  \begin{equation}
    \label{equation:konno-filtration-pieces}
    \FF^i \left( \bigwedge^p \mathfrak{g}/\mathfrak{p} \right) \coloneqq \{ \text{subspace with weights $\beta$ such that } \left( \beta , \fundamental{k}^\vee \right) \leq \upperbound_p - i \}
  \end{equation}
  with $\upperbound_p$ being the maximum of $(\beta, \fundamental{k}^\vee)$ with
  $\beta$ ranging over the weights of $\bigwedge^p \mathfrak{g}/\mathfrak{p}$.
  Using the equivalence \eqref{equation:cohG-GP} we have an associated filtration
  \begin{equation}
    \bigwedge^p\tangent_{G/P}=\FF^0\left( \bigwedge^p\tangent_{G/P} \right)\supseteq\FF^1\left( \bigwedge^p\tangent_{G/P} \right)\supseteq\ldots
  \end{equation}
  of equivariant vector bundles.
\end{definition}

This filtration has the desired properties, by the following lemma.
\begin{lemma}
  \label{lemma:konno-filtration-properties}
  The filtration \eqref{equation:konno-filtration} is a finite decreasing filtration of~$\bigwedge^p \mathfrak{g}/\mathfrak{p}$ by subrepresentations of~$P$. Its associated graded pieces are completely reducible~$P$\dash representations.
\end{lemma}

\begin{proof}
  Since $\bigwedge^p \mathfrak{g}/\mathfrak{p}$ is finite-dimensional, it is clear that the filtration is finite and decreasing.

  To prove the complete reducibility of $\FF^i/\FF^{i+1}$ it is enough to show that $\mathfrak{n}$ acts trivially on $\FF^i/\FF^{i+1}$. Indeed, in such a case the representation is induced from the Levi subalgebra $\mathfrak{l}$, and is automatically completely reducible.

  Recall from \cref{subsection:new-setup-and-notation} that we have the decomposition
  \begin{equation}
    \mathfrak{n} = \bigoplus_{\beta \in \roots^- \setminus \roots^-_L} \mathfrak{g}_{\beta}.
  \end{equation}
  Hence, for any $\mathfrak{g}_{\beta}$ appearing in this decompositon the root $\beta$
  must have $\alpha_k$ with a negative coefficient in its expression in terms of simple
  roots.

  Let $V$ be a representation of $\mathfrak{p}$, and recall the basic fact that a root space $\mathfrak{g}_{\beta}$ maps a weight space $V_{\alpha}$ to the weight space $V_{\alpha + \beta}$, if such a weight space exists, or to zero otherwise. Applying this to our situation, we immediately see that $\mathfrak{n}$ maps $\FF^i$ to $\FF^{i+1}$. Therefore, the action of $\mathfrak{n}$ on $\FF^i/\FF^{i+1}$ is trivial.
\end{proof}

\begin{remark}
  For~$p=1$ this filtration is dual to the lower central series for~$\mathfrak{n}$.
\end{remark}

\paragraph{Spectral sequence associated to the filtration}
We will denote the~$i$th piece of the associated graded as
\begin{equation}
  \GG^i\left( \bigwedge^p\tangent_{G/P} \right)
  \coloneqq
  \FF^i\left( \bigwedge^p\tangent_{G/P} \right)/\FF^i\left( \bigwedge^p\tangent_{G/P} \right).
\end{equation}
We obtain the following spectral sequence for every~$p\geq 0$:
\begin{equation}
  \label{equation:konno}
  \mathrm{E}_1^{i,q-i}
  =
  \HH^q\left( G/P,\GG^i\left( \bigwedge^p\tangent_{G/P} \right) \right)
  \Rightarrow
  \HH^q(G/P,\bigwedge^p\tangent_{G/P}).
\end{equation}
On the~$\mathrm{E}_1$\dash page only cohomology of completely reducible equivariant vector bundles appears, for which one can use the Borel--Weil--Bott theorem. All maps in the spectral sequence are equivariant.


\subsection{Vanishing for cominuscule varieties}
We have the following characterisation of cominuscule maximal parabolics.
\begin{lemma}
  \label{lemma:tangent-cominuscule-new}
  Let~$G$ and~$B$ be as before.
  \begin{enumerate}
    \item\label{enumerate:cominuscule-one-step}
      $P$ is cominuscule if and only if the filtration \eqref{equation:konno-filtration}
      on $\bigwedge^p \mathfrak{g}/\mathfrak{p}$ is a one-step filtration for all $p$;
    \item\label{enumerate:cominuscule-reducible}
      If $P$ is cominuscule, then the $P$-representations $\bigwedge^p \mathfrak{g}/\mathfrak{p}$
      are completely reducible for all $p$;
    \item\label{enumerate:cominuscule-abelian}
      $P$ is cominuscule if and only if the nilradical $\mathfrak{n}$ is an abelian Lie algebra.
  \end{enumerate}
\end{lemma}

\begin{proof}
  \begin{enumerate}
    \item This follows from \eqref{equation:cominuscule-definition},
      \cref{lemma:tangent-general-new}~\cref{enumerate:tangent-5}, and \ref{equation:konno-filtration-pieces}.
    \item This follows from the first claim and \cref{lemma:konno-filtration-properties}.
    \item This fact is well-known \cite[Lemma~2.2]{MR1189494}.
  \end{enumerate}
\end{proof}

Let $G/P$ be a cominuscule variety. By \cref{lemma:tangent-cominuscule-new}
the tangent bundle $\tangent_{G/P}$ and its exterior powers $\bigwedge^p \tangent_{G/P}$
are completely reducible $G$-equivariant vector bundles. Hence, one can use the Borel--Weil--Bott
theorem to compute their cohomology by applying it to each irreducible
summand $\mathcal{E}^{\lambda}$ individually.

\begin{proposition}
  \label{proposition:weights-in-cominuscule-case-are-dominant}
  The highest weight $\lambda$ of an irreducible summand $\mathcal{E}^{\lambda}$
  of $\bigwedge^p \tangent_{G/P}$ on a cominuscule variety $G/P$ is $G$-dominant.
  In particular, $\bigwedge^p \tangent_{G/P}$ have no higher cohomology.
\end{proposition}

\begin{proof}
  Recall that our maximal parabolic $P$ corresponds to the $k$th vertex of the Dynkin diagram.
  Thus, it is enough to show $(\lambda, \alpha_k) \geq 0$. By \cref{lemma:tangent-general-new},
  any such $\lambda$ is of the form $\beta_1 + \ldots + \beta_p$ with $\beta_i \in \roots^+ \setminus \roots_L^+$.
  Hence, it is enough to show that $(\beta, \alpha_k) \geq 0$ for all $\beta \in \roots^+ \setminus \roots_L^+$,
  which in turn is equivalent to showing that $(\gamma, \alpha_k) \leq 0$ for all
  $\gamma \in \roots^- \setminus \roots_L^-$.

  Recall from \cref{subsection:new-setup-and-notation} that we have
  \begin{equation}
    \mathfrak{n} = \bigoplus_{\gamma \in \roots^- \setminus \roots_L^-} \mathfrak{g}_{\gamma}.
  \end{equation}
  Since $\mathfrak{n}$ is abelian by \cref{lemma:tangent-cominuscule-new}, it follows
  that for any $\gamma_{1,2} \in \roots^- \setminus \roots_L^-$ their sum  $\gamma_1 + \gamma_2$
  is never a root of $\mathfrak{g}$. Therefore, as $-\alpha_k$ is in $\roots^- \setminus \roots_L^-$,
  we obtain that for any $\gamma \in \roots^- \setminus \roots_L^-$ the difference $\gamma - \alpha_k$
  is never a root of $\mathfrak{g}$. Hence, we obtain $(\gamma, \alpha_k) \leq 0$ (see \cite[Lemma~9.4]{MR0499562}).
\end{proof}
%
%

\subsection{Vanishing for adjoint varieties}
We begin with the description of the nilradical $\mathfrak{n}$ in the adjoint case.
\begin{definition}
  \label{definition:heisenberg-lie-algebra}
  The \emph{$r$th Heisenberg Lie algebra} $\mathfrak{n}_r$ is the~$2r+1$\dash dimensional Lie algebra defined as a central extension by~$\groundfield \cdot e_0$ of a~$2r$\dash dimensional abelian Lie algebra spanned by elements~$e_1,\ldots,e_{2r}$ such that~$[e_i,e_{i+r}]=-[e_{i+r},e_i]=e_0$ for all~$i=1,\ldots,r$ with all other brackets of basis vectors being zero.
\end{definition}

In particular, the lower central series of the Heisenberg Lie algebra~$\mathfrak{n}=\mathfrak{n}_r$ is of the form
\begin{equation}
  \label{equation:ses-heisenberg}
  0\to[\mathfrak{n},\mathfrak{n}]\to\mathfrak{n}\to\mathfrak{n}/[\mathfrak{n},\mathfrak{n}]=\mathfrak{n}^\ab\to 0,
\end{equation}
with~$\dim_{\groundfield}[\mathfrak{n},\mathfrak{n}]=1$.

We have an ample supply of parabolic subalgebras whose nilradical has this property.
\begin{lemma}
  \label{lemma:adjoint-heisenberg-nilradical}
  Let~$G/P$ be an adjoint generalised Grassmannian. Then the nilradical $\mathfrak{n}$ is a Heisenberg Lie algebra.
\end{lemma}

\begin{proof}
  Recall that our maximal parabolic $P$ corresponds to the $k$th vertex of the Dynkin diagram.
  From the description of the highest root in \cite[Appendix]{MR0240238}
  we obtain that
  \begin{equation}\label{equation:inequalities-adjoint-proof-heisenberg}
    \begin{aligned}
      & \left( \alpha, \fundamental{k}^\vee \right) \leq 1 \quad \text{for} \quad \alpha \in \roots^+ \setminus \{\Theta\}, \\
      & \left( \Theta, \fundamental{k}^\vee \right) = 2,
    \end{aligned}
  \end{equation}
  where $\Theta$ is the highest root.
  This corresponds to the classification of quasi-cominuscule non-cominuscule weights in \cite[Remark~2.3]{MR0553746}.
  From the root decomposition for $\mfn$ and
  \eqref{equation:inequalities-adjoint-proof-heisenberg} it follows that
  $[\mathfrak{n},\mathfrak{n}] = \mathfrak{g}_{-\Theta} \subset\mathrm{Z}(\mathfrak{n})$.
  One can conclude as in \cite[\S4.2.1]{MR2532439} that the induced pairing is actually non-degenerate.
  %
\end{proof}

Thus, for an adjoint generalised Grassmannian $G/P$ we have the following description of the nilradical
\begin{equation}
  \label{equation:heisenberg-lie-algebra}
  \mathfrak{n} = \mathfrak{g}_{\gamma_0} \oplus \left( \bigoplus_{i=1}^r \mathfrak{g}_{\gamma_i} \right) \oplus \left( \bigoplus_{i=r+1}^{2r} \mathfrak{g}_{\gamma_i} \right),
\end{equation}
where $\gamma_i \in \roots^- \setminus \roots_L^-$ corresponds to $e_i$ in \cref{definition:heisenberg-lie-algebra}.

Let $G/P$ be an adjoint generalised Grassmannian. In this case neither the tangent bundle~$\tangent_{G/P}$, nor its exterior powers $\bigwedge^p \tangent_{G/P}$, are completely reducible, and, therefore, we cannot directly apply the Borel--Weil--Bott theorem to prove vanishing and we need to appeal to an appropriate filtration whose associated graded is completely reducible.

Let $\mathcal{E}^{\lambda}$ be an irreducible direct summand of a graded piece of the filtration with highest weight $\lambda$. To show vanishing of higher cohomology of $\bigwedge^p \tangent_{G/P}$ it is enough to show vanishing of higher cohomology of any such $\mathcal{E}^{\lambda}$.

\begin{proposition}
  \label{proposition:vanishing-for-summands-adjoint-case}
  Let~$G/P$ be an adjoint generalised Grassmannian. Any irreducible direct summand
  $\mathcal{E}^{\lambda}$ of a graded piece of the filtration on~$\bigwedge^p\tangent_{G/P}$
  has no higher cohomology.
\end{proposition}

\begin{proof}
  Recall that our maximal parabolic $P$ corresponds to the $k$th vertex of the Dynkin diagram.
  Let $\lambda$ be the highest weight of such an irreducible direct summand $\mathcal{E}^{\lambda}$.
  As $\mathcal{E}^{\lambda}$ is irreducible, we can use Borel--Weil--Bott to compute
  its cohomology. Assume that $\mathcal{E}^{\lambda}$ has non-trivial cohomology,
  then the weight $\lambda + \rho$ has to be regular. We are going to show that this
  non-trivial cohomology can only live in degree zero.

  By \cref{lemma:new-lemma-on-weights}~\cref{enumerate:weights-5}
  it is enough to show that $( \lambda + \rho , \alpha_k^\vee ) \geq 0$,
  which in its turn is equivalent to $( \lambda , \alpha_k^\vee ) \geq -1$.
  By \cref{lemma:tangent-general-new}~\cref{enumerate:weights-5},
  any such $\lambda$ is of the form $\beta_1 + \ldots + \beta_p$ with $\beta_i \in \roots^+ \setminus \roots_L^+$.
  Hence, the statement of the proposition follows from the following simple lemma.
\end{proof}

\begin{lemma}
  \label{lemma:npp-root-with-negative-pairing}
  If $\dim_\groundfield [\mathfrak{n}, \mathfrak{n}] = 1$, then there exists a unique $\beta \in \roots^+ \setminus \roots_L^+$
  such that $( \beta, \alpha_k^\vee ) < 0$. Namely, we have
  \begin{equation}
    \beta = \Theta - \alpha_k \quad \text{and} \quad  \left( \beta, \alpha_k^\vee \right) = -1,
  \end{equation}
  where $\Theta$ is the highest root.
\end{lemma}

\begin{proof}
  We begin by proving that there exists at most one $\beta \in \roots^+ \setminus \roots_L^+$
  with $( \beta, \alpha_k^\vee ) < 0$. If $\beta \in \roots^+ \setminus \roots_L^+$
  is such an element, then by \cite[Lemma~9.4]{MR0499562} the sum $\beta + \alpha_k$ is a root
  (and lies in $\roots^+ \setminus \roots_L^+$). Since $\beta + \alpha_k$ is a root, the root
  subspaces $\mathfrak{g}_{-\beta}$ and $\mathfrak{g}_{-\alpha_k}$ have a non-trivial Lie bracket
  equal to $\mathfrak{g}_{-\beta-\alpha_k}$.
  Using \eqref{equation:heisenberg-lie-algebra}, we obtain that
  $-\beta-\alpha_k = \gamma_i$ for some $i \in \{0,1,\ldots,2m\}$.
  From the explicit description of the Lie bracket of a Heisenberg Lie algebra
  (see \cref{definition:heisenberg-lie-algebra}) we conclude that $-\beta-\alpha_k = \gamma_0 = -\Theta$,
  and obtain the desired $\beta = \Theta-\alpha_k$.

  Since $\Theta$ is the $k$th fundamental weight and since by definition
  we have $( \alpha_k, \alpha_k^\vee ) = 2$, the equality $( \beta, \alpha_k^\vee) = -1$ follows.
\end{proof}

\begin{remark}\label{remark:vanishing-adjoint-type-a}
  We will now explain how the discussion changes for the adjoint partial flag variety in type~A,
  which is not a generalised Grassmannian.

  First we note that \cref{lemma:adjoint-heisenberg-nilradical} and its proof carry over almost verbatim to this case.
  The coweight $\fundamental{k}^\vee$ needs to be replaced by the coweight $(\fundamental{1} + \fundamental{n})^\vee$
  (defined in \cref{subsection:new-setup-and-notation}).

  The setup for \cref{proposition:vanishing-for-summands-adjoint-case} carries over as follows:
  taking any filtration of $\tangent_{G/P}$ with completely reducible associated graded
  (e.g.~the Jordan-H\"older filtration, or a suitable generalisation of Konno's filtration)
  we again apply \cref{lemma:new-lemma-on-weights}~\cref{enumerate:weights-5} and need to show
  \begin{equation}
    \left( \lambda , \alpha_k^\vee \right) \geq -1 \quad \text{for} \quad k = 1,n.
  \end{equation}
  Hence, the statement of the proposition follows from the analogue of \cref{lemma:npp-root-with-negative-pairing}
  given below. Thus, we have the desired vanishing.
\end{remark}

\begin{lemma}
  Let $G/P$ be the adjoint variety in type~A. For~$k=1,n$ there exists
  a unique $\beta \in \roots^+ \setminus \roots_L^+$ such that $( \beta, \alpha_k^\vee ) < 0$.
  Namely, we have
  \begin{equation}
    \beta = \Theta - \alpha_k \quad \text{and} \quad  \left( \beta, \alpha_k^\vee \right) = -1,
  \end{equation}
  where $\Theta$ is the highest root.
\end{lemma}

\begin{proof}
  As in the proof of \cref{lemma:npp-root-with-negative-pairing} for $k=1,n$ one
  shows that an element $\beta \in \roots^+ \setminus \roots_L^+$ with the property
  $( \beta, \alpha_k^\vee ) < 0$ must be of the form $\beta = \Theta-\alpha_k$.
  Since $\Theta = \omega_1 + \omega_n$, the equality $( \beta, \alpha_k^\vee ) = -1$ also follows.
\end{proof}

\subsection{Vanishing for coadjoint non-adjoint varieties}
Now we are left with just two cases: the symplectic Grassmannians~$\SGr(2,2n)$ and the exceptional Grassmannian~$(\mathrm{F}_4,\alpha_4)$. We need an analogue of \cref{lemma:npp-root-with-negative-pairing} for these cases. Instead of adapting the argument from the previous section to the dual root system we are going to show this by a direct computation, using the description of roots from \cite[Appendix]{MR0240238}.

\begin{lemma}
  For the symplectic Grassmannian $\SGr(2,2n) = (\mathrm{C}_n,\alpha_2)$ there exist
  a unique $\beta \in \roots^+ \setminus \roots^+_L$ such that $( \beta, \alpha_2^\vee ) < 0$.
  Namely, we have
  \begin{equation}
    \beta = \theta - \alpha_2 \quad \text{and} \quad  ( \beta, \alpha_2^\vee ) = -1,
  \end{equation}
  where $\theta$ is the highest short root.
\end{lemma}

\begin{proof}
  The second coroot is
  \begin{equation}
    \alpha_2^\vee = e_2 - e_3.
  \end{equation}
  Computing the pairing $\left( \alpha, \alpha_2^\vee \right)$ for all $\alpha \in \roots^+$ we see that the only possible negative values of the pairing are $-1$ and $-2$.

  \begin{enumerate}
    \item The value $-2$ arises only once as $\left( 2e_3, \alpha_2^\vee \right) = -2$. However, since
      \begin{equation}
        2e_3 = 2 \alpha_3 + 2 \alpha_4 + \ldots + 2 \alpha_{n-1} + \alpha_n,
      \end{equation}
      the root $2e_3$ in not in $\roots^+ \setminus \roots^+_L$.

    \item The value $-1$ arises multiple times. Namely, the positive root $\alpha$ can be~$e_1 + e_3$, $e_3 + e_j$ with $j > 3$, $e_1 - e_2$ or $e_3 - e_j$ with $j > 3$. Now one checks easily that in all cases except the first one the roots are not in $\roots^+ \setminus \roots^+_L$. In the first case we have
      \begin{equation}
        e_1 + e_3 = \alpha_1 + \alpha_2 + 2 \alpha_3 + \ldots + 2 \alpha_{n-1} + \alpha_n = \theta - \alpha_2.
      \end{equation}
      This finishes the proof.
  \end{enumerate}
\end{proof}

\begin{lemma}
  For the generalised Grassmannian $(\mathrm{F}_4,\alpha_4)$ there exist a unique
  $\beta \in \roots^+ \setminus \roots^+_L$ such that $\left( \beta, \alpha_4^\vee \right) < 0$.
  Namely, we have
  \begin{equation}
    \beta = \theta - \alpha_4 \quad \text{and} \quad  \left( \beta, \alpha_4^\vee \right) = -1,
  \end{equation}
  where $\theta$ is the highest short root.
\end{lemma}

\begin{proof}
  The fourth coroot is
  \begin{equation}
    \alpha_4^\vee = e_1 - e_2 - e_3 - e_4.
  \end{equation}

  One checks easily that the positive roots with negative pairing with $\alpha_4^\vee$ are
  \begin{equation}
    \begin{aligned}
      & e_i \quad \text{for} \quad 2 \leq i \leq 4 \\
      & \frac{1}{2} (e_1 + e_2 + e_3 + e_4)
    \end{aligned}
  \end{equation}
  in which case the pairing is $-1$, and
  \begin{equation}
    \begin{aligned}
      & e_i + e_j \quad \text{for} \quad 2 \leq i < j \leq 4,
    \end{aligned}
  \end{equation}
  in which case the pairing is $-2$. Rewriting all of them in terms of simple roots we obtain the list $\alpha_1 + \alpha_2 + \alpha_3$, $\alpha_2 + \alpha_3$, $\alpha_3$, $\alpha_1 + 2\alpha_2 + 3\alpha_3 + \alpha_4$, $\alpha_2 + 2\alpha_3$, $\alpha_1 + \alpha_2 + 2\alpha_3$, $\alpha_1 + 2\alpha_2 + 2\alpha_3$ and we see that the only non-parabolic root is
  \begin{equation}
    \alpha_1 + 2\alpha_2 + 3\alpha_3 + \alpha_4 = \frac{1}{2} (e_1 + e_2 + e_3 + e_4) = \theta - \alpha_4.
  \end{equation}
  This finishes the proof.
\end{proof}

\section{Description for cominuscule and adjoint varieties}
\label{section:description}
We will now prove \cref{theorem:cominuscule-decomposition,theorem:adjoint-decomposition}, by explicitly computing the global sections of~$\bigwedge^i\tangent_{G/P}$ for cominuscule and adjoint varieties, as representations of~$\mathfrak{g}$. In this way we will have described a part of the Gerstenhaber algebra structure on~$\HHHH^\bullet(G/P)$.

\subsection{Hochschild cohomology of cominuscule varieties}
\label{subsection:hochschild-cohomology-cominuscule}
As a warmup for the proof in \cref{subsection:hochschild-cohomology-adjoint} and to keep the discussion in this paper self-contained, we will give some details on the proof of \cref{theorem:cominuscule-decomposition}. This result is classical, and follows readily from Kostant's theorem.
Recall that the nilpotent radical~$\mathfrak{n}$ as a representation of~$P$ is associated to the cotangent bundle,
with its dual~$\mathfrak{g}/\mathfrak{p}$ associated to the tangent bundle.

\begin{proof}[Proof of \cref{theorem:cominuscule-decomposition}]
  Let~$G/P$ be a cominuscule variety associated to the fundamental weight~$\fundamental{k}$.
  By \cref{lemma:tangent-cominuscule-new}~\cref{enumerate:cominuscule-abelian} we have that~$\mathfrak{n}$ is an abelian Lie algebra.
  Therefore we get that the differential in the Chevalley--Eilenberg complex vanishes, hence
  \begin{equation}
    \HH_\CE^i(\mathfrak{n},k)\cong\bigwedge^i\mathfrak{n}
  \end{equation}
  whilst Kostant's theorem gives
  \begin{equation}
    \HH_\CE^i(\mathfrak{n},k)
    \cong
    \bigoplus_{\mathclap{\substack{w\in\cosets \\ \ell(w)=i}}}\representation{w\cdot 0}{\mathfrak{l}}
  \end{equation}
  as~$\mathfrak{l}$\dash representations, but also as~$\mathfrak{p}$\dash representations as~$\mathfrak{n}$ acts trivially
  (see \cref{lemma:tangent-cominuscule-new}~\cref{enumerate:cominuscule-reducible}).
  Therefore under the equivalences from \eqref{equation:cohG-GP} we have that
  \begin{equation}
    \Omega_{G/P}^i
    \cong
    \bigoplus_{\mathclap{\substack{w\in\cosets \\ \ell(w)=i}}}\bundle{w\cdot 0}
  \end{equation}
  as equivariant vector bundles. Using that~$\bigwedge^i\tangent_{G/P}\cong\Omega_{G/P}^{\dim G/P-i}\otimes\omega_{G/P}^\vee$, and~$\omega_{G/P}^\vee\cong\bundle{\fanoindex_{G/P}\fundamental{k}}$ where~$\fanoindex_{G/P}$ is the index of~$G/P$, we apply the Borel--Weil--Bott theorem. As already observed in \cref{proposition:weights-in-cominuscule-case-are-dominant}, the resulting weights for the summands of~$\bigwedge^i\tangent_{G/P}$ are all dominant, thus contribute (non-trivially) to the global sections (and not in higher degree), leading to the decomposition \eqref{equation:cominuscule-decomposition} in the statement of \cref{theorem:cominuscule-decomposition}.
\end{proof}

So what allowed us to conclude in this case is the fact that~$\mathfrak{n}$ is abelian, effectively reducing the computation to Kostant's theorem and applying Borel--Weil--Bott.

\begin{remark}
  Without the twist by~$\omega_{G/P}^\vee$ one is actually computing the Hodge numbers of~$G/P$. By Borel--Hirzebruch we know that the non-zero Hodge numbers~$\mathrm{h}^{i,i}(G/P)$ are given by the cardinality of the subset of~$\cosets$ of elements of length~$i$, i.e.,~if one were to use Borel--Weil--Bott to compute this, all weights are regular but not dominant for~$i\geq 1$, and their index is precisely~$i$.
\end{remark}

Because this result is standard, we will only give one small example.

\begin{example}
  \label{example:A3-P2}
  In type~A every Grassmannian is a cominuscule partial flag variety.
  Let us consider the case~$\Gr(2,4)$, which has dimension~4 and index~4 (and it is isomorphic to the quadric~$Q^4$).
  The parabolic Bruhat graph describing~$\cosets$ is given in \cref{figure:parabolic-bruhat-A3-P2},
  so we just compute the weights~$w\cdot 0+4\fundamental{2}$ and obtain \cref{table:wedge-A3-P2}.

  \begin{table}[ht!]
    \centering
    \begin{tabular}{rrrrrr}
      \toprule
      weight      & rank & degree & representation & dimension & sum of roots \\
      \midrule
      $(0, 0, 0)$ & $1$  & $0$    & $(0, 0, 0)$    & $1$       & $(0, 0, 0)$ \\
      \midrule
      $(1, 0, 1)$ & $4$  & $0$    & $(1, 0, 1)$    & $15$      & $(1, 1, 1)$ \\
      \midrule
      $(0, 1, 2)$ & $3$  & $0$    & $(0, 1, 2)$    & $45$      & $(1, 2, 2)$ \\
      $(2, 1, 0)$ & $3$  & $0$    & $(2, 1, 0)$    & $45$      & $(2, 2, 1)$ \\
      \midrule
      $(1, 2, 1)$ & $4$  & $0$    & $(1, 2, 1)$    & $175$     & $(2, 3, 2)$ \\
      \midrule
      $(0, 4, 0)$ & $1$  & $0$    & $(0, 4, 0)$    & $105$     & $(2, 4, 2)$ \\
    \bottomrule
    \end{tabular}
    \caption{Associated graded for $\bigwedge^\bullet\tangent_{\Gr(2,4)}$}
    \label{table:wedge-A3-P2}
  \end{table}
\end{example}

\subsection{Hochschild cohomology of adjoint varieties}
\label{subsection:hochschild-cohomology-adjoint}
In this section we prove \cref{theorem:adjoint-decomposition}.
By \cref{lemma:adjoint-heisenberg-nilradical}
the nilradical in this case is a Heisenberg Lie algebra,
which means that Kostant's theorem doesn't compute the exterior powers of the tangent bundle on the nose.
But it is possible to bootstrap from this theorem,
as the structure of~$\mathfrak{n}$ is still manageable.

One of the ingredients in the proof of \cref{theorem:adjoint-decomposition} is the following description of the Betti numbers of~$\HH_\CE^\bullet(\mathfrak{n},\groundfield)$, for which an elementary proof can be found as \cite[Theorem~2.2(i)]{MR0677223}. A more conceptual (and lengthier) proof can be found as \cite[Corollary~4.4]{MR2026860}. 

\begin{proposition}[Santharoubane]
  \label{proposition:santharoubane}
  Let~$\mathfrak{n}_r$ be the Heisenberg Lie algebra of dimension~$2r+1$. Then
  \begin{equation}
    \dim_\groundfield\HH_\CE^i(\mathfrak{n},\groundfield)=
    \begin{cases}
      \binom{2r}{i}-\binom{2r}{i-2} & i=0,\ldots,r \\
      \binom{2r}{i-2}-\binom{2r}{i} & i=r+1,\ldots,2r+1. 
    \end{cases}
  \end{equation}
\end{proposition}

\paragraph{The Hochschild--Serre spectral sequence}
The Hochschild--Serre spectral sequence associated to the sequence \eqref{equation:ses-heisenberg} is
\begin{equation}
  \mathrm{E}_2^{p,q}
  =
  \HH_\CE^p\big( \mathfrak{n}^\ab,\HH_\CE^q([\mathfrak{n},\mathfrak{n}],\groundfield) \big)
  \Rightarrow
  \HH_\CE^{p+q}(\mathfrak{n},\groundfield)
\end{equation}
In the adjoint case, $[\mathfrak{n},\mathfrak{n}]$ is 1-dimensional, so the sequence is concentrated in 2 rows, and it degenerates at the~$\mathrm{E}_3$\dash page. As will become clear, the spectral sequence is highly non-degenerate on the~$\mathrm{E}_2$\dash page.

Because~$\HH_\CE^0([\mathfrak{n},\mathfrak{n}],\groundfield)\cong\groundfield$
and~$\HH_\CE^1([\mathfrak{n},\mathfrak{n}],\groundfield)\cong[\mathfrak{n},\mathfrak{n}]$ as~$\mathfrak{l}$\dash representations,
the~$\mathrm{E}_2$\dash page of the Hochschild--Serre spectral sequence has the form
\begin{equation}
  \label{equation:hochschild-serre-visual}
  \begin{tikzcd}[column sep=small]
    \HH_\CE^0(\mathfrak{n}^\ab,[\mathfrak{n},\mathfrak{n}]) \arrow[rrd, "\mathrm{d}_2^{0,1}"] & \HH_\CE^1(\mathfrak{n}^\ab,[\mathfrak{n},\mathfrak{n}]) \arrow[rrd, "\mathrm{d}_2^{1,1}"] & \HH_\CE^2(\mathfrak{n}^\ab,[\mathfrak{n},\mathfrak{n}]) \arrow[rrd, "\mathrm{d}_2^{2,1}"] & \HH_\CE^3(\mathfrak{n}^\ab,[\mathfrak{n},\mathfrak{n}]) & \ldots \\
    \HH_\CE^0(\mathfrak{n}^\ab,\groundfield) & \HH_\CE^1(\mathfrak{n}^\ab,\groundfield) & \HH_\CE^2(\mathfrak{n}^\ab,\groundfield) & \HH_\CE^3(\mathfrak{n}^\ab,\groundfield) & \ldots
  \end{tikzcd}
\end{equation}
with all terms zero outside~$\{0,1\}\times\{0,1,\ldots,2r\}$.
As~$\mathfrak{n}^\ab$ is abelian,
and the action of~$\mathfrak{n}^\ab$ on both~$\groundfield$ and~$[\mathfrak{n},\mathfrak{n}]$
is trivial by \eqref{equation:heisenberg-lie-algebra},
the differential in the Chevalley--Eilenberg complex vanishes, and we have that
\begin{equation}
  \label{equation:dim-E2-term}
  \dim_\groundfield\mathrm{E}_2^{p,q}
  =
  \binom{2r}{p}
\end{equation}
for~$q=0,1$ and~$p=0,\ldots,2r$.

The following lemma is the key result in describing the Hochschild cohomology of partial flag varieties of adjoint type.
\begin{lemma}
  \label{lemma:key-lemma-hochschild-serre}
  The differentials~$\mathrm{d}_2^{i,1}\colon\HH_\CE^i(\mathfrak{n}^\ab,[\mathfrak{n},\mathfrak{n}])\to\HH_\CE^{i+2}(\mathfrak{n}^\ab,\groundfield)$ in the Hochschild--Serre spectral sequence \eqref{equation:hochschild-serre-visual} are
  \begin{enumerate}
    \item injective for~$i\leq r-1$;
    \item surjective for~$i\geq r-1$.
  \end{enumerate}
  In particular, the differential~$\mathrm{d}_2^{r-1,1}$ is an isomorphism.
\end{lemma}

\begin{proof}
  The proof for injectivity is by induction on~$i$. The statement is vacuous for~$i=-2,-1$ as the domain is zero. By \cref{proposition:santharoubane} we have that
  \begin{equation}
    \dim_\groundfield\HH_\CE^1(\mathfrak{n},\groundfield)=\binom{2r}{1}=\dim_\groundfield\mathrm{E}_\infty^{0,1}+\dim_\groundfield\mathrm{E}_\infty^{1,0}.
  \end{equation}
  Because~$\mathrm{E}_2^{1,0}$ has no incoming differential and is~$\binom{2r}{1}$\dash dimensional, we see that~$\mathrm{d}_2^{0,1}$ must be injective, so that~$\mathrm{E}_3^{0,1}=\mathrm{E}_\infty^{0,1}=0$. Continuing by induction for~$i=2,\ldots,r$ we use \cref{proposition:santharoubane}, together with \eqref{equation:dim-E2-term} to conclude that all differentials must be injective so that the appropriate dimension in the abutment is reached.

  The proof for surjectivity is by a descending induction on~$i$, and is similar.
\end{proof}
Hence the entries~$\mathrm{E}_3^{i,j}$ on the~$\mathrm{E}_3=\mathrm{E}_\infty$-page look like
\begin{equation}
  \label{equation:E3-low}
  \begin{tikzcd}[column sep=tiny, row sep=tiny]
    0 & 0 & 0 & 0 & \ldots \\
    \HH_\CE^0(\mathfrak{n}^\ab,\groundfield) & \HH_\CE^1(\mathfrak{n}^\ab,\groundfield) & \coker\mathrm{d}_2^{0,1} & \coker\mathrm{d}_2^{1,1} & \ldots
  \end{tikzcd}
\end{equation}
for~$i=0,1,2,3$, resp.
\begin{equation}
  \label{equation:E3-middle}
  \begin{tikzcd}[column sep=tiny, row sep=tiny]
    \ldots & 0 & 0 & \ker\mathrm{d}_2^{r,1} & \ker\mathrm{d}_2^{r+1,1} & \ker\mathrm{d}_2^{r+2,1} & \ldots \\
    \ldots & \coker\mathrm{d}_2^{r-4,1} & \coker\mathrm{d}_2^{r-3,1} & \coker\mathrm{d}_2^{r-2,1} & 0 & 0 & \ldots
  \end{tikzcd}
\end{equation}
for~$i=r-2,r-1,\ldots,r+2$, resp.
\begin{equation}
  \label{equation:E3-high}
  \begin{tikzcd}[column sep=tiny, row sep=tiny]
    \ldots & \ker\mathrm{d}_2^{2r-3,1} & \ker\mathrm{d}_2^{2r-2,1} & \HH_\CE^{2r-1}(\mathfrak{n}^\ab,[\mathfrak{n},\mathfrak{n}]) & \HH_\CE^{2r}(\mathfrak{n}^\ab,[\mathfrak{n},\mathfrak{n}]) \\
    \ldots & 0 & 0 & 0 & 0
  \end{tikzcd}
\end{equation}
for~$i=2r-3,2r-2,2r-1,2r$.

Using Kostant's theorem (see \cref{theorem:kostant}) we have a description for the Lie algebra cohomology~$\HH_\CE^\bullet(\mathfrak{n},\groundfield)$,
but there is no immediate link with exterior powers of the (co)tangent bundle anymore.
Rather we have the following lemma.
\begin{lemma}
  Let~$G/P$ be an adjoint variety of dimension~$2r+1$. There exists a short exact sequence
  \begin{equation}
    \label{equation:ses-adjoint-tangent}
    0\to\mathcal{E}\to\tangent_{G/P}\to\mathcal{L}\to 0
  \end{equation}
  where
  \begin{itemize}
    \item $\mathcal{L}$ is the line bundle~$\mathcal{O}_{G/P}(1)$ (in type~A more appropriately written~$\mathcal{O}_{G/P}(1,1)$),
    \item $\mathcal{E}$ is the vector bundle of rank~$2r$ associated to the dual of the~$L$\dash representation~$\mathfrak{n}^\ab$.
  \end{itemize}
  Outside type~A we have that~$\mathfrak{n}^{\ab,\vee}$ is irreducible, in type~A is the direct sum of two irreducible representations.
\end{lemma}

\begin{proof}
  The sequence is the dual of the short exact sequence of equivariant vector bundles associated to \eqref{equation:ses-heisenberg}. As~$[\mathfrak{n},\mathfrak{n}]$ is one-dimensional we have that it is irreducible, and the highest weight of its dual corresponds to the adjoint representation. As the action of~$\mathfrak{n}$ on~$\mathfrak{n}^\ab$ is trivial, we have that it is completely reducible.
\end{proof}

If we wish to compute the global sections of~$\bigwedge^p\tangent_{G/P}$ we are reduced to computing the global sections of the short exact sequence
\begin{equation}
  \label{equation:ses-adjoint-wedge-tangent}
  0\to\bigwedge^p\mathcal{E}\to\bigwedge^p\tangent_{G/P}\to\mathcal{L}\otimes\bigwedge^{p-1}\mathcal{E}\to 0.
\end{equation}
The outer terms are \emph{completely reducible} equivariant vector bundles associated to the duals of the~$L$\dash representations~$\HH_\CE^{p-1}(\mathfrak{n}^\ab,[\mathfrak{n},\mathfrak{n}])$ and~$\HH_\CE^p(\mathfrak{n}^\ab,\groundfield)$. They can be determined inductively from the Hochschild--Serre spectral sequence and the knowledge of its abutment as follows.

\begin{proof}[Proof of \cref{theorem:adjoint-decomposition}]
  From the vanishing result in \cref{proposition:vanishing-for-summands-adjoint-case}, we obtain that
  \begin{equation}
    \HHHH^i(G/P)
    \cong
    \HH^0(G/P,\bigwedge^i\tangent_{G/P})
    \cong
    \HH^0(G/P,\bigwedge^i\mathcal{E})\oplus\HH^0(G/P,\mathcal{L}\otimes\bigwedge^{i-1}\mathcal{E}).
  \end{equation}
  Hence it suffices to describe the highest weights that determine the bundles~$\bigwedge^i\mathcal{E}$, from which the description for~$\mathcal{L}\otimes\bigwedge^{i-1}\mathcal{E}$ follows. Recall that by Borel--Weil--Bott the weights for both are either regular dominant, or singular\footnote{This observation explains why we have to consider the restricted sum in \eqref{equation:restricted-kostant} when describing the global sections (see also \cref{example:restricted-sum}).}.

  To determine~$\HH_\CE^i(\mathfrak{n}^\ab,\groundfield)$ (or rather its dual) we use the description of the~$\mathrm{E}_3$\dash pages \eqref{equation:E3-low}, \eqref{equation:E3-middle} and \eqref{equation:E3-high}. For~$i=0,1$ it is given by Kostant's description of~$\HH_\CE^i(\mathfrak{n},\groundfield)$.

  For~$i=2,\ldots,r$ there is a recursion involving the contributions of~$\HH_\CE^{i-2}(\mathfrak{n}^\ab,[\mathfrak{n},\mathfrak{n}])$ which are determined by the isomorphism
  \begin{equation}
    \HH_\CE^{i-2}(\mathfrak{n}^\ab,[\mathfrak{n},\mathfrak{n}])
    \cong
    \HH_\CE^{i-2}(\mathfrak{n}^\ab,\groundfield)\otimes_\groundfield[\mathfrak{n},\mathfrak{n}].
  \end{equation}
  The argument is dual for the second half, starting with~$i=2r$ and recursing downwards to~$i=r$. There is a shift by an extra copy of~$[\mathfrak{n},\mathfrak{n}]^\vee$ originating from the fact that \eqref{equation:E3-high} has zeroes on the bottom row, so that Kostant's theorem is rather describing the \emph{top row} of the~$\mathrm{E}_3$\dash page.

  Now the formula in \eqref{equation:adjoint-decomposition} is obtained by keeping track of the recursion with steps of size~2 and the contributions of~$\bigwedge^i\mathcal{E}$ and~$\mathcal{L}\otimes\bigwedge^{i-1}\mathcal{L}$.
\end{proof}


\begin{example}
  \label{example:restricted-sum}
  The necessity to restrict only to regular weights is obvious already for~$\tangent_{G/P}$. The sequence \eqref{equation:ses-adjoint-tangent} gives rise to the short exact sequence
  \begin{equation}
    0\to 0\to\mathfrak{g}\to\mathfrak{g}\to 0
  \end{equation}
  after taking global sections, so there is no contribution from~$\mathcal{E}$ in the description \eqref{equation:adjoint-decomposition}.
\end{example}

We will give two examples in full detail, to illustrate the somewhat involved recursive procedure outlined above in practice.

\begin{example}
  \label{example:A3-P13}
  The adjoint partial flag variety in type~$\mathrm{A}_3$ is~$\mathbb{P}(\tangent_{\mathbb{P}^3})$, which has dimension~5 and index~3. The parabolic Bruhat graph in \cref{figure:parabolic-bruhat-A3-P13} can be used in conjunction with \cref{theorem:adjoint-decomposition} to determine the Hochschild cohomology.

  In \cref{table:wedge-A3-P13} the associated graded of~$\bigwedge^i\tangent_{\mathbb{P}(\tangent_{\mathbb{P}^3})})$ is given, for~$i=0,\ldots,5$. The decomposition obtained from \eqref{equation:ses-adjoint-tangent} and \eqref{equation:ses-adjoint-wedge-tangent} is indicated by the grouping of the terms: we first give~$\bigwedge^{i-1}\otimes\mathcal{L}$. Again the need for the restriction to only regular weights in \cref{theorem:adjoint-decomposition} is immediate.

  For~$\tangent_{\mathbb{P}(\tangent_{\mathbb{P}^3})}$ we have 3~summands: one coming from~$\mathcal{L}$, the other two coming from~$\mathcal{E}$. That there are two follows from the fact that there are two Weyl group elements of colength~1 in \cref{figure:parabolic-bruhat-A3-P13}.

  For~$\bigwedge^2\tangent_{\mathbb{P}(\tangent_{\mathbb{P}^3})}$ there are 6~summands: 2~coming from~$\mathcal{E}\otimes\mathcal{L}$, the other~4 coming from~$\bigwedge^2\mathcal{E}$. That there are~4 is part of the recursion: there are three Weyl group elements of colength~2, and one of colength~0.

  For~$\bigwedge^3\tangent_{\mathbb{P}(\tangent_{\mathbb{P}^3})}$ the roles are reversed: 4~summands come from~$\bigwedge^2\mathcal{E}$ which was determined in the previous step, whilst there are~2 summands coming from~$\bigwedge^3\mathcal{E}$. The rest is similar.

  \begin{table}
    \centering
    \begin{tabular}{crrcrcr}
      \toprule
      weight         & rank & degree & representation & dimension & sum of roots \\
      \midrule
      $(0,0,0)$      & 1    & 0      & $(0,0,0)$      & 1         & $(0,0,0)$ \\

      \midrule
      $(1,0,1)$      & 1    & 0      & $(1,0,1)$      & 15        & $(1,1,1)$ \\
      \addlinespace
      $(1,1,-1)$     & 2    &        &                &           & $(1,1,0)$ \\
      $(-1,1,1)$     & 2    &        &                &           & $(0,1,1)$ \\

      \midrule
      $(0,1,2)$      & 2    & 0      & $(0,1,2)$      & 45        & $(1,2,2)$ \\
      $(2,1,0)$      & 2    & 0      & $(2,1,0)$      & 45        & $(2,1,0)$ \\
      \addlinespace
      $(-1,0,3)$     & 1    &        &                &           & $(0,1,2)$ \\
      $(3,0,-1)$     & 1    &        &                &           & $(2,1,0)$ \\
      $(1,0,1)$      & 1    & 0      & $(1,0,1)$      & 15        & $(1,1,1)$ \\
      $(0,2,0)$      & 3    & 0      & $(0,2,0)$      & 20        & $(1,2,1)$ \\

      \midrule
      $(4,0,0)$      & 1    & 0      & $(4,0,0)$      & 35        & $(3,2,1)$ \\
      $(0,0,4)$      & 1    & 0      & $(0,0,4)$      & 35        & $(1,2,3)$ \\
      $(2,0,2)$      & 1    & 0      & $(2,0,2)$      & 84        & $(2,2,2)$ \\
      $(1,2,1)$      & 3    & 0      & $(1,2,1)$      & 175       & $(2,3,2)$ \\
      \addlinespace
      $(0,1,2)$      & 2    & 0      & $(0,1,2)$      & 45        & $(1,2,2)$ \\
      $(2,1,0)$      & 2    & 0      & $(2,1,0)$      & 45        & $(2,2,1)$ \\

      \midrule
      $(1,1,3)$      & 2    & 0      & $(1,1,3)$      & 256       & $(2,3,3)$ \\
      $(3,1,1)$      & 2    & 0      & $(3,1,1)$      & 256       & $(3,3,2)$ \\
      \addlinespace
      $(2,0,2)$      & 1    & 0      & $(2,0,2)$      & 84        & $(2,2,2)$ \\

      \midrule
      $(3,0,3)$      & 1    & 0      & $(3,0,3)$      & 300       & $(3,3,3)$ \\
      \bottomrule
    \end{tabular}
    \caption{Associated graded for $\bigwedge^\bullet\tangent_{\mathbb{P}(\tangent_{\mathbb{P}^3})}$}
    \label{table:wedge-A3-P13}
  \end{table}
\end{example}

\begin{example}
  \label{example:B3-P2}
  The adjoint partial flag variety in type~$\mathrm{B}_3$ is~$\OGr(2,7)$, which has dimension~$7$ and index~$4$. The parabolic Bruhat graph in this case is given in \cref{figure:parabolic-bruhat-B3-P2}.

  \begin{figure}[ht!]
    \centering
    \begin{tikzpicture}
      \foreach \x/\y [count=\k] in {0/0, 1/0, 2/0.5, 2/-0.5, 3/0.5, 3/-0.5, 4/0.5, 4/-0.5, 5/0.5, 5/-0.5, 6/0, 7/0} {
        \coordinate (v\k) at (\x,\y);
      }

      \foreach \s/\t/\a in {
        1/2/2,
        2/3/1, 2/4/3,
        3/5/3, 4/5/1, 4/6/2,
        7/9/2, 8/9/1, 8/10/3,
        9/11/3, 10/11/1,
        11/12/2%
      } {
        \draw (v\s) -- (v\t) node [midway, fill=white] {\footnotesize \a};
      }

      \foreach \s/\t in {5/7, 5/8, 6/7, 6/8} {
        \draw (v\s) -- (v\t);
      }

      \foreach \k in {1,...,12} {
        \draw[fill=black] (v\k) circle (1.5pt);
      }
    \end{tikzpicture}
    \caption{Parabolic Bruhat graph for the (adjoint) parabolic $\dynkin[parabolic=2]{B}{3}$}
    \label{figure:parabolic-bruhat-B3-P2}
  \end{figure}

  In \cref{table:wedge-B3-P2} the associated graded of~$\bigwedge^i\tangent_{\OGr(2,7)}$ is given, for~$i=0,\ldots,7$. As before, the parabolic Bruhat graph in \cref{figure:parabolic-bruhat-B3-P2} can be used in conjunction with \cref{theorem:adjoint-decomposition} to determine the Hochschild cohomology. The decomposition obtained from \eqref{equation:ses-adjoint-tangent} and \eqref{equation:ses-adjoint-wedge-tangent} is indicated by the grouping of the terms: we first give~$\bigwedge^{i-1}\otimes\mathcal{L}$. Again the need for the restriction to only regular weights in \cref{theorem:adjoint-decomposition} is immediate.

  For~$\tangent_{\OGr(2,7)}$ we have two summands: one coming from~$\mathcal{L}$, the other coming from~$\mathcal{E}$. That there is only a single summand here is immediately visible from the parabolic Bruhat graph.

  For~$\bigwedge^2\tangent_{\OGr(2,7)}$ there are three ``new'' summands coming from~$\bigwedge^2\mathcal{E}$. This is again visible from the parabolic Bruhat graph, where there is one contribution of a line bundle from Kostant's theorem for~$i=7$ and two contributions from Kostant's theorem for~$i=5$, whose weights happen to be singular. The rest of the example proceeds along similar lines.

  \begin{table}[ht!]
    \centering
    \begin{tabular}{crrcrc}
      \toprule
      weight       & rank & degree & representation & dimension & sum of roots \\
      \midrule
      $(0, 0, 0)$  & $1$  & $0$    & $(0, 0, 0)$    & $1$       & $(0, 0, 0)$ \\
      \midrule
      $(0, 1, 0)$  & $1$  & $0$    & $(0, 1, 0)$    & $21$      & $(1, 2, 2)$ \\
      \addlinespace
      $(1, -1, 2)$ & $6$  &        &                &           & $(1, 1, 2)$ \\
      \midrule
      $(1, 0, 2)$  & $6$  & $0$    & $(1, 0, 2)$    & $189$     & $(2, 3, 4)$ \\
      \addlinespace
      $(0, -1, 4)$ & $5$  &        &                &           & $(1, 2, 4)$ \\
      $(0, 1, 0)$  & $1$  & $0$    & $(0, 1, 0)$    & $21$      & $(1, 2, 2)$ \\
      $(2, -1, 2)$ & $9$  &        &                &           & $(2, 2, 3)$ \\
      \midrule
      $(0, 0, 4)$  & $5$  & $0$    & $(0, 0, 4)$    & $294$     & $(2, 4, 6)$ \\
      $(0, 2, 0)$  & $1$  & $0$    & $(0, 2, 0)$    & $168$     & $(2, 4, 4)$ \\
      $(2, 0, 2)$  & $9$  & $0$    & $(2, 0, 2)$    & $616$     & $(3, 4, 5)$ \\
      \addlinespace
      $(1, 0, 2)$  & $6$  & $0$    & $(1, 0, 2)$    & $189$     & $(2, 3, 4)$ \\
      $(3, 0, 0)$  & $4$  & $0$    & $(3, 0, 0)$    & $77$      & $(3, 3, 3)$ \\
      $(1, -1, 4)$ & $10$ &        &                &           & $(2, 3, 5)$ \\
      \midrule
      $(1, 1, 2)$  & $6$  & $0$    & $(1, 1, 2)$    & $1617$    & $(3, 5, 6)$ \\
      $(3, 1, 0)$  & $4$  & $0$    & $(3, 1, 0)$    & $819$     & $(4, 5, 5)$ \\
      $(1, 0, 4)$  & $10$ & $0$    & $(1, 0, 4)$    & $1386$    & $(3, 5, 7)$ \\
      \addlinespace
      $(0, 0, 4)$  & $5$  & $0$    & $(0, 0, 4)$    & $294$     & $(2, 4, 6)$ \\
      $(2, 0, 2)$  & $9$  & $0$    & $(2, 0, 2)$    & $616$     & $(3, 4, 5)$ \\
      $(0, 2, 0)$  & $1$  & $0$    & $(0, 2, 0)$    & $168$     & $(2, 4, 4)$ \\
      \midrule
      $(0, 1, 4)$  & $5$  & $0$    & $(0, 1, 4)$    & $2310$    & $(3, 6, 8)$ \\
      $(2, 1, 2)$  & $9$  & $0$    & $(2, 1, 2)$    & $4550$    & $(4, 6, 7)$ \\
      $(0, 3, 0)$  & $1$  & $0$    & $(0, 3, 0)$    & $825$     & $(3, 6, 6)$ \\
      \addlinespace
      $(1, 1, 2)$  & $6$  & $0$    & $(1, 1, 2)$    & $1617$    & $(3, 5, 6)$ \\
      \midrule
      $(1, 2, 2)$  & $6$  & $0$    & $(1, 2, 2)$    & $7722$    & $(4, 7, 8)$ \\
      \addlinespace
      $(0, 3, 0)$  & $1$  & $0$    & $(0, 3, 0)$    & $825$     & $(3, 6, 6)$ \\
      \midrule
      $(0, 4, 0)$  & $1$  & $0$    & $(0, 4, 0)$    & $3003$    & $(4, 8, 8)$ \\
      \bottomrule
    \end{tabular}
    \caption{Associated graded for $\bigwedge^\bullet\tangent_{\OGr(2,7)}$}
    \label{table:wedge-B3-P2}
  \end{table}
\end{example}

\section{On the (possible) non-vanishing of the higher cohomologies}
\label{section:counterexample}
In this section we discuss what we know for generalised Grasmannians which are \emph{not} covered by \cref{theorem:vanishing}. The two main results are a proof of \cref{proposition:non-vanishing} and an elaboration of the caveat in \cref{remark:caveat} to \cref{conjecture:non-vanishing}. To conclude this section we explain the relationship between our vanishing results and Bott vanishing: our results give new cases in which Bott vanishing fails for generalised Grassmannians.

\subsection{Non-vanishing for \texorpdfstring{$\SGr(3,2n)$}{SGr(3,2n)}}
\label{subsection:non-vanishing}
Consider symplectic isotropic Grassmannians $\SGr(3,2n)$ with $n \geq 4$, which can be realized as the quotient of the symplectic group $\SP_{2n}$ with respect to the maximal parabolic subgroup attached to the third node of the Dynkin diagram $\mathrm{C}_n$, i.e.,~for
\begin{equation}
  \dynkin{C}{**x.**}.
\end{equation}

\paragraph{Setup}
Let $V$ be a $2n$-dimensional vector space endowed with a symplectic form $\omega$, and let $v_1, \ldots , v_{2n}$ be a basis of $V$ such that $\omega(v_i, v_{2n+1-i}) = 1$ for $1 \leq i \leq n$ and all other pairings between basis elements vanish. For $1 \leq k \leq n$ the symplectic isotropic Grassmannian $\SGr(k,V)=\SGr(k,2n)$ is the variety parametrising isotropic $k$-dimensional subspaces in~$V$. As any isotropic subspace is a subspace, we have a natural closed immersion
\begin{equation}
  \SGr(k,2n) \hookrightarrow \Gr(k,V)=\Gr(k,2n).
\end{equation}
The symplectic form $\omega$ gives rise to a global section $s_{\omega}$ of the vector bundle $\bigwedge^2 \mathcal{U}^\vee$ on $\Gr(k,2n)$, and the subvariety $\SGr(k,V)$ is the zero locus of $s_{\omega}$.

One can realise $\SGr(k,2n)$ and $\Gr(k,2n)$ as quotients
\begin{equation}
  \SGr(k,2n) = \SP_{2n} / P_{\omega} \quad \text{and} \quad \Gr(k,2n) = \GL_{2n} / P,
\end{equation}
where we have taken $P$ and $P_{\omega}$ to be the stabilisers of the ``standard'' isotropic~$k$\dash dimensional subspace spanned by the basis vectors $v_{2n}, v_{2n-1}, \ldots, v_{2n-k+1}$. Naturally, we have the inclusion of the parabolics $P_{\omega}  \subseteq P$.

Similarly, we have the embedding of the corresponding Levi subgroups
\begin{equation}
  L_{\omega} = \SP_{2n-2k} \times \GL_k  \subset L = \GL_{2n-k} \times \GL_k.
\end{equation}

In the setup above we have that the maximal torus $T$ is given by the diagonal matrices of the form $(t_1, \ldots t_n ,t_n^{-1}, \ldots , t_1^{-1})$. This way we identify the weight lattice of $\SP_{2n}$ with $\mathbb{Z}^n$ in such a way that the simple roots are
\begin{equation}
  \alpha_i = e_i - e_{i+1} \quad \text{for} \quad 1 \leq i \leq n-1, \quad \text{and} \quad \alpha_n = 2e_n,
\end{equation}
and the fundamental weights are
\begin{equation}
  \fundamental{i} = e_1 + \ldots + e_i \quad \text{for} \quad 1 \leq i \leq n,
\end{equation}
where $e_1, \ldots, e_n$ is the standard basis of $\mathbb{Z}^n$.

\begin{lemma}
  For~$n\geq 4$ the associated graded of~$\tangent_{\SGr(3,2n)}$ is given by
  \begin{equation}
    \label{equation:tangent-bundle-sgr-3-2n}
    \gr(\tangent_{\SGr(3,2n)}) = \bundle{2\fundamental{1}} \oplus \bundle{\fundamental{1} - \fundamental{3} + \fundamental{4}}.
  \end{equation}
  In the special case~$n=3$ the tangent bundle is irreducible and we have $\tangent_{\SGr(3,6)} = \bundle{2\fundamental{1}}$.
\end{lemma}

\begin{proof}
  Since the symplectic isotropic Grassmannian~$\SGr(3,2n)$ is a closed subvariety of the ordinary Grassmannian~$\Gr(3,2n)$ cut out by a regular section of~$\bigwedge^2 \mathcal{U}^\vee$, we have the short exact sequence
  \begin{equation}
    \label{equation:short-exact-sequence-sgr-3-2n}
    0 \to \tangent_{\SGr(3,2n)} \to i^* \tangent_{\Gr(3,2n)} \to i^* \bigwedge^2 \mathcal{U}^\vee \to 0
  \end{equation}
  of $\SP_{2n}$-equivariant bundles. By \eqref{equation:cohG-GP} it corresponds to a short exact sequence of representations of the parabolic subgroup $P$ of $\SP_{2n}$. By restricting these representations to the Levi~$L_{\omega}$ we will be able to deduce the desired description of the associated graded $\gr(\tangent_{\SGr(3,2n)})$. That is we need to determine the representations of $L_{\omega}$ corresponding to~$i^* \tangent_{\Gr(3,2n)}$ and~$i^* \bigwedge^2 \mathcal{U}^\vee$, and then we can just remove the contribution of the latter bundle from the former.

  The tangent bundle to $\Gr(3,2n)$ is~$\tangent_{\Gr(3,2n)} = \mathcal{Q} \otimes \mathcal{U}^\vee$ where
  \begin{equation}
    0\to\mathcal{U}\to V\otimes\mathcal{O}_{\Gr(3,2n)}\to\mathcal{Q}\to 0
  \end{equation}
  is the universal short exact sequence on~$\Gr(3,2n)$. Hence it is the bundle corresponding to the representation
  \begin{equation}
    V/U \otimes U^\vee,
  \end{equation}
  of $L=\GL_{2n-3} \times \GL_3$ where $U$ is the standard subspace described above. Here the representation $U^\vee$ is the dual of the standard representation of $\GL_3$, and the representation $V/U$ is the standard representation of $\GL_{2n-3}$.

  Now we need to consider these representations as representations of the Levi quotient~$L_{\omega} = \SP_{2n-6} \times \GL_3$. Restricting to $L_{\omega}$ we get
  \begin{equation}
    \left(W\oplus U^\vee \right) \otimes U^\vee,
  \end{equation}
  where $W$ is the standard representation of $\SP_{2n-6}$. We can further rewrite it as
  \begin{equation}
    \left(W \otimes U^\vee\right)\oplus\left(\bigwedge^2U\right)^\vee \oplus \left(\Sym^2U\right)^\vee.
  \end{equation}
  In terms of fundamental weights of $\SP_{2n}$ we get
  \begin{equation}
    \begin{array}{rll}
      W \otimes \left(U\right)^\vee &\leftrightarrow (1,0,0;1,0, \ldots, 0) &= \fundamental{1} - \fundamental{3} + \fundamental{4} \\
      \left(\bigwedge^2U\right)^\vee &\leftrightarrow (1,1,0;0, \ldots, 0) &= \fundamental{2} \\
     \left(\Sym^2U\right)^\vee &\leftrightarrow (2,0,0;0, \ldots, 0) &= 2\fundamental{1}
    \end{array}
  \end{equation}
  The summand with weight~$\fundamental{2}$ gets cancelled in \eqref{equation:short-exact-sequence-sgr-3-2n} and we obtain the claim.

  When~$n=3$ it suffices to observe that~$W=0$.
\end{proof}

\begin{lemma}
  \label{lemma:ss-wedge-2-sgr3-2n}
  For $n \geq 5$ the associated graded of $\bigwedge^2 \tangent_{\SGr(3,2n)}$ is given by
  \begin{equation}
    \label{equation:second-wedge-tangent-bundle-sgr-3-2n}
    \gr\left( \bigwedge^2 \tangent_{\SGr(3,2n)} \right) = \bundle{2\fundamental{1} + \fundamental{2}}
    \oplus \bundle{3\fundamental{1} - \fundamental{3} + \fundamental{4}}
    \oplus \bundle{\fundamental{1} + \fundamental{2} - \fundamental{3} + \omega_4}
    \oplus \bundle{\fundamental{2} - 2\fundamental{3} + 2\fundamental{4}}
    \oplus \bundle{2\fundamental{1} - \fundamental{3} + \fundamental{5}}
    \oplus \bundle{2\fundamental{1}}.
  \end{equation}
  For $n = 4$ the summand $\bundle{2\fundamental{1} - \fundamental{3} + \fundamental{5}}$ should be omitted. For $n=3$ we have  $\bigwedge^2 \tangent_{\SGr(3,6)} = \bundle{2\fundamental{1} + \fundamental{2}}$.
\end{lemma}

\begin{proof}
  Note that we have
  \begin{equation}
    \bigwedge^2 (\bundle{2\fundamental{1}} \oplus \bundle{\fundamental{1} - \fundamental{3} + \fundamental{4}}) =
    \bigwedge^2 \bundle{2\fundamental{1}}
    \oplus \left( \bundle{2\fundamental{1}} \otimes \bundle{\fundamental{1} - \fundamental{3} + \fundamental{4}} \right)
    \oplus \bigwedge^2 \bundle{\fundamental{1} - \fundamental{3} + \fundamental{4}}.
  \end{equation}
  From this we compute that
  \begin{equation}
    \begin{aligned}
      \bigwedge^2 \bundle{2\fundamental{1}}
      &\cong \bundle{2\fundamental{1} + \fundamental{2}} \\
      \bundle{2\fundamental{1}} \otimes \bundle{\fundamental{1} - \fundamental{3} + \fundamental{4}}
      &=\bundle{3\fundamental{1} - \fundamental{3} + \fundamental{4}}
      \oplus \bundle{\fundamental{1} + \fundamental{2} - \fundamental{3} + \omega_4} \\
      \bigwedge^2 \bundle{\fundamental{1} - \fundamental{3} + \fundamental{4}}
      &=\bundle{\fundamental{2} - 2\fundamental{3} + 2\fundamental{4}}
      \oplus \bundle{2\fundamental{1} - \fundamental{3} + \fundamental{5}}
      \oplus \bundle{2\fundamental{1}}
    \end{aligned}
  \end{equation}
  which proves the claim for~$n\geq 5$. For~$n=4$ (resp.~$n=3$) we discard the contributions involving~$\omega_5$ (resp.~$\omega_4$ and~$\omega_5$).
\end{proof}

\begin{proof}[Proof of \cref{proposition:non-vanishing}]
  Applying Borel--Weil--Bott (complemented with \cref{lemma:new-lemma-on-weights}~\cref{enumerate:weights-1} and \cref{enumerate:weights-5}) we see that the only summands of $\gr(\bigwedge^2 \tangent_{\SGr(3,2n)})$ with non-trivial cohomology are
  \begin{align}
    \HH^{\bullet}(\SGr(3,2n),\bundle{2\fundamental{1} + \fundamental{2}}) &\cong \representation{2\fundamental{1} + \fundamental{2}}{\SP_{2n}}[0] \\
    \HH^{\bullet}(\SGr(3,2n),\bundle{2\fundamental{1}}) &\cong \representation{2\fundamental{1}}{\SP_{2n}}[0] \\
    \HH^{\bullet}(\SGr(3,2n),\bundle{\fundamental{2} - 2\fundamental{3} + 2\fundamental{4}})&\cong \representation{\fundamental{4}}{\SP_{2n}} [-1]\label{equation:non-trivial-h1-sgr-3-2n}
  \end{align}
  where~$[-i]$ indicates the degree in which the cohomology lives.
  In particular, we see that in the spectral sequence obtained from the filtration on~$\bigwedge^2\tangent_{\SGr(3,2n)}$ no cancellation can happen and \eqref{equation:non-trivial-h1-sgr-3-2n} contributes non-trivially to~$\HHHH^3(\SGr(3,2n)$ in the term~$\HH^1(\SGr(3,2n),\bigwedge^2\tangent_{\SGr(3,2n)})$ of the Hochschild--Kostant--Rosenberg decomposition.
\end{proof}

\begin{remark}
  In the special case~$n=3$ we have that~$\SGr(3,6)$ is a cominuscule variety, in which case it is covered by \cref{theorem:vanishing,theorem:cominuscule-decomposition}. From \cref{lemma:ss-wedge-2-sgr3-2n} we obtain the isomorphism~$\bigwedge^2\tangent_{\SGr(3,6)}\cong\bundle{2\fundamental{1} + \fundamental{2}}$ which by the proof of \cref{proposition:non-vanishing} (only) has global sections.
\end{remark}

\begin{remark}
  \label{remark:SGr-3-8}
  Having shown that~$\HH^1(\SGr(3,2n),\bigwedge^2\tangent_{\SGr(3,2n)})\neq 0$ we can wonder what happens for higher exterior powers. Applying this method for higher exterior powers and~$n=4$ reveals that for~$\bigwedge^p\tangent_{\SGr(3,8)}$, with~$p=3,4,5,6$ there are summands in the associated graded which by Borel--Weil--Bott have an~$\HH^1$. But every representation that arises in this~$\HH^1$ also appears as the~$\HH^0$ of a different summand. Hence in the spectral sequence it is possible that these get cancelled, and we cannot conclude whether they are preserved in the abutment.

  In \cref{table:wedge-3-C4-P3} we have collected the summands and their cohomology as obtained from the Borel--Weil--Bott theorem. The Euler characteristic of the isotypical component associated to the highest weight~$2\omega_1+\omega_4$ is zero, hence it is not clear whether it is being cancelled or not in the spectral sequence. The same is true for~$\bigwedge^p\tangent_{\SGr(3,8)}$ with~$p=4,5,6$: it is not possible from the components in the spectral sequence to deduce (non-)vanishing.

\end{remark}

\begin{table}
  \centering
  \begin{tabular}{crrcrc}
    \toprule
    weight          & rank & degree & representation & dimension & sum of roots \\
    \midrule
    $(0, 0, -2, 3)$ & $4$  &        &                &           & $(1, 2, 3, 3)$ \\
    $(1, 1, -1, 1)$ & $16$ &        &                &           & $(2, 3, 3, 2)$ \\
    \addlinespace
    $(4, 0, 0, 0)$  & $15$ & $0$    & $(4, 0, 0, 0)$ & $330$     & $(4, 4, 4, 2)$ \\
    $(0, 2, 0, 0)$  & $6$  & $0$    & $(0, 2, 0, 0)$ & $308$     & $(2, 4, 4, 2)$ \\
    $(2, 1, 0, 0)$  & $15$ & $0$    & $(2, 1, 0, 0)$ & $594$     & $(3, 4, 4, 2)$ \\
    $(2, 1, -2, 2)$ & $45$ & $1$    & $(2, 0, 0, 1)$ & $1155$    & $(3, 4, 4, 3)$ \\
    $(1, 0, -1, 2)$ & $9$  &        &                &           & $(2, 3, 4, 3)$ \\
    \addlinespace
    $(1, 2, -1, 1)$ & $30$ &        &                &           & $(3, 5, 5, 3)$ \\
    $(3, 1, -1, 1)$ & $48$ &        &                &           & $(4, 5, 5, 3)$ \\
    $(2, 0, 0, 1)$  & $12$ & $0$    & $(2, 0, 0, 1)$ & $1155$    & $(3, 4, 5, 3)$ \\
    \addlinespace
    $(3, 0, 1, 0)$  & $10$ & $0$    & $(3, 0, 1, 0)$ & $3696$    & $(4, 5, 6, 3)$ \\
    $(0, 3, 0, 0)$  & $10$ & $0$    & $(0, 3, 0, 0)$ & $2184$    & $(3, 6, 6, 3)$ \\
    \bottomrule
  \end{tabular}
  \caption{Associated graded for $\bigwedge^{p}\tangent_{\SGr(3,8)}$ for $p=3,4,5,6$}
  \label{table:wedge-3-C4-P3}
\end{table}

\subsection{Potential vanishing for \texorpdfstring{$\OGr(n-1,2n+1)$}{OGr(n-1,2n+1)}}
\label{subsection:open-case-Bn-Pn-1}
In this section we discuss the caveat expressed in \cref{remark:caveat} by explaining the indeterminacy in our methods for the orthogonal Grassmannian~$\OGr(n-1,2n+1)$, in particular when~$n=4$. Using the method outlined below we can compute the sheaf cohomology of the associated graded of the equivariant vector bundle~$\bigwedge^p\tangent_{\OGr(3,9)}$ associated to the marked Dynkin diagram
\begin{equation}
  \dynkin[parabolic=4]{B}{4}.
\end{equation}
What happens in this case, and likewise for~$n=5,6,7,8$, is similar to the phenomenon described in \cref{remark:SGr-3-8} for~$\bigwedge^3\tangent_{\SGr(3,8)}$. We expect it is to continue for all~$\OGr(n-1,2n+1)$ with~$n\geq 4$, associated to the marked Dynkin diagram
\begin{equation}
  \dynkin[parabolic=8]{B}{}.
\end{equation}
This makes it impossible to conclude anything from the spectral sequence \eqref{equation:konno} without a better understanding of the differentials.

\paragraph{Potential vanishing}
We can now elaborate on the ``potential vanishing phenomenon'' in the smallest case, where all the possibly non-zero differentials are in fact surjections, which means that any contributions to higher cohomology get cancelled. In this case the phenomenon is restricted to degrees~0 and~1. For other generalised Grassmannians (where \cref{conjecture:non-vanishing} predicts non-vanishing) there exist examples where the associated graded has cohomology in higher degrees, making the analysis of the spectral sequence harder. The result in \cref{subsection:non-vanishing} is an instance where the differential is actually zero.

\begin{example}
  \label{example:OGr-3-9}
  Consider~$\bigwedge^2\tangent_{\OGr(3,9)}$. One can compute that its associated graded has~6~summands. These are collected in \cref{table:wedge-2-B4-P3}, and one can read off from the table that Konno's filtration is a~2\dash step filtration in this case. The relevant part of the~$\mathrm{E}_1$\dash page of the spectral sequence has the form
  \begin{equation}
    \begin{tikzcd}
      \HH^0(\OGr(3,9),\bundle{2\fundamental{4}})\oplus\HH^0(\OGr(3,9),\bundle{\fundamental{1}+\fundamental{3}}) \arrow[r, "\mathrm{d}_1"] & \HH^1(\OGr(3,9),\bundle{\fundamental{2}-2\fundamental{3}+\fundamental{4}}) \\
      & \HH^0(\OGr(3,9),\bundle{\fundamental{2}})
    \end{tikzcd}
  \end{equation}
  which after applying Borel--Weil--Bott becomes in terms of~$\mathfrak{so}_9$\dash representations
  \begin{equation}
    \begin{tikzcd}
      \representation{2\fundamental{4}}{\mathfrak{so}_9}\oplus\representation{\fundamental{1}+\fundamental{3}}{\mathfrak{so}_9} \arrow[r, "\mathrm{d}_1"] & \representation{2\fundamental{4}}{\mathfrak{so}_9} \\
      & \representation{\fundamental{2}}{\mathfrak{so}_9}.
    \end{tikzcd}
  \end{equation}
  This method does not tell whether the map~$\mathrm{d}_1$ is zero or not. If it is zero, i.e.,~the spectral sequence degenerates on the~$\mathrm{E}_1$\dash page, then
  \begin{equation}
    \begin{aligned}
      \HH^0(\OGr(3,9),\bigwedge^2\tangent_{\OGr(3,9)})&\cong\representation{2\fundamental{4}}{\mathfrak{so}_9}\oplus\representation{\fundamental{1}+\fundamental{3}}{\mathfrak{so}_9}\oplus\representation{\fundamental{2}}{\mathfrak{so}_9} \\
      \HH^1(\OGr(3,9),\bigwedge^2\tangent_{\OGr(3,9)})&\cong\representation{2\fundamental{4}}{\mathfrak{so}_9}
    \end{aligned}
  \end{equation}
  whereas if it is nonzero the spectral sequence degenerates on the~$\mathrm{E}_2$\dash page after cancelling two representations, and
  \begin{equation}
    \label{equation:hemelsoet}
    \begin{aligned}
      \HH^0(\OGr(3,9),\bigwedge^2\tangent_{\OGr(3,9)})&\cong\representation{\fundamental{1}+\fundamental{3}}{\mathfrak{so}_9}\oplus\representation{\fundamental{2}}{\mathfrak{so}_9} \\
      \HH^1(\OGr(3,9),\bigwedge^2\tangent_{\OGr(3,9)})&\cong0.
    \end{aligned}
  \end{equation}

  \begin{table}[ht!]
    \centering
    \begin{tabular}{crrcrc}
      \toprule
      weight          & rank & degree & representation & dimension & sum of roots \\
      \midrule
      $(0, 1, 0, 0)$  & $3$  & $0$    & $(0, 1, 0, 0)$ & $36$      & $(1, 2, 2, 2)$ \\
      $(2, 0, -1, 2)$ & $18$ &        &                &           & $(2, 2, 2, 3)$ \\
      $(0, 1, -2, 4)$ & $15$ & $1$    & $(0, 0, 0, 2)$ & $126$     & $(1, 2, 2, 4)$ \\
      \addlinespace
      $(0, 0, 0, 2)$  & $3$  & $0$    & $(0, 0, 0, 2)$ & $126$     & $(1, 2, 3, 4)$ \\
      $(1, 1, -1, 2)$ & $24$ &        &                &           & $(2, 3, 3, 4)$ \\
      $(1, 0, 1, 0)$  & $3$  & $0$    & $(1, 0, 1, 0)$ & $594$     & $(2, 3, 4, 4)$ \\
      \bottomrule
    \end{tabular}
    \caption{Associated graded for $\bigwedge^{2}\tangent_{\OGr(3,9)}$}
    \label{table:wedge-2-B4-P3}
  \end{table}
\end{example}

\begin{remark}
  We have been informed by Nicolas Hemelsoet that he has used the method from \cite{MR4187255} to check that $\representation{2\fundamental{4}}{\mathfrak{so}_9}$ does not appear in~$\HH^0(\OGr(3,9),\bigwedge^2\tangent_{\OGr(3,9)})$, hence the differential~$\mathrm{d}_1$ is in fact nonzero, and we are in situation \eqref{equation:hemelsoet}.
\end{remark}

Therefore, as expressed in \cref{remark:caveat}, it is not clear in the statement of \cref{conjecture:non-vanishing} whether to include or exclude the family~$\OGr(n-1,2n+1)$ for~$n\geq 4$.
%

\subsection{Remarks on Bott vanishing}
\label{subsection:bott-vanishing}
Finally we wish to give a brief overview of the relationship of the (non-)vanishing results in this paper and the notion of Bott vanishing.
This is the following vanishing property, which is rather strong as the discussion following the definition shows.

\begin{definition}
  A smooth projective variety~$X$ satisfies \emph{Bott vanishing} if
  \begin{equation}
    \HH^j(X,\Omega_X^i\otimes\mathcal{L})=0
  \end{equation}
  for all ample line bundles~$\mathcal{L}$, all~$i\geq 0$ and all~$j\geq 1$.
\end{definition}
In particular, a smooth projective \emph{Fano} variety (such as~$G/P$) satisfying Bott vanishing immediately satisfies the vanishing property for the Hochschild--Kostant--Rosenberg decomposition we set out to study for generalised Grassmannians.

It is known (due to Bott) that~$\mathbb{P}^n$ satisfies Bott vanishing. More generally toric varieties satisfy Bott vanishing, in which case it is called Danilov--Steenbrink--Bott vanishing. Hence Fano toric varieties are automatically Hochschild global in the sense of \eqref{equation:hochschild-global}, and a combinatorial description of the Hochschild cohomology can be obtained from \cite[Theorems~2.14(2) and~3.6(2)]{MR1900589}. Recently the first non-toric Fano variety, namely~$\Bl_4\mathbb{P}^2$, satisfying Bott vanishing was found by Totaro \cite[Theorem~2.1]{MR4082249}, and this was generalised by Torres in \cite{2003.10617v1}.

On the other hand it is expected (see \cite[Remark~2]{MR1464183}) that Bott vanishing does \emph{not} hold for any partial flag variety which is not~$\mathbb{P}^n$, hence the vanishing result in the cominuscule and (co)adjoint case in \cref{theorem:vanishing} cannot be the consequence of Bott vanishing for the ample line bundle~$\omega_{G/P}^\vee$.

In the case of full flag varieties, the failure of Bott vanishing is shown in \cite[Corollary~13]{MR2062394}.
The situation for generalised Grassmannians is far less well-understood:

\begin{itemize}
  \item In the cominuscule case explicit examples of the failure of Bott vanishing are given in \cite[\S4.3]{MR1464183}.
  \item In the (co)adjoint case the methods used in this paper might be useful in exhibiting examples of the failure of Bott vanishing for generalised Grassmannians, but we leave this for future work. The adjoint case in type~A is covered by \cite[\S4.2]{MR1464183}.
  \item Outside these cases, a positive answer to \cref{conjecture:non-vanishing} would give explicit examples of the failure of Bott vanishing for~$\mathcal{L}\cong\omega_{G/P}^\vee$.
\end{itemize}
Hence \cref{proposition:non-vanishing} provides the first example of the failure of Bott vanishing
for a generalised Grassmannian which is not cominuscule,
which corresponds to \cref{corollary:no-bott-vanishing-for-sgr-3-2n}.


\begin{proof}[Proof of \cref{corollary:no-bott-vanishing-for-sgr-3-2n}]
  Consider the (very) ample line bundle~$\mathcal{L}\coloneqq\omega_{\SGr(3,2n)}^\vee$. Then
  \begin{equation}
    \HH^1(\SGr(3,2n),\Omega_{\SGr(3,2n)}^{\dim\SGr(3,2n)-2}\otimes\omega_{\SGr(3,2n)}^\vee)
    \cong
    \HH^1(\SGr(3,2n),\bigwedge^2\tangent_{\SGr(3,2n)})\neq 0
  \end{equation}
  by \cref{proposition:non-vanishing}.
\end{proof}

\newpage

\appendix

\section{The description for the orthogonal Grassmannian \texorpdfstring{$\OGr(3,9)$}{OGr(3,9)}}
\label{appendix:B4-P3}
In this appendix we present the cohomology of the associated graded of~$\bigwedge^i\tangent_{\OGr(3,9)}$ for~$i=3,\ldots,6$, similar to the discussion in \cref{example:OGr-3-9} for~$\bigwedge^2\tangent_{\OGr(3,9)}$. From the coefficient of~$\simple{3}$ one can read off the position in the spectral sequence, and one concludes that it is unclear whether the spectral sequence degenerates on the~$\mathrm{E}_1$\dash page.

\input{B4-P3}

\clearpage

\section{The description for the exceptional coadjoint \texorpdfstring{$(\mathrm{F}_4,\simple{4})$}{(F4,P4)}}
\label{appendix:F4-P4}
Lacking a description of the Hochschild cohomology of a coadjoint (but not adjoint) generalised Grassmannian as in \cref{theorem:adjoint-decomposition} (for which vanishing is known by \cref{theorem:vanishing}), we give in this appendix the description of the Hochschild cohomology of the exceptional coadjoint variety~$(\mathrm{F}_4,\simple{4})$. This is the~15\dash dimensional variety of index~11 associated to the marked Dynkin diagram
\begin{equation}
  \dynkin[parabolic=8]{F}{4}.
\end{equation}

In \cref{table:wedge-0-F4-P4,table:wedge-1-F4-P4,table:wedge-2-F4-P4,table:wedge-3-F4-P4,table:wedge-4-F4-P4,table:wedge-5-F4-P4,table:wedge-6-F4-P4-1,table:wedge-6-F4-P4-2,table:wedge-7-F4-P4-1,table:wedge-7-F4-P4-2,table:wedge-8-F4-P4-1,table:wedge-8-F4-P4-2,table:wedge-9-F4-P4-1,table:wedge-9-F4-P4-2,table:wedge-10-F4-P4,table:wedge-11-F4-P4,table:wedge-12-F4-P4,table:wedge-13-F4-P4,table:wedge-14-F4-P4,table:wedge-15-F4-P4} we have collected this description. We only list the weight that describes the vector bundle as a quadruple of coefficients for the fundamental weights, the rank of the vector bundle, and if the weight is regular\footnote{This is precisely the case when the coefficient for~$\fundamental{4}$ is not~$-1$.} (in which case it is necessarily dominant by the vanishing) the dimension of the associated representation. Here the pieces are grouped according to the filtration of Konno, in terms of the coefficient of~$\simple{4}$.

\begin{remark}
  This leaves only~1 family of varieties for which vanishing is known, but a systematic description is not yet available. These are the symplectic Grassmannians~$\SGr(2,2n)$. The method used in \cref{subsection:hochschild-cohomology-adjoint} does not apply here, because~$[\mathfrak{n},\mathfrak{n}]$ is~3\dash dimensional.
\end{remark}

\input{F4-P4}

\clearpage

\printbibliography

\end{document}

%% file: B4-P3.tex
\begin{table}[H]\centering
\begin{tabular}{crrcrc}\toprule
weight & rank & degree & representation & dimension & sum of roots \\ \midrule
$\left(0, 0, 0, 2\right)$ & $3$ & $0$ & $\left(0, 0, 0, 2\right)$ & $126$ & $\left(1, 2, 3, 4\right)$ \\
$\left(1, 1, -2, 4\right)$ & $40$ & $1$ & $\left(1, 0, 0, 2\right)$ & $924$ & $\left(2, 3, 3, 5\right)$ \\
$\left(3, 0, 0, 0\right)$ & $10$ & $0$ & $\left(3, 0, 0, 0\right)$ & $156$ & $\left(3, 3, 3, 3\right)$ \\
$\left(0, 0, -2, 6\right)$ & $7$ &  &  &  & $\left(1, 2, 3, 6\right)$ \\
$\left(1, 1, -1, 2\right)$ & $24$ &  &  &  & $\left(2, 3, 3, 4\right)$ \\
\addlinespace
$\left(0, 2, 0, 0\right)$ & $6$ & $0$ & $\left(0, 2, 0, 0\right)$ & $495$ & $\left(2, 4, 4, 4\right)$ \\
$\left(1, 0, 1, 0\right)$ & $3$ & $0$ & $\left(1, 0, 1, 0\right)$ & $594$ & $\left(2, 3, 4, 4\right)$ \\
$\left(2, 1, -1, 2\right)$ & $45$ &  &  &  & $\left(3, 4, 4, 5\right)$ \\
$\left(1, 0, 0, 2\right)$ & $9$ & $0$ & $\left(1, 0, 0, 2\right)$ & $924$ & $\left(2, 3, 4, 5\right)$ \\
$\left(1, 0, -1, 4\right)$ & $15$ &  &  &  & $\left(2, 3, 4, 6\right)$ \\
$\left(0, 2, -2, 4\right)$ & $30$ & $1$ & $\left(0, 1, 0, 2\right)$ & $2772$ & $\left(2, 4, 4, 6\right)$ \\
\addlinespace
$\left(0, 1, 0, 2\right)$ & $9$ & $0$ & $\left(0, 1, 0, 2\right)$ & $2772$ & $\left(2, 4, 5, 6\right)$ \\
$\left(2, 0, 0, 2\right)$ & $18$ & $0$ & $\left(2, 0, 0, 2\right)$ & $3900$ & $\left(3, 4, 5, 6\right)$ \\
\addlinespace
$\left(0, 0, 2, 0\right)$ & $1$ & $0$ & $\left(0, 0, 2, 0\right)$ & $1980$ & $\left(2, 4, 6, 6\right)$ \\
\bottomrule\end{tabular}
\caption{Associated graded for $\bigwedge^{3}\tangent_{\OGr(3,9)}$}
\label{wedge-3-B4-P3}
\end{table}

\begin{table}[H]\centering
\begin{tabular}{crrcrc}\toprule
weight & rank & degree & representation & dimension & sum of roots \\ \midrule
$\left(2, 1, -1, 2\right)$ & $45$ &  &  &  & $\left(3, 4, 4, 5\right)$ \\
$\left(1, 0, -2, 6\right)$ & $21$ &  &  &  & $\left(2, 3, 4, 7\right)$ \\
$\left(1, 0, 0, 2\right)$ & $9$ & $0$ & $\left(1, 0, 0, 2\right)$ & $924$ & $\left(2, 3, 4, 5\right)$ \\
$\left(0, 2, 0, 0\right)$ & $6$ & $0$ & $\left(0, 2, 0, 0\right)$ & $495$ & $\left(2, 4, 4, 4\right)$ \\
$\left(0, 2, -2, 4\right)$ & $30$ & $1$ & $\left(0, 1, 0, 2\right)$ & $2772$ & $\left(2, 4, 4, 6\right)$ \\
$\left(1, 0, -1, 4\right)$ & $15$ &  &  &  & $\left(2, 3, 4, 6\right)$ \\
\addlinespace
$\left(0, 1, 0, 2\right)$ & $9$ & $0$ & $\left(0, 1, 0, 2\right)$ & $2772$ & $\left(2, 4, 5, 6\right)$ \\
$\left(1, 2, -2, 4\right)$ & $75$ & $1$ & $\left(1, 1, 0, 2\right)$ & $15444$ & $\left(3, 5, 5, 7\right)$ \\
$\left(2, 0, -1, 4\right)$ & $30$ &  &  &  & $\left(3, 4, 5, 7\right)$ \\
$\left(0, 1, -1, 4\right)$ & $15$ &  &  &  & $\left(2, 4, 5, 7\right)$ \\
$\left(3, 1, 0, 0\right)$ & $24$ & $0$ & $\left(3, 1, 0, 0\right)$ & $2772$ & $\left(4, 5, 5, 5\right)$ \\
$\left(2, 0, 1, 0\right)$ & $6$ & $0$ & $\left(2, 0, 1, 0\right)$ & $2457$ & $\left(3, 4, 5, 5\right)$ \\
$\left(0, 1, -2, 6\right)$ & $21$ & $1$ & $\left(0, 0, 0, 4\right)$ & $2772$ & $\left(2, 4, 5, 8\right)$ \\
$\left(1, 2, -1, 2\right)$ & $45$ &  &  &  & $\left(3, 5, 5, 6\right)$ \\
$\left(0, 1, 0, 2\right)$ & $9$ & $0$ & $\left(0, 1, 0, 2\right)$ & $2772$ & $\left(2, 4, 5, 6\right)$ \\
$\left(2, 0, 0, 2\right)$ & $18$ & $0$ & $\left(2, 0, 0, 2\right)$ & $3900$ & $\left(3, 4, 5, 6\right)$ \\
\addlinespace
$\left(0, 0, 2, 0\right)$ & $1$ & $0$ & $\left(0, 0, 2, 0\right)$ & $1980$ & $\left(2, 4, 6, 6\right)$ \\
$\left(1, 1, 1, 0\right)$ & $8$ & $0$ & $\left(1, 1, 1, 0\right)$ & $9009$ & $\left(3, 5, 6, 6\right)$ \\
$\left(3, 0, 0, 2\right)$ & $30$ & $0$ & $\left(3, 0, 0, 2\right)$ & $12375$ & $\left(4, 5, 6, 7\right)$ \\
$\left(1, 1, 0, 2\right)$ & $24$ & $0$ & $\left(1, 1, 0, 2\right)$ & $15444$ & $\left(3, 5, 6, 7\right)$ \\
$\left(1, 1, -1, 4\right)$ & $40$ &  &  &  & $\left(3, 5, 6, 8\right)$ \\
$\left(0, 0, 0, 4\right)$ & $5$ & $0$ & $\left(0, 0, 0, 4\right)$ & $2772$ & $\left(2, 4, 6, 8\right)$ \\
\addlinespace
$\left(1, 0, 1, 2\right)$ & $9$ & $0$ & $\left(1, 0, 1, 2\right)$ & $25740$ & $\left(3, 5, 7, 8\right)$ \\
\bottomrule\end{tabular}
\caption{Associated graded for $\bigwedge^{4}\tangent_{\OGr(3,9)}$}
\label{wedge-4-B4-P3}
\end{table}

\begin{table}[H]\centering
\begin{tabular}{crrcrc}\toprule
weight & rank & degree & representation & dimension & sum of roots \\ \midrule
$\left(0, 1, 0, 2\right)$ & $9$ & $0$ & $\left(0, 1, 0, 2\right)$ & $2772$ & $\left(2, 4, 5, 6\right)$ \\
$\left(0, 1, -1, 4\right)$ & $15$ &  &  &  & $\left(2, 4, 5, 7\right)$ \\
$\left(0, 1, -2, 6\right)$ & $21$ & $1$ & $\left(0, 0, 0, 4\right)$ & $2772$ & $\left(2, 4, 5, 8\right)$ \\
$\left(1, 2, -1, 2\right)$ & $45$ &  &  &  & $\left(3, 5, 5, 6\right)$ \\
$\left(2, 0, -1, 4\right)$ & $30$ &  &  &  & $\left(3, 4, 5, 7\right)$ \\
$\left(2, 0, 1, 0\right)$ & $6$ & $0$ & $\left(2, 0, 1, 0\right)$ & $2457$ & $\left(3, 4, 5, 5\right)$ \\
\addlinespace
$\left(1, 1, 0, 2\right)$ & $24$ & $0$ & $\left(1, 1, 0, 2\right)$ & $15444$ & $\left(3, 5, 6, 7\right)$ \\
$\left(2, 2, -1, 2\right)$ & $81$ &  &  &  & $\left(4, 6, 6, 7\right)$ \\
$\left(3, 0, 0, 2\right)$ & $30$ & $0$ & $\left(3, 0, 0, 2\right)$ & $12375$ & $\left(4, 5, 6, 7\right)$ \\
$\left(1, 1, -2, 6\right)$ & $56$ & $1$ & $\left(1, 0, 0, 4\right)$ & $18018$ & $\left(3, 5, 6, 9\right)$ \\
$\left(0, 0, -1, 6\right)$ & $7$ &  &  &  & $\left(2, 4, 6, 9\right)$ \\
$\left(0, 0, 1, 2\right)$ & $3$ & $0$ & $\left(0, 0, 1, 2\right)$ & $4158$ & $\left(2, 4, 6, 7\right)$ \\
$\left(1, 1, 0, 2\right)$ & $24$ & $0$ & $\left(1, 1, 0, 2\right)$ & $15444$ & $\left(3, 5, 6, 7\right)$ \\
$\left(0, 3, 0, 0\right)$ & $10$ & $0$ & $\left(0, 3, 0, 0\right)$ & $4004$ & $\left(3, 6, 6, 6\right)$ \\
$\left(1, 1, 1, 0\right)$ & $8$ & $0$ & $\left(1, 1, 1, 0\right)$ & $9009$ & $\left(3, 5, 6, 6\right)$ \\
$\left(1, 1, -1, 4\right)$ & $40$ &  &  &  & $\left(3, 5, 6, 8\right)$ \\
$\left(0, 3, -2, 4\right)$ & $50$ & $1$ & $\left(0, 2, 0, 2\right)$ & $27456$ & $\left(3, 6, 6, 8\right)$ \\
$\left(1, 1, -1, 4\right)$ & $40$ &  &  &  & $\left(3, 5, 6, 8\right)$ \\
$\left(0, 0, 0, 4\right)$ & $5$ & $0$ & $\left(0, 0, 0, 4\right)$ & $2772$ & $\left(2, 4, 6, 8\right)$ \\
\addlinespace
$\left(1, 0, 1, 2\right)$ & $9$ & $0$ & $\left(1, 0, 1, 2\right)$ & $25740$ & $\left(3, 5, 7, 8\right)$ \\
$\left(2, 1, -1, 4\right)$ & $75$ &  &  &  & $\left(4, 6, 7, 9\right)$ \\
$\left(1, 0, 0, 4\right)$ & $15$ & $0$ & $\left(1, 0, 0, 4\right)$ & $18018$ & $\left(3, 5, 7, 9\right)$ \\
$\left(0, 2, -1, 4\right)$ & $30$ &  &  &  & $\left(3, 6, 7, 9\right)$ \\
$\left(4, 0, 1, 0\right)$ & $15$ & $0$ & $\left(4, 0, 1, 0\right)$ & $20196$ & $\left(5, 6, 7, 7\right)$ \\
$\left(2, 1, 1, 0\right)$ & $15$ & $0$ & $\left(2, 1, 1, 0\right)$ & $31500$ & $\left(4, 6, 7, 7\right)$ \\
$\left(1, 0, -1, 6\right)$ & $21$ &  &  &  & $\left(3, 5, 7, 10\right)$ \\
$\left(2, 1, 0, 2\right)$ & $45$ & $0$ & $\left(2, 1, 0, 2\right)$ & $54675$ & $\left(4, 6, 7, 8\right)$ \\
$\left(0, 2, 0, 2\right)$ & $18$ & $0$ & $\left(0, 2, 0, 2\right)$ & $27456$ & $\left(3, 6, 7, 8\right)$ \\
$\left(1, 0, 1, 2\right)$ & $9$ & $0$ & $\left(1, 0, 1, 2\right)$ & $25740$ & $\left(3, 5, 7, 8\right)$ \\
\addlinespace
$\left(0, 1, 2, 0\right)$ & $3$ & $0$ & $\left(0, 1, 2, 0\right)$ & $27027$ & $\left(3, 6, 8, 8\right)$ \\
$\left(2, 0, 1, 2\right)$ & $18$ & $0$ & $\left(2, 0, 1, 2\right)$ & $96228$ & $\left(4, 6, 8, 9\right)$ \\
$\left(0, 1, 0, 4\right)$ & $15$ & $0$ & $\left(0, 1, 0, 4\right)$ & $46332$ & $\left(3, 6, 8, 10\right)$ \\
\bottomrule\end{tabular}
\caption{Associated graded for $\bigwedge^{5}\tangent_{\OGr(3,9)}$}
\label{wedge-5-B4-P3}
\end{table}

\begin{table}[H]\centering
\begin{tabular}{crrcrc}\toprule
weight & rank & degree & representation & dimension & sum of roots \\ \midrule
$\left(0, 3, 0, 0\right)$ & $10$ & $0$ & $\left(0, 3, 0, 0\right)$ & $4004$ & $\left(3, 6, 6, 6\right)$ \\
$\left(1, 1, 0, 2\right)$ & $24$ & $0$ & $\left(1, 1, 0, 2\right)$ & $15444$ & $\left(3, 5, 6, 7\right)$ \\
$\left(0, 0, 1, 2\right)$ & $3$ & $0$ & $\left(0, 0, 1, 2\right)$ & $4158$ & $\left(2, 4, 6, 7\right)$ \\
$\left(0, 0, -1, 6\right)$ & $7$ &  &  &  & $\left(2, 4, 6, 9\right)$ \\
$\left(1, 1, -1, 4\right)$ & $40$ &  &  &  & $\left(3, 5, 6, 8\right)$ \\
\addlinespace
$\left(0, 2, 0, 2\right)$ & $18$ & $0$ & $\left(0, 2, 0, 2\right)$ & $27456$ & $\left(3, 6, 7, 8\right)$ \\
$\left(1, 0, 1, 2\right)$ & $9$ & $0$ & $\left(1, 0, 1, 2\right)$ & $25740$ & $\left(3, 5, 7, 8\right)$ \\
$\left(1, 0, 0, 4\right)$ & $15$ & $0$ & $\left(1, 0, 0, 4\right)$ & $18018$ & $\left(3, 5, 7, 9\right)$ \\
$\left(0, 2, -1, 4\right)$ & $30$ &  &  &  & $\left(3, 6, 7, 9\right)$ \\
$\left(0, 2, -2, 6\right)$ & $42$ & $1$ & $\left(0, 1, 0, 4\right)$ & $46332$ & $\left(3, 6, 7, 10\right)$ \\
$\left(1, 0, -1, 6\right)$ & $21$ &  &  &  & $\left(3, 5, 7, 10\right)$ \\
$\left(2, 1, 0, 2\right)$ & $45$ & $0$ & $\left(2, 1, 0, 2\right)$ & $54675$ & $\left(4, 6, 7, 8\right)$ \\
$\left(1, 3, -1, 2\right)$ & $72$ &  &  &  & $\left(4, 7, 7, 8\right)$ \\
$\left(0, 2, 0, 2\right)$ & $18$ & $0$ & $\left(0, 2, 0, 2\right)$ & $27456$ & $\left(3, 6, 7, 8\right)$ \\
$\left(2, 1, -1, 4\right)$ & $75$ &  &  &  & $\left(4, 6, 7, 9\right)$ \\
$\left(1, 0, 0, 4\right)$ & $15$ & $0$ & $\left(1, 0, 0, 4\right)$ & $18018$ & $\left(3, 5, 7, 9\right)$ \\
$\left(2, 1, 1, 0\right)$ & $15$ & $0$ & $\left(2, 1, 1, 0\right)$ & $31500$ & $\left(4, 6, 7, 7\right)$ \\
$\left(1, 0, 2, 0\right)$ & $3$ & $0$ & $\left(1, 0, 2, 0\right)$ & $12012$ & $\left(3, 5, 7, 7\right)$ \\
\addlinespace
$\left(1, 2, 0, 2\right)$ & $45$ & $0$ & $\left(1, 2, 0, 2\right)$ & $128700$ & $\left(4, 7, 8, 9\right)$ \\
$\left(3, 1, 0, 2\right)$ & $72$ & $0$ & $\left(3, 1, 0, 2\right)$ & $153153$ & $\left(5, 7, 8, 9\right)$ \\
$\left(2, 0, 1, 2\right)$ & $18$ & $0$ & $\left(2, 0, 1, 2\right)$ & $96228$ & $\left(4, 6, 8, 9\right)$ \\
$\left(0, 1, -1, 6\right)$ & $21$ &  &  &  & $\left(3, 6, 8, 11\right)$ \\
$\left(2, 0, -1, 6\right)$ & $42$ &  &  &  & $\left(4, 6, 8, 11\right)$ \\
$\left(0, 1, 1, 2\right)$ & $9$ & $0$ & $\left(0, 1, 1, 2\right)$ & $60060$ & $\left(3, 6, 8, 9\right)$ \\
$\left(2, 0, 1, 2\right)$ & $18$ & $0$ & $\left(2, 0, 1, 2\right)$ & $96228$ & $\left(4, 6, 8, 9\right)$ \\
$\left(0, 1, 2, 0\right)$ & $3$ & $0$ & $\left(0, 1, 2, 0\right)$ & $27027$ & $\left(3, 6, 8, 8\right)$ \\
$\left(1, 2, 1, 0\right)$ & $15$ & $0$ & $\left(1, 2, 1, 0\right)$ & $71500$ & $\left(4, 7, 8, 8\right)$ \\
$\left(1, 2, -1, 4\right)$ & $75$ &  &  &  & $\left(4, 7, 8, 10\right)$ \\
$\left(0, 1, 0, 4\right)$ & $15$ & $0$ & $\left(0, 1, 0, 4\right)$ & $46332$ & $\left(3, 6, 8, 10\right)$ \\
$\left(2, 0, 0, 4\right)$ & $30$ & $0$ & $\left(2, 0, 0, 4\right)$ & $69300$ & $\left(4, 6, 8, 10\right)$ \\
$\left(0, 1, 0, 4\right)$ & $15$ & $0$ & $\left(0, 1, 0, 4\right)$ & $46332$ & $\left(3, 6, 8, 10\right)$ \\
\addlinespace
$\left(0, 0, 2, 2\right)$ & $3$ & $0$ & $\left(0, 0, 2, 2\right)$ & $56628$ & $\left(3, 6, 9, 10\right)$ \\
$\left(1, 1, 0, 4\right)$ & $40$ & $0$ & $\left(1, 1, 0, 4\right)$ & $235950$ & $\left(4, 7, 9, 11\right)$ \\
$\left(3, 0, 2, 0\right)$ & $10$ & $0$ & $\left(3, 0, 2, 0\right)$ & $127296$ & $\left(5, 7, 9, 9\right)$ \\
$\left(0, 0, 0, 6\right)$ & $7$ & $0$ & $\left(0, 0, 0, 6\right)$ & $28314$ & $\left(3, 6, 9, 12\right)$ \\
$\left(1, 1, 1, 2\right)$ & $24$ & $0$ & $\left(1, 1, 1, 2\right)$ & $297297$ & $\left(4, 7, 9, 10\right)$ \\
\bottomrule\end{tabular}
\caption{Associated graded for $\bigwedge^{6}\tangent_{\OGr(3,9)}$}
\label{wedge-6-B4-P3}
\end{table}

%% file: F4-P4.tex
\begin{table}
\centering
\begin{tabular}{crrcrcr}\toprule
weight         & rank & degree & representation & dimension & sum of roots \\ \midrule
$(0, 0, 0, 0)$ & $1$  & $0$    & $(0, 0, 0, 0)$ & $1$       & $(0, 0, 0, 0)$ \\
\bottomrule\end{tabular}
\caption{Associated graded for $\bigwedge^{0}\tangent_{G/P}$ for $(\mathrm{F}_4,\alpha_4)$}
\label{table:wedge-0-F4-P4}
\end{table}

\begin{table}
\centering
\begin{tabular}{crrcrcr}\toprule
weight          & rank & degree & representation & dimension & sum of roots \\ \midrule
$(0, 0, 1, -1)$ & $8$  &        &                &           & $(1, 2, 3, 1)$ \\
\addlinespace
$(1, 0, 0, 0)$  & $7$  & $0$    & $(1, 0, 0, 0)$ & $52$      & $(2, 3, 4, 2)$ \\
\bottomrule\end{tabular}
\caption{Associated graded for $\bigwedge^{1}\tangent_{G/P}$ for $(\mathrm{F}_4,\alpha_4)$}
\label{table:wedge-1-F4-P4}
\end{table}

\begin{table}
\centering
\begin{tabular}{crrcrcr}\toprule
weight          & rank & degree & representation & dimension & sum of roots \\ \midrule
$(1, 0, 0, 0)$  & $7$  & $0$    & $(1, 0, 0, 0)$ & $52$      & $(2, 3, 4, 2)$ \\
$(0, 1, 0, -1)$ & $21$ &        &                &           & $(2, 4, 5, 2)$ \\
\addlinespace
$(1, 0, 1, -1)$ & $48$ &        &                &           & $(3, 5, 7, 3)$ \\
$(0, 0, 1, 0)$  & $8$  & $0$    & $(0, 0, 1, 0)$ & $273$     & $(2, 4, 6, 3)$ \\
\addlinespace
$(0, 1, 0, 0)$  & $21$ & $0$    & $(0, 1, 0, 0)$ & $1274$    & $(3, 6, 8, 4)$ \\
\bottomrule\end{tabular}
\caption{Associated graded for $\bigwedge^{2}\tangent_{G/P}$ for $(\mathrm{F}_4,\alpha_4)$}
\label{table:wedge-2-F4-P4}
\end{table}

\begin{table}
\centering
\begin{tabular}{crrcrcr}\toprule
weight          & rank  & degree & representation & dimension & sum of roots \\ \midrule
$(1, 0, 1, -1)$ & $48$  &        &                &           & $(3, 5, 7, 3)$ \\
$(0, 0, 1, 0)$  & $8$   & $0$    & $(0, 0, 1, 0)$ & $273$     & $(2, 4, 6, 3)$ \\
\addlinespace
$(0, 0, 0, 2)$  & $1$   & $0$    & $(0, 0, 0, 2)$ & $324$     & $(2, 4, 6, 4)$ \\
$(2, 0, 0, 0)$  & $27$  & $0$    & $(2, 0, 0, 0)$ & $1053$    & $(4, 6, 8, 4)$ \\
$(0, 1, 0, 0)$  & $21$  & $0$    & $(0, 1, 0, 0)$ & $1274$    & $(3, 6, 8, 4)$ \\
$(1, 1, 0, -1)$ & $105$ &        &                &           & $(4, 7, 9, 4)$ \\
$(1, 0, 0, 1)$  & $7$   & $0$    & $(1, 0, 0, 1)$ & $1053$    & $(3, 5, 7, 4)$ \\
$(0, 0, 2, -1)$ & $35$  &        &                &           & $(3, 6, 9, 4)$ \\
\addlinespace
$(0, 0, 1, 1)$  & $8$   & $0$    & $(0, 0, 1, 1)$ & $4096$    & $(3, 6, 9, 5)$ \\
$(1, 0, 1, 0)$  & $48$  & $0$    & $(1, 0, 1, 0)$ & $8424$    & $(4, 7, 10, 5)$ \\
$(0, 1, 1, -1)$ & $112$ &        &                &           & $(4, 8, 11, 5)$ \\
\addlinespace
$(0, 0, 2, 0)$  & $35$  & $0$    & $(0, 0, 2, 0)$ & $19448$   & $(4, 8, 12, 6)$ \\
\bottomrule\end{tabular}
\caption{Associated graded for $\bigwedge^{3}\tangent_{G/P}$ for $(\mathrm{F}_4,\alpha_4)$}
\label{table:wedge-3-F4-P4}
\end{table}

\begin{table}
\centering
\begin{tabular}{crrcrcr}\toprule
weight          & rank  & degree & representation & dimension & sum of roots \\ \midrule
$(0, 0, 2, -1)$ & $35$  &        &                &           & $(3, 6, 9, 4)$ \\
$(1, 0, 0, 1)$  & $7$   & $0$    & $(1, 0, 0, 1)$ & $1053$    & $(3, 5, 7, 4)$ \\
$(0, 0, 0, 2)$  & $1$   & $0$    & $(0, 0, 0, 2)$ & $324$     & $(2, 4, 6, 4)$ \\
$(2, 0, 0, 0)$  & $27$  & $0$    & $(2, 0, 0, 0)$ & $1053$    & $(4, 6, 8, 4)$ \\
\addlinespace
$(2, 0, 1, -1)$ & $168$ &        &                &           & $(5, 8, 11, 5)$ \\
$(0, 1, 1, -1)$ & $112$ &        &                &           & $(4, 8, 11, 5)$ \\
$(1, 0, 1, 0)$  & $48$  & $0$    & $(1, 0, 1, 0)$ & $8424$    & $(4, 7, 10, 5)$ \\
$(0, 0, 1, 1)$  & $8$   & $0$    & $(0, 0, 1, 1)$ & $4096$    & $(3, 6, 9, 5)$ \\
$(1, 0, 1, 0)$  & $48$  & $0$    & $(1, 0, 1, 0)$ & $8424$    & $(4, 7, 10, 5)$ \\
$(0, 0, 1, 1)$  & $8$   & $0$    & $(0, 0, 1, 1)$ & $4096$    & $(3, 6, 9, 5)$ \\
\addlinespace
$(1, 1, 0, 0)$  & $105$ & $0$    & $(1, 1, 0, 0)$ & $29172$   & $(5, 9, 12, 6)$ \\
$(1, 0, 0, 2)$  & $7$   & $0$    & $(1, 0, 0, 2)$ & $10829$   & $(4, 7, 10, 6)$ \\
$(0, 0, 2, 0)$  & $35$  & $0$    & $(0, 0, 2, 0)$ & $19448$   & $(4, 8, 12, 6)$ \\
$(0, 0, 0, 3)$  & $1$   & $0$    & $(0, 0, 0, 3)$ & $2652$    & $(3, 6, 9, 6)$ \\
$(0, 0, 2, 0)$  & $35$  & $0$    & $(0, 0, 2, 0)$ & $19448$   & $(4, 8, 12, 6)$ \\
$(0, 2, 0, -1)$ & $168$ &        &                &           & $(5, 10, 13, 6)$ \\
$(2, 0, 0, 1)$  & $27$  & $0$    & $(2, 0, 0, 1)$ & $17901$   & $(5, 8, 11, 6)$ \\
$(1, 0, 2, -1)$ & $189$ &        &                &           & $(5, 9, 13, 6)$ \\
$(0, 1, 0, 1)$  & $21$  & $0$    & $(0, 1, 0, 1)$ & $19278$   & $(4, 8, 11, 6)$ \\
\addlinespace
$(0, 1, 1, 0)$  & $112$ & $0$    & $(0, 1, 1, 0)$ & $107406$  & $(5, 10, 14, 7)$ \\
$(1, 0, 1, 1)$  & $48$  & $0$    & $(1, 0, 1, 1)$ & $106496$  & $(5, 9, 13, 7)$ \\
$(0, 0, 1, 2)$  & $8$   & $0$    & $(0, 0, 1, 2)$ & $34749$   & $(4, 8, 12, 7)$ \\
$(0, 0, 3, -1)$ & $112$ &        &                &           & $(5, 10, 15, 7)$ \\
\addlinespace
$(0, 0, 2, 1)$  & $35$  & $0$    & $(0, 0, 2, 1)$ & $205751$  & $(5, 10, 15, 8)$ \\
\bottomrule\end{tabular}
\caption{Associated graded for $\bigwedge^{4}\tangent_{G/P}$ for $(\mathrm{F}_4,\alpha_4)$}
\label{table:wedge-4-F4-P4}
\end{table}

\begin{table}
\centering
\begin{tabular}{crrcrcr}\toprule
weight          & rank  & degree & representation & dimension & sum of roots \\ \midrule
$(1, 0, 1, 0)$  & $48$  & $0$    & $(1, 0, 1, 0)$ & $8424$    & $(4, 7, 10, 5)$ \\
$(0, 0, 1, 1)$  & $8$   & $0$    & $(0, 0, 1, 1)$ & $4096$    & $(3, 6, 9, 5)$ \\
\addlinespace
$(0, 0, 2, 0)$  & $35$  & $0$    & $(0, 0, 2, 0)$ & $19448$   & $(4, 8, 12, 6)$ \\
$(1, 0, 2, -1)$ & $189$ &        &                &           & $(5, 9, 13, 6)$ \\
$(0, 1, 0, 1)$  & $21$  & $0$    & $(0, 1, 0, 1)$ & $19278$   & $(4, 8, 11, 6)$ \\
$(0, 0, 0, 3)$  & $1$   & $0$    & $(0, 0, 0, 3)$ & $2652$    & $(3, 6, 9, 6)$ \\
$(2, 0, 0, 1)$  & $27$  & $0$    & $(2, 0, 0, 1)$ & $17901$   & $(5, 8, 11, 6)$ \\
$(0, 1, 0, 1)$  & $21$  & $0$    & $(0, 1, 0, 1)$ & $19278$   & $(4, 8, 11, 6)$ \\
$(1, 0, 0, 2)$  & $7$   & $0$    & $(1, 0, 0, 2)$ & $10829$   & $(4, 7, 10, 6)$ \\
$(1, 0, 0, 2)$  & $7$   & $0$    & $(1, 0, 0, 2)$ & $10829$   & $(4, 7, 10, 6)$ \\
$(1, 1, 0, 0)$  & $105$ & $0$    & $(1, 1, 0, 0)$ & $29172$   & $(5, 9, 12, 6)$ \\
$(3, 0, 0, 0)$  & $77$  & $0$    & $(3, 0, 0, 0)$ & $12376$   & $(6, 9, 12, 6)$ \\
\addlinespace
$(0, 1, 1, 0)$  & $112$ & $0$    & $(0, 1, 1, 0)$ & $107406$  & $(5, 10, 14, 7)$ \\
$(0, 0, 3, -1)$ & $112$ &        &                &           & $(5, 10, 15, 7)$ \\
$(1, 0, 1, 1)$  & $48$  & $0$    & $(1, 0, 1, 1)$ & $106496$  & $(5, 9, 13, 7)$ \\
$(1, 0, 1, 1)$  & $48$  & $0$    & $(1, 0, 1, 1)$ & $106496$  & $(5, 9, 13, 7)$ \\
$(1, 1, 1, -1)$ & $512$ &        &                &           & $(6, 11, 15, 7)$ \\
$(2, 0, 1, 0)$  & $168$ & $0$    & $(2, 0, 1, 0)$ & $119119$  & $(6, 10, 14, 7)$ \\
$(0, 0, 1, 2)$  & $8$   & $0$    & $(0, 0, 1, 2)$ & $34749$   & $(4, 8, 12, 7)$ \\
$(0, 0, 1, 2)$  & $8$   & $0$    & $(0, 0, 1, 2)$ & $34749$   & $(4, 8, 12, 7)$ \\
$(1, 0, 1, 1)$  & $48$  & $0$    & $(1, 0, 1, 1)$ & $106496$  & $(5, 9, 13, 7)$ \\
$(0, 1, 1, 0)$  & $112$ & $0$    & $(0, 1, 1, 0)$ & $107406$  & $(5, 10, 14, 7)$ \\
\addlinespace
$(0, 0, 2, 1)$  & $35$  & $0$    & $(0, 0, 2, 1)$ & $205751$  & $(5, 10, 15, 8)$ \\
$(1, 0, 2, 0)$  & $189$ & $0$    & $(1, 0, 2, 0)$ & $420147$  & $(6, 11, 16, 8)$ \\
$(0, 1, 0, 2)$  & $21$  & $0$    & $(0, 1, 0, 2)$ & $160056$  & $(5, 10, 14, 8)$ \\
$(0, 1, 2, -1)$ & $378$ &        &                &           & $(6, 12, 17, 8)$ \\
$(1, 1, 0, 1)$  & $105$ & $0$    & $(1, 1, 0, 1)$ & $379848$  & $(6, 11, 15, 8)$ \\
$(0, 0, 2, 1)$  & $35$  & $0$    & $(0, 0, 2, 1)$ & $205751$  & $(5, 10, 15, 8)$ \\
$(1, 0, 0, 3)$  & $7$   & $0$    & $(1, 0, 0, 3)$ & $76076$   & $(5, 9, 13, 8)$ \\
$(1, 0, 2, 0)$  & $189$ & $0$    & $(1, 0, 2, 0)$ & $420147$  & $(6, 11, 16, 8)$ \\
$(0, 1, 0, 2)$  & $21$  & $0$    & $(0, 1, 0, 2)$ & $160056$  & $(5, 10, 14, 8)$ \\
\addlinespace
$(0, 1, 1, 1)$  & $112$ & $0$    & $(0, 1, 1, 1)$ & $1118208$ & $(6, 12, 17, 9)$ \\
$(1, 0, 1, 2)$  & $48$  & $0$    & $(1, 0, 1, 2)$ & $787644$  & $(6, 11, 16, 9)$ \\
$(0, 0, 1, 3)$  & $8$   & $0$    & $(0, 0, 1, 3)$ & $212992$  & $(5, 10, 15, 9)$ \\
$(0, 0, 3, 0)$  & $112$ & $0$    & $(0, 0, 3, 0)$ & $629356$  & $(6, 12, 18, 9)$ \\
\addlinespace
$(0, 1, 0, 3)$  & $21$  & $0$    & $(0, 1, 0, 3)$ & $952952$  & $(6, 12, 17, 10)$ \\
\bottomrule\end{tabular}
\caption{Associated graded for $\bigwedge^{5}\tangent_{G/P}$ for $(\mathrm{F}_4,\alpha_4)$}
\label{table:wedge-5-F4-P4}
\end{table}

\begin{table}
\centering
\begin{tabular}{crrcrcr}\toprule
weight          & rank  & degree & representation & dimension & sum of roots \\ \midrule
$(1, 0, 0, 2)$  & $7$   & $0$    & $(1, 0, 0, 2)$ & $10829$   & $(4, 7, 10, 6)$ \\
$(0, 1, 0, 1)$  & $21$  & $0$    & $(0, 1, 0, 1)$ & $19278$   & $(4, 8, 11, 6)$ \\
\addlinespace
$(2, 0, 1, 0)$  & $168$ & $0$    & $(2, 0, 1, 0)$ & $119119$  & $(6, 10, 14, 7)$ \\
$(0, 1, 1, 0)$  & $112$ & $0$    & $(0, 1, 1, 0)$ & $107406$  & $(5, 10, 14, 7)$ \\
$(1, 0, 1, 1)$  & $48$  & $0$    & $(1, 0, 1, 1)$ & $106496$  & $(5, 9, 13, 7)$ \\
$(0, 0, 1, 2)$  & $8$   & $0$    & $(0, 0, 1, 2)$ & $34749$   & $(4, 8, 12, 7)$ \\
$(1, 0, 1, 1)$  & $48$  & $0$    & $(1, 0, 1, 1)$ & $106496$  & $(5, 9, 13, 7)$ \\
$(0, 0, 1, 2)$  & $8$   & $0$    & $(0, 0, 1, 2)$ & $34749$   & $(4, 8, 12, 7)$ \\
\addlinespace
$(0, 1, 2, -1)$ & $378$ &        &                &           & $(6, 12, 17, 8)$ \\
$(1, 1, 0, 1)$  & $105$ & $0$    & $(1, 1, 0, 1)$ & $379848$  & $(6, 11, 15, 8)$ \\
$(0, 0, 2, 1)$  & $35$  & $0$    & $(0, 0, 2, 1)$ & $205751$  & $(5, 10, 15, 8)$ \\
$(1, 0, 0, 3)$  & $7$   & $0$    & $(1, 0, 0, 3)$ & $76076$   & $(5, 9, 13, 8)$ \\
$(1, 0, 2, 0)$  & $189$ & $0$    & $(1, 0, 2, 0)$ & $420147$  & $(6, 11, 16, 8)$ \\
$(0, 1, 0, 2)$  & $21$  & $0$    & $(0, 1, 0, 2)$ & $160056$  & $(5, 10, 14, 8)$ \\
$(1, 1, 0, 1)$  & $105$ & $0$    & $(1, 1, 0, 1)$ & $379848$  & $(6, 11, 15, 8)$ \\
$(1, 0, 0, 3)$  & $7$   & $0$    & $(1, 0, 0, 3)$ & $76076$   & $(5, 9, 13, 8)$ \\
$(0, 0, 2, 1)$  & $35$  & $0$    & $(0, 0, 2, 1)$ & $205751$  & $(5, 10, 15, 8)$ \\
$(0, 1, 0, 2)$  & $21$  & $0$    & $(0, 1, 0, 2)$ & $160056$  & $(5, 10, 14, 8)$ \\
$(1, 0, 2, 0)$  & $189$ & $0$    & $(1, 0, 2, 0)$ & $420147$  & $(6, 11, 16, 8)$ \\
$(2, 1, 0, 0)$  & $330$ & $0$    & $(2, 1, 0, 0)$ & $340119$  & $(7, 12, 16, 8)$ \\
$(2, 0, 0, 2)$  & $27$  & $0$    & $(2, 0, 0, 2)$ & $160056$  & $(6, 10, 14, 8)$ \\
$(0, 1, 0, 2)$  & $21$  & $0$    & $(0, 1, 0, 2)$ & $160056$  & $(5, 10, 14, 8)$ \\
\addlinespace
$(0, 1, 1, 1)$  & $112$ & $0$    & $(0, 1, 1, 1)$ & $1118208$ & $(6, 12, 17, 9)$ \\
$(0, 1, 1, 1)$  & $112$ & $0$    & $(0, 1, 1, 1)$ & $1118208$ & $(6, 12, 17, 9)$ \\
$(0, 0, 3, 0)$  & $112$ & $0$    & $(0, 0, 3, 0)$ & $629356$  & $(6, 12, 18, 9)$ \\
$(1, 0, 1, 2)$  & $48$  & $0$    & $(1, 0, 1, 2)$ & $787644$  & $(6, 11, 16, 9)$ \\
$(1, 0, 1, 2)$  & $48$  & $0$    & $(1, 0, 1, 2)$ & $787644$  & $(6, 11, 16, 9)$ \\
$(1, 1, 1, 0)$  & $512$ & $0$    & $(1, 1, 1, 0)$ & $1801371$ & $(7, 13, 18, 9)$ \\
$(2, 0, 1, 1)$  & $168$ & $0$    & $(2, 0, 1, 1)$ & $1327104$ & $(7, 12, 17, 9)$ \\
$(1, 0, 3, -1)$ & $560$ &        &                &           & $(7, 13, 19, 9)$ \\
$(0, 0, 1, 3)$  & $8$   & $0$    & $(0, 0, 1, 3)$ & $212992$  & $(5, 10, 15, 9)$ \\
$(0, 1, 1, 1)$  & $112$ & $0$    & $(0, 1, 1, 1)$ & $1118208$ & $(6, 12, 17, 9)$ \\
$(1, 0, 1, 2)$  & $48$  & $0$    & $(1, 0, 1, 2)$ & $787644$  & $(6, 11, 16, 9)$ \\
$(0, 0, 1, 3)$  & $8$   & $0$    & $(0, 0, 1, 3)$ & $212992$  & $(5, 10, 15, 9)$ \\
$(0, 0, 3, 0)$  & $112$ & $0$    & $(0, 0, 3, 0)$ & $629356$  & $(6, 12, 18, 9)$ \\
\bottomrule\end{tabular}
\caption{Associated graded for $\bigwedge^{6}\tangent_{G/P}$ for $(\mathrm{F}_4,\alpha_4)$ (part 1)}
\label{table:wedge-6-F4-P4-1}
\end{table}

\begin{table}
\centering
\begin{tabular}{crrcrcr}\toprule
weight          & rank  & degree & representation & dimension  & sum of roots \\ \midrule
$(0, 0, 2, 2)$  & $35$  & $0$    & $(0, 0, 2, 2)$ & $1341522$ & $(6, 12, 18, 10)$ \\
$(1, 0, 2, 1)$  & $189$ & $0$    & $(1, 0, 2, 1)$ & $3921372$ & $(7, 13, 19, 10)$ \\
$(0, 1, 0, 3)$  & $21$  & $0$    & $(0, 1, 0, 3)$ & $952952$  & $(6, 12, 17, 10)$ \\
$(0, 1, 2, 0)$  & $378$ & $0$    & $(0, 1, 2, 0)$ & $3508596$ & $(7, 14, 20, 10)$ \\
$(1, 1, 0, 2)$  & $105$ & $0$    & $(1, 1, 0, 2)$ & $2792556$ & $(7, 13, 18, 10)$ \\
$(0, 0, 2, 2)$  & $35$  & $0$    & $(0, 0, 2, 2)$ & $1341522$ & $(6, 12, 18, 10)$ \\
$(1, 0, 0, 4)$  & $7$   & $0$    & $(1, 0, 0, 4)$ & $412776$  & $(6, 11, 16, 10)$ \\
$(1, 0, 2, 1)$  & $189$ & $0$    & $(1, 0, 2, 1)$ & $3921372$ & $(7, 13, 19, 10)$ \\
$(0, 1, 0, 3)$  & $21$  & $0$    & $(0, 1, 0, 3)$ & $952952$  & $(6, 12, 17, 10)$ \\
\addlinespace
$(0, 0, 1, 4)$  & $8$   & $0$    & $(0, 0, 1, 4)$ & $1042899$ & $(6, 12, 18, 11)$ \\
$(1, 0, 1, 3)$  & $48$  & $0$    & $(1, 0, 1, 3)$ & $4313088$ & $(7, 13, 19, 11)$ \\
$(0, 1, 1, 2)$  & $112$ & $0$    & $(0, 1, 1, 2)$ & $7113106$ & $(7, 14, 20, 11)$ \\
\addlinespace
$(1, 0, 0, 5)$  & $7$   & $0$    & $(1, 0, 0, 5)$ & $1850212$ & $(7, 13, 19, 12)$ \\
\bottomrule\end{tabular}
\caption{Associated graded for $\bigwedge^{6}\tangent_{G/P}$ for $(\mathrm{F}_4,\alpha_4)$ (part 2)}
\label{table:wedge-6-F4-P4-2}
\end{table}

\begin{table}
\centering
\begin{tabular}{crrcrcr}\toprule
weight          & rank  & degree & representation & dimension  & sum of roots \\ \midrule
$(0, 0, 1, 2)$  & $8$   & $0$    & $(0, 0, 1, 2)$ & $34749$    & $(4, 8, 12, 7)$ \\
\addlinespace
$(0, 0, 0, 4)$  & $1$   & $0$    & $(0, 0, 0, 4)$ & $16302$    & $(4, 8, 12, 8)$ \\
$(2, 0, 0, 2)$  & $27$  & $0$    & $(2, 0, 0, 2)$ & $160056$   & $(6, 10, 14, 8)$ \\
$(0, 1, 0, 2)$  & $21$  & $0$    & $(0, 1, 0, 2)$ & $160056$   & $(5, 10, 14, 8)$ \\
$(1, 1, 0, 1)$  & $105$ & $0$    & $(1, 1, 0, 1)$ & $379848$   & $(6, 11, 15, 8)$ \\
$(1, 0, 0, 3)$  & $7$   & $0$    & $(1, 0, 0, 3)$ & $76076$    & $(5, 9, 13, 8)$ \\
$(0, 0, 2, 1)$  & $35$  & $0$    & $(0, 0, 2, 1)$ & $205751$   & $(5, 10, 15, 8)$ \\
\addlinespace
$(0, 1, 1, 1)$  & $112$ & $0$    & $(0, 1, 1, 1)$ & $1118208$  & $(6, 12, 17, 9)$ \\
$(0, 0, 3, 0)$  & $112$ & $0$    & $(0, 0, 3, 0)$ & $629356$   & $(6, 12, 18, 9)$ \\
$(1, 0, 1, 2)$  & $48$  & $0$    & $(1, 0, 1, 2)$ & $787644$   & $(6, 11, 16, 9)$ \\
$(1, 0, 1, 2)$  & $48$  & $0$    & $(1, 0, 1, 2)$ & $787644$   & $(6, 11, 16, 9)$ \\
$(1, 1, 1, 0)$  & $512$ & $0$    & $(1, 1, 1, 0)$ & $1801371$  & $(7, 13, 18, 9)$ \\
$(2, 0, 1, 1)$  & $168$ & $0$    & $(2, 0, 1, 1)$ & $1327104$  & $(7, 12, 17, 9)$ \\
$(0, 0, 1, 3)$  & $8$   & $0$    & $(0, 0, 1, 3)$ & $212992$   & $(5, 10, 15, 9)$ \\
$(0, 0, 1, 3)$  & $8$   & $0$    & $(0, 0, 1, 3)$ & $212992$   & $(5, 10, 15, 9)$ \\
$(1, 0, 1, 2)$  & $48$  & $0$    & $(1, 0, 1, 2)$ & $787644$   & $(6, 11, 16, 9)$ \\
$(0, 1, 1, 1)$  & $112$ & $0$    & $(0, 1, 1, 1)$ & $1118208$  & $(6, 12, 17, 9)$ \\
\addlinespace
$(1, 1, 0, 2)$  & $105$ & $0$    & $(1, 1, 0, 2)$ & $2792556$  & $(7, 13, 18, 10)$ \\
$(0, 1, 2, 0)$  & $378$ & $0$    & $(0, 1, 2, 0)$ & $3508596$  & $(7, 14, 20, 10)$ \\
$(0, 0, 4, -1)$ & $294$ &        &                &            & $(7, 14, 21, 10)$ \\
$(0, 1, 0, 3)$  & $21$  & $0$    & $(0, 1, 0, 3)$ & $952952$   & $(6, 12, 17, 10)$ \\
$(0, 0, 0, 5)$  & $1$   & $0$    & $(0, 0, 0, 5)$ & $81081$    & $(5, 10, 15, 10)$ \\
$(0, 2, 0, 1)$  & $168$ & $0$    & $(0, 2, 0, 1)$ & $2488563$  & $(7, 14, 19, 10)$ \\
$(2, 0, 0, 3)$  & $27$  & $0$    & $(2, 0, 0, 3)$ & $1002456$  & $(7, 12, 17, 10)$ \\
$(1, 0, 0, 4)$  & $7$   & $0$    & $(1, 0, 0, 4)$ & $412776$   & $(6, 11, 16, 10)$ \\
$(1, 0, 2, 1)$  & $189$ & $0$    & $(1, 0, 2, 1)$ & $3921372$  & $(7, 13, 19, 10)$ \\
$(0, 0, 2, 2)$  & $35$  & $0$    & $(0, 0, 2, 2)$ & $1341522$  & $(6, 12, 18, 10)$ \\
$(0, 0, 2, 2)$  & $35$  & $0$    & $(0, 0, 2, 2)$ & $1341522$  & $(6, 12, 18, 10)$ \\
$(1, 0, 2, 1)$  & $189$ & $0$    & $(1, 0, 2, 1)$ & $3921372$  & $(7, 13, 19, 10)$ \\
$(0, 1, 0, 3)$  & $21$  & $0$    & $(0, 1, 0, 3)$ & $952952$   & $(6, 12, 17, 10)$ \\
$(0, 0, 2, 2)$  & $35$  & $0$    & $(0, 0, 2, 2)$ & $1341522$  & $(6, 12, 18, 10)$ \\
$(1, 0, 2, 1)$  & $189$ & $0$    & $(1, 0, 2, 1)$ & $3921372$  & $(7, 13, 19, 10)$ \\
$(1, 1, 0, 2)$  & $105$ & $0$    & $(1, 1, 0, 2)$ & $2792556$  & $(7, 13, 18, 10)$ \\
$(2, 0, 2, 0)$  & $616$ & $0$    & $(2, 0, 2, 0)$ & $4582656$  & $(8, 14, 20, 10)$ \\
$(0, 0, 2, 2)$  & $35$  & $0$    & $(0, 0, 2, 2)$ & $1341522$  & $(6, 12, 18, 10)$ \\
\bottomrule\end{tabular}
\caption{Associated graded for $\bigwedge^{7}\tangent_{G/P}$ for $(\mathrm{F}_4,\alpha_4)$ (part 1)}
\label{table:wedge-7-F4-P4-1}
\end{table}

\begin{table}
\centering
\begin{tabular}{crrcrcr}\toprule
weight & rank & degree & representation & dimension & sum of roots \\ \midrule
$(0, 1, 1, 2)$  & $112$ & $0$    & $(0, 1, 1, 2)$ & $7113106$  & $(7, 14, 20, 11)$ \\
$(0, 1, 1, 2)$  & $112$ & $0$    & $(0, 1, 1, 2)$ & $7113106$  & $(7, 14, 20, 11)$ \\
$(0, 0, 3, 1)$  & $112$ & $0$    & $(0, 0, 3, 1)$ & $5218304$  & $(7, 14, 21, 11)$ \\
$(1, 0, 1, 3)$  & $48$  & $0$    & $(1, 0, 1, 3)$ & $4313088$  & $(7, 13, 19, 11)$ \\
$(1, 0, 1, 3)$  & $48$  & $0$    & $(1, 0, 1, 3)$ & $4313088$  & $(7, 13, 19, 11)$ \\
$(1, 1, 1, 1)$  & $512$ & $0$    & $(1, 1, 1, 1)$ & $16777216$ & $(8, 15, 21, 11)$ \\
$(2, 0, 1, 2)$  & $168$ & $0$    & $(2, 0, 1, 2)$ & $8843094$  & $(8, 14, 20, 11)$ \\
$(1, 0, 3, 0)$  & $560$ & $0$    & $(1, 0, 3, 0)$ & $10482472$ & $(8, 15, 22, 11)$ \\
$(0, 0, 1, 4)$  & $8$   & $0$    & $(0, 0, 1, 4)$ & $1042899$  & $(6, 12, 18, 11)$ \\
$(0, 1, 1, 2)$  & $112$ & $0$    & $(0, 1, 1, 2)$ & $7113106$  & $(7, 14, 20, 11)$ \\
$(1, 0, 1, 3)$  & $48$  & $0$    & $(1, 0, 1, 3)$ & $4313088$  & $(7, 13, 19, 11)$ \\
$(0, 0, 1, 4)$  & $8$   & $0$    & $(0, 0, 1, 4)$ & $1042899$  & $(6, 12, 18, 11)$ \\
$(0, 0, 3, 1)$  & $112$ & $0$    & $(0, 0, 3, 1)$ & $5218304$  & $(7, 14, 21, 11)$ \\
\addlinespace
$(1, 1, 0, 3)$  & $105$ & $0$    & $(1, 1, 0, 3)$ & $15031926$ & $(8, 15, 21, 12)$ \\
$(1, 0, 0, 5)$  & $7$   & $0$    & $(1, 0, 0, 5)$ & $1850212$  & $(7, 13, 19, 12)$ \\
$(0, 0, 2, 3)$  & $35$  & $0$    & $(0, 0, 2, 3)$ & $6680856$  & $(7, 14, 21, 12)$ \\
$(0, 0, 0, 6)$  & $1$   & $0$    & $(0, 0, 0, 6)$ & $342056$   & $(6, 12, 18, 12)$ \\
$(0, 0, 2, 3)$  & $35$  & $0$    & $(0, 0, 2, 3)$ & $6680856$  & $(7, 14, 21, 12)$ \\
$(0, 2, 0, 2)$  & $168$ & $0$    & $(0, 2, 0, 2)$ & $15997696$ & $(8, 16, 22, 12)$ \\
$(2, 0, 0, 4)$  & $27$  & $0$    & $(2, 0, 0, 4)$ & $4940676$  & $(8, 14, 20, 12)$ \\
$(1, 0, 2, 2)$  & $189$ & $0$    & $(1, 0, 2, 2)$ & $23056488$ & $(8, 15, 22, 12)$ \\
$(0, 1, 0, 4)$  & $21$  & $0$    & $(0, 1, 0, 4)$ & $4528953$  & $(7, 14, 20, 12)$ \\
\addlinespace
$(1, 0, 1, 4)$  & $48$  & $0$    & $(1, 0, 1, 4)$ & $19214624$ & $(8, 15, 22, 13)$ \\
$(0, 0, 1, 5)$  & $8$   & $0$    & $(0, 0, 1, 5)$ & $4313088$  & $(7, 14, 21, 13)$ \\
\addlinespace
$(0, 0, 0, 7)$  & $1$   & $0$    & $(0, 0, 0, 7)$ & $1264120$  & $(7, 14, 21, 14)$ \\
\bottomrule\end{tabular}
\caption{Associated graded for $\bigwedge^{7}\tangent_{G/P}$ for $(\mathrm{F}_4,\alpha_4)$ (part 2)}
\label{table:wedge-7-F4-P4-2}
\end{table}

\begin{table}
\centering
\begin{tabular}{crrcrcr}\toprule
weight & rank & degree & representation & dimension & sum of roots \\ \midrule
$(0, 0, 0, 4)$ & $1$   & $0$ & $(0, 0, 0, 4)$ & $16302$    & $(4, 8, 12, 8)$ \\
\addlinespace
$(1, 0, 1, 2)$ & $48$  & $0$ & $(1, 0, 1, 2)$ & $787644$   & $(6, 11, 16, 9)$ \\
$(0, 0, 1, 3)$ & $8$   & $0$ & $(0, 0, 1, 3)$ & $212992$   & $(5, 10, 15, 9)$ \\
\addlinespace
$(1, 1, 0, 2)$ & $105$ & $0$ & $(1, 1, 0, 2)$ & $2792556$  & $(7, 13, 18, 10)$ \\
$(1, 0, 0, 4)$ & $7$   & $0$ & $(1, 0, 0, 4)$ & $412776$   & $(6, 11, 16, 10)$ \\
$(0, 0, 2, 2)$ & $35$  & $0$ & $(0, 0, 2, 2)$ & $1341522$  & $(6, 12, 18, 10)$ \\
$(0, 0, 0, 5)$ & $1$   & $0$ & $(0, 0, 0, 5)$ & $81081$    & $(5, 10, 15, 10)$ \\
$(0, 0, 2, 2)$ & $35$  & $0$ & $(0, 0, 2, 2)$ & $1341522$  & $(6, 12, 18, 10)$ \\
$(0, 2, 0, 1)$ & $168$ & $0$ & $(0, 2, 0, 1)$ & $2488563$  & $(7, 14, 19, 10)$ \\
$(2, 0, 0, 3)$ & $27$  & $0$ & $(2, 0, 0, 3)$ & $1002456$  & $(7, 12, 17, 10)$ \\
$(1, 0, 2, 1)$ & $189$ & $0$ & $(1, 0, 2, 1)$ & $3921372$  & $(7, 13, 19, 10)$ \\
$(0, 1, 0, 3)$ & $21$  & $0$ & $(0, 1, 0, 3)$ & $952952$   & $(6, 12, 17, 10)$ \\
\addlinespace
$(0, 1, 1, 2)$ & $112$ & $0$ & $(0, 1, 1, 2)$ & $7113106$  & $(7, 14, 20, 11)$ \\
$(0, 1, 1, 2)$ & $112$ & $0$ & $(0, 1, 1, 2)$ & $7113106$  & $(7, 14, 20, 11)$ \\
$(0, 0, 3, 1)$ & $112$ & $0$ & $(0, 0, 3, 1)$ & $5218304$  & $(7, 14, 21, 11)$ \\
$(1, 0, 1, 3)$ & $48$  & $0$ & $(1, 0, 1, 3)$ & $4313088$  & $(7, 13, 19, 11)$ \\
$(1, 0, 1, 3)$ & $48$  & $0$ & $(1, 0, 1, 3)$ & $4313088$  & $(7, 13, 19, 11)$ \\
$(1, 1, 1, 1)$ & $512$ & $0$ & $(1, 1, 1, 1)$ & $16777216$ & $(8, 15, 21, 11)$ \\
$(2, 0, 1, 2)$ & $168$ & $0$ & $(2, 0, 1, 2)$ & $8843094$  & $(8, 14, 20, 11)$ \\
$(1, 0, 3, 0)$ & $560$ & $0$ & $(1, 0, 3, 0)$ & $10482472$ & $(8, 15, 22, 11)$ \\
$(0, 0, 1, 4)$ & $8$   & $0$ & $(0, 0, 1, 4)$ & $1042899$  & $(6, 12, 18, 11)$ \\
$(0, 1, 1, 2)$ & $112$ & $0$ & $(0, 1, 1, 2)$ & $7113106$  & $(7, 14, 20, 11)$ \\
$(1, 0, 1, 3)$ & $48$  & $0$ & $(1, 0, 1, 3)$ & $4313088$  & $(7, 13, 19, 11)$ \\
$(0, 0, 1, 4)$ & $8$   & $0$ & $(0, 0, 1, 4)$ & $1042899$  & $(6, 12, 18, 11)$ \\
$(0, 0, 3, 1)$ & $112$ & $0$ & $(0, 0, 3, 1)$ & $5218304$  & $(7, 14, 21, 11)$ \\
\addlinespace
$(1, 1, 0, 3)$ & $105$ & $0$ & $(1, 1, 0, 3)$ & $15031926$ & $(8, 15, 21, 12)$ \\
$(0, 1, 2, 1)$ & $378$ & $0$ & $(0, 1, 2, 1)$ & $28481544$ & $(8, 16, 23, 12)$ \\
$(0, 0, 4, 0)$ & $294$ & $0$ & $(0, 0, 4, 0)$ & $11955216$ & $(8, 16, 24, 12)$ \\
$(0, 1, 0, 4)$ & $21$  & $0$ & $(0, 1, 0, 4)$ & $4528953$  & $(7, 14, 20, 12)$ \\
$(0, 0, 0, 6)$ & $1$   & $0$ & $(0, 0, 0, 6)$ & $342056$   & $(6, 12, 18, 12)$ \\
$(0, 2, 0, 2)$ & $168$ & $0$ & $(0, 2, 0, 2)$ & $15997696$ & $(8, 16, 22, 12)$ \\
$(2, 0, 0, 4)$ & $27$  & $0$ & $(2, 0, 0, 4)$ & $4940676$  & $(8, 14, 20, 12)$ \\
$(1, 0, 0, 5)$ & $7$   & $0$ & $(1, 0, 0, 5)$ & $1850212$  & $(7, 13, 19, 12)$ \\
$(1, 0, 2, 2)$ & $189$ & $0$ & $(1, 0, 2, 2)$ & $23056488$ & $(8, 15, 22, 12)$ \\
$(0, 0, 2, 3)$ & $35$  & $0$ & $(0, 0, 2, 3)$ & $6680856$  & $(7, 14, 21, 12)$ \\
$(0, 0, 2, 3)$ & $35$  & $0$ & $(0, 0, 2, 3)$ & $6680856$  & $(7, 14, 21, 12)$ \\
$(1, 0, 2, 2)$ & $189$ & $0$ & $(1, 0, 2, 2)$ & $23056488$ & $(8, 15, 22, 12)$ \\
$(0, 1, 0, 4)$ & $21$  & $0$ & $(0, 1, 0, 4)$ & $4528953$  & $(7, 14, 20, 12)$ \\
$(0, 0, 2, 3)$ & $35$  & $0$ & $(0, 0, 2, 3)$ & $6680856$  & $(7, 14, 21, 12)$ \\
$(1, 0, 2, 2)$ & $189$ & $0$ & $(1, 0, 2, 2)$ & $23056488$ & $(8, 15, 22, 12)$ \\
$(1, 1, 0, 3)$ & $105$ & $0$ & $(1, 1, 0, 3)$ & $15031926$ & $(8, 15, 21, 12)$ \\
$(2, 0, 2, 1)$ & $616$ & $0$ & $(2, 0, 2, 1)$ & $38854452$ & $(9, 16, 23, 12)$ \\
$(0, 0, 2, 3)$ & $35$  & $0$ & $(0, 0, 2, 3)$ & $6680856$  & $(7, 14, 21, 12)$ \\
\bottomrule\end{tabular}
\caption{Associated graded for $\bigwedge^{8}\tangent_{G/P}$ for $(\mathrm{F}_4,\alpha_4)$ (part 1)}
\label{table:wedge-8-F4-P4-1}
\end{table}

\begin{table}
\centering
\begin{tabular}{crrcrcr}\toprule
weight & rank & degree & representation & dimension & sum of roots \\ \midrule
$(0, 1, 1, 3)$ & $112$ & $0$ & $(0, 1, 1, 3)$ & $34504704$ & $(8, 16, 23, 13)$ \\
$(0, 0, 3, 2)$ & $112$ & $0$ & $(0, 0, 3, 2)$ & $28068768$ & $(8, 16, 24, 13)$ \\
$(1, 0, 1, 4)$ & $48$  & $0$ & $(1, 0, 1, 4)$ & $19214624$ & $(8, 15, 22, 13)$ \\
$(1, 0, 1, 4)$ & $48$  & $0$ & $(1, 0, 1, 4)$ & $19214624$ & $(8, 15, 22, 13)$ \\
$(1, 1, 1, 2)$ & $512$ & $0$ & $(1, 1, 1, 2)$ & $97274034$ & $(9, 17, 24, 13)$ \\
$(2, 0, 1, 3)$ & $168$ & $0$ & $(2, 0, 1, 3)$ & $44355584$ & $(9, 16, 23, 13)$ \\
$(0, 0, 1, 5)$ & $8$   & $0$ & $(0, 0, 1, 5)$ & $4313088$  & $(7, 14, 21, 13)$ \\
$(0, 0, 1, 5)$ & $8$   & $0$ & $(0, 0, 1, 5)$ & $4313088$  & $(7, 14, 21, 13)$ \\
$(1, 0, 1, 4)$ & $48$  & $0$ & $(1, 0, 1, 4)$ & $19214624$ & $(8, 15, 22, 13)$ \\
$(0, 1, 1, 3)$ & $112$ & $0$ & $(0, 1, 1, 3)$ & $34504704$ & $(8, 16, 23, 13)$ \\
\addlinespace
$(0, 0, 0, 7)$ & $1$   & $0$ & $(0, 0, 0, 7)$ & $1264120$  & $(7, 14, 21, 14)$ \\
$(2, 0, 0, 5)$ & $27$  & $0$ & $(2, 0, 0, 5)$ & $20407140$ & $(9, 16, 23, 14)$ \\
$(0, 1, 0, 5)$ & $21$  & $0$ & $(0, 1, 0, 5)$ & $18206370$ & $(8, 16, 23, 14)$ \\
$(1, 1, 0, 4)$ & $105$ & $0$ & $(1, 1, 0, 4)$ & $65609375$ & $(9, 17, 24, 14)$ \\
$(1, 0, 0, 6)$ & $7$   & $0$ & $(1, 0, 0, 6)$ & $7147140$  & $(8, 15, 22, 14)$ \\
$(0, 0, 2, 4)$ & $35$  & $0$ & $(0, 0, 2, 4)$ & $27625000$ & $(8, 16, 24, 14)$ \\
\addlinespace
$(0, 0, 1, 6)$ & $8$   & $0$ & $(0, 0, 1, 6)$ & $15611882$ & $(8, 16, 24, 15)$ \\
\bottomrule\end{tabular}
\caption{Associated graded for $\bigwedge^{8}\tangent_{G/P}$ for $(\mathrm{F}_4,\alpha_4)$ (part 2)}
\label{table:wedge-8-F4-P4-2}
\end{table}

\begin{table}
\centering
\begin{tabular}{crrcrcr}\toprule
weight         & rank  & degree & representation & dimension   & sum of roots \\ \midrule
$(1, 0, 0, 4)$ & $7$   & $0$    & $(1, 0, 0, 4)$ & $412776$    & $(6, 11, 16, 10)$ \\
\addlinespace
$(0, 0, 1, 4)$ & $8$   & $0$    & $(0, 0, 1, 4)$ & $1042899$   & $(6, 12, 18, 11)$ \\
$(1, 0, 1, 3)$ & $48$  & $0$    & $(1, 0, 1, 3)$ & $4313088$   & $(7, 13, 19, 11)$ \\
$(0, 1, 1, 2)$ & $112$ & $0$    & $(0, 1, 1, 2)$ & $7113106$   & $(7, 14, 20, 11)$ \\
\addlinespace
$(0, 0, 2, 3)$ & $35$  & $0$    & $(0, 0, 2, 3)$ & $6680856$   & $(7, 14, 21, 12)$ \\
$(1, 0, 2, 2)$ & $189$ & $0$    & $(1, 0, 2, 2)$ & $23056488$  & $(8, 15, 22, 12)$ \\
$(0, 1, 0, 4)$ & $21$  & $0$    & $(0, 1, 0, 4)$ & $4528953$   & $(7, 14, 20, 12)$ \\
$(0, 1, 2, 1)$ & $378$ & $0$    & $(0, 1, 2, 1)$ & $28481544$  & $(8, 16, 23, 12)$ \\
$(1, 1, 0, 3)$ & $105$ & $0$    & $(1, 1, 0, 3)$ & $15031926$  & $(8, 15, 21, 12)$ \\
$(0, 0, 2, 3)$ & $35$  & $0$    & $(0, 0, 2, 3)$ & $6680856$   & $(7, 14, 21, 12)$ \\
$(1, 0, 0, 5)$ & $7$   & $0$    & $(1, 0, 0, 5)$ & $1850212$   & $(7, 13, 19, 12)$ \\
$(1, 0, 2, 2)$ & $189$ & $0$    & $(1, 0, 2, 2)$ & $23056488$  & $(8, 15, 22, 12)$ \\
$(0, 1, 0, 4)$ & $21$  & $0$    & $(0, 1, 0, 4)$ & $4528953$   & $(7, 14, 20, 12)$ \\
\addlinespace
$(0, 1, 1, 3)$ & $112$ & $0$    & $(0, 1, 1, 3)$ & $34504704$  & $(8, 16, 23, 13)$ \\
$(0, 1, 1, 3)$ & $112$ & $0$    & $(0, 1, 1, 3)$ & $34504704$  & $(8, 16, 23, 13)$ \\
$(0, 0, 3, 2)$ & $112$ & $0$    & $(0, 0, 3, 2)$ & $28068768$  & $(8, 16, 24, 13)$ \\
$(1, 0, 1, 4)$ & $48$  & $0$    & $(1, 0, 1, 4)$ & $19214624$  & $(8, 15, 22, 13)$ \\
$(1, 0, 1, 4)$ & $48$  & $0$    & $(1, 0, 1, 4)$ & $19214624$  & $(8, 15, 22, 13)$ \\
$(1, 1, 1, 2)$ & $512$ & $0$    & $(1, 1, 1, 2)$ & $97274034$  & $(9, 17, 24, 13)$ \\
$(2, 0, 1, 3)$ & $168$ & $0$    & $(2, 0, 1, 3)$ & $44355584$  & $(9, 16, 23, 13)$ \\
$(1, 0, 3, 1)$ & $560$ & $0$    & $(1, 0, 3, 1)$ & $78962688$  & $(9, 17, 25, 13)$ \\
$(0, 0, 1, 5)$ & $8$   & $0$    & $(0, 0, 1, 5)$ & $4313088$   & $(7, 14, 21, 13)$ \\
$(0, 1, 1, 3)$ & $112$ & $0$    & $(0, 1, 1, 3)$ & $34504704$  & $(8, 16, 23, 13)$ \\
$(1, 0, 1, 4)$ & $48$  & $0$    & $(1, 0, 1, 4)$ & $19214624$  & $(8, 15, 22, 13)$ \\
$(0, 0, 1, 5)$ & $8$   & $0$    & $(0, 0, 1, 5)$ & $4313088$   & $(7, 14, 21, 13)$ \\
$(0, 0, 3, 2)$ & $112$ & $0$    & $(0, 0, 3, 2)$ & $28068768$  & $(8, 16, 24, 13)$ \\
\addlinespace
$(0, 1, 2, 2)$ & $378$ & $0$    & $(0, 1, 2, 2)$ & $149652360$ & $(9, 18, 26, 14)$ \\
$(1, 1, 0, 4)$ & $105$ & $0$    & $(1, 1, 0, 4)$ & $65609375$  & $(9, 17, 24, 14)$ \\
$(0, 0, 2, 4)$ & $35$  & $0$    & $(0, 0, 2, 4)$ & $27625000$  & $(8, 16, 24, 14)$ \\
$(1, 0, 0, 6)$ & $7$   & $0$    & $(1, 0, 0, 6)$ & $7147140$   & $(8, 15, 22, 14)$ \\
$(1, 0, 2, 3)$ & $189$ & $0$    & $(1, 0, 2, 3)$ & $105257880$ & $(9, 17, 25, 14)$ \\
$(0, 1, 0, 5)$ & $21$  & $0$    & $(0, 1, 0, 5)$ & $18206370$  & $(8, 16, 23, 14)$ \\
$(1, 1, 0, 4)$ & $105$ & $0$    & $(1, 1, 0, 4)$ & $65609375$  & $(9, 17, 24, 14)$ \\
$(1, 0, 0, 6)$ & $7$   & $0$    & $(1, 0, 0, 6)$ & $7147140$   & $(8, 15, 22, 14)$ \\
$(0, 0, 2, 4)$ & $35$  & $0$    & $(0, 0, 2, 4)$ & $27625000$  & $(8, 16, 24, 14)$ \\
$(0, 1, 0, 5)$ & $21$  & $0$    & $(0, 1, 0, 5)$ & $18206370$  & $(8, 16, 23, 14)$ \\
$(1, 0, 2, 3)$ & $189$ & $0$    & $(1, 0, 2, 3)$ & $105257880$ & $(9, 17, 25, 14)$ \\
$(2, 1, 0, 3)$ & $330$ & $0$    & $(2, 1, 0, 3)$ & $131625000$ & $(10, 18, 25, 14)$ \\
$(2, 0, 0, 5)$ & $27$  & $0$    & $(2, 0, 0, 5)$ & $20407140$  & $(9, 16, 23, 14)$ \\
$(0, 1, 0, 5)$ & $21$  & $0$    & $(0, 1, 0, 5)$ & $18206370$  & $(8, 16, 23, 14)$ \\
\bottomrule\end{tabular}
\caption{Associated graded for $\bigwedge^{9}\tangent_{G/P}$ for $(\mathrm{F}_4,\alpha_4)$ (part 1)}
\label{table:wedge-9-F4-P4-1}
\end{table}

\begin{table}
\centering
\begin{tabular}{crrcrcr}\toprule
weight & rank & degree & representation & dimension & sum of roots \\ \midrule
$(2, 0, 1, 4)$ & $168$ & $0$    & $(2, 0, 1, 4)$ & $183324141$ & $(10, 18, 26, 15)$ \\
$(0, 1, 1, 4)$ & $112$ & $0$    & $(0, 1, 1, 4)$ & $139087676$ & $(9, 18, 26, 15)$ \\
$(1, 0, 1, 5)$ & $48$  & $0$    & $(1, 0, 1, 5)$ & $73322496$  & $(9, 17, 25, 15)$ \\
$(0, 0, 1, 6)$ & $8$   & $0$    & $(0, 0, 1, 6)$ & $15611882$  & $(8, 16, 24, 15)$ \\
$(1, 0, 1, 5)$ & $48$  & $0$    & $(1, 0, 1, 5)$ & $73322496$  & $(9, 17, 25, 15)$ \\
$(0, 0, 1, 6)$ & $8$   & $0$    & $(0, 0, 1, 6)$ & $15611882$  & $(8, 16, 24, 15)$ \\
\addlinespace
$(1, 0, 0, 7)$ & $7$   & $0$    & $(1, 0, 0, 7)$ & $24488568$  & $(9, 17, 25, 16)$ \\
$(0, 1, 0, 6)$ & $21$  & $0$    & $(0, 1, 0, 6)$ & $64194312$  & $(9, 18, 26, 16)$ \\
\bottomrule\end{tabular}
\caption{Associated graded for $\bigwedge^{9}\tangent_{G/P}$ for $(\mathrm{F}_4,\alpha_4)$ (part 2)}
\label{table:wedge-9-F4-P4-2}
\end{table}

\begin{table}
\centering
\begin{tabular}{crrcrcr}\toprule
weight & rank & degree & representation & dimension & sum of roots \\ \midrule
$(0, 1, 0, 4)$ & $21$  & $0$ & $(0, 1, 0, 4)$ & $4528953$   & $(7, 14, 20, 12)$ \\
\addlinespace
$(0, 1, 1, 3)$ & $112$ & $0$ & $(0, 1, 1, 3)$ & $34504704$  & $(8, 16, 23, 13)$ \\
$(1, 0, 1, 4)$ & $48$  & $0$ & $(1, 0, 1, 4)$ & $19214624$  & $(8, 15, 22, 13)$ \\
$(0, 0, 1, 5)$ & $8$   & $0$ & $(0, 0, 1, 5)$ & $4313088$   & $(7, 14, 21, 13)$ \\
$(0, 0, 3, 2)$ & $112$ & $0$ & $(0, 0, 3, 2)$ & $28068768$  & $(8, 16, 24, 13)$ \\
\addlinespace
$(0, 0, 2, 4)$ & $35$  & $0$ & $(0, 0, 2, 4)$ & $27625000$  & $(8, 16, 24, 14)$ \\
$(1, 0, 2, 3)$ & $189$ & $0$ & $(1, 0, 2, 3)$ & $105257880$ & $(9, 17, 25, 14)$ \\
$(0, 1, 0, 5)$ & $21$  & $0$ & $(0, 1, 0, 5)$ & $18206370$  & $(8, 16, 23, 14)$ \\
$(0, 1, 2, 2)$ & $378$ & $0$ & $(0, 1, 2, 2)$ & $149652360$ & $(9, 18, 26, 14)$ \\
$(1, 1, 0, 4)$ & $105$ & $0$ & $(1, 1, 0, 4)$ & $65609375$  & $(9, 17, 24, 14)$ \\
$(0, 0, 2, 4)$ & $35$  & $0$ & $(0, 0, 2, 4)$ & $27625000$  & $(8, 16, 24, 14)$ \\
$(1, 0, 0, 6)$ & $7$   & $0$ & $(1, 0, 0, 6)$ & $7147140$   & $(8, 15, 22, 14)$ \\
$(1, 0, 2, 3)$ & $189$ & $0$ & $(1, 0, 2, 3)$ & $105257880$ & $(9, 17, 25, 14)$ \\
$(0, 1, 0, 5)$ & $21$  & $0$ & $(0, 1, 0, 5)$ & $18206370$  & $(8, 16, 23, 14)$ \\
\addlinespace
$(0, 1, 1, 4)$ & $112$ & $0$ & $(0, 1, 1, 4)$ & $139087676$ & $(9, 18, 26, 15)$ \\
$(0, 0, 3, 3)$ & $112$ & $0$ & $(0, 0, 3, 3)$ & $119488512$ & $(9, 18, 27, 15)$ \\
$(1, 0, 1, 5)$ & $48$  & $0$ & $(1, 0, 1, 5)$ & $73322496$  & $(9, 17, 25, 15)$ \\
$(1, 0, 1, 5)$ & $48$  & $0$ & $(1, 0, 1, 5)$ & $73322496$  & $(9, 17, 25, 15)$ \\
$(1, 1, 1, 3)$ & $512$ & $0$ & $(1, 1, 1, 3)$ & $436207616$ & $(10, 19, 27, 15)$ \\
$(2, 0, 1, 4)$ & $168$ & $0$ & $(2, 0, 1, 4)$ & $183324141$ & $(10, 18, 26, 15)$ \\
$(0, 0, 1, 6)$ & $8$   & $0$ & $(0, 0, 1, 6)$ & $15611882$  & $(8, 16, 24, 15)$ \\
$(0, 0, 1, 6)$ & $8$   & $0$ & $(0, 0, 1, 6)$ & $15611882$  & $(8, 16, 24, 15)$ \\
$(1, 0, 1, 5)$ & $48$  & $0$ & $(1, 0, 1, 5)$ & $73322496$  & $(9, 17, 25, 15)$ \\
$(0, 1, 1, 4)$ & $112$ & $0$ & $(0, 1, 1, 4)$ & $139087676$ & $(9, 18, 26, 15)$ \\
\addlinespace
$(0, 0, 2, 5)$ & $35$  & $0$ & $(0, 0, 2, 5)$ & $99243144$  & $(9, 18, 27, 16)$ \\
$(1, 0, 2, 4)$ & $189$ & $0$ & $(1, 0, 2, 4)$ & $404061372$ & $(10, 19, 28, 16)$ \\
$(0, 1, 0, 6)$ & $21$  & $0$ & $(0, 1, 0, 6)$ & $64194312$  & $(9, 18, 26, 16)$ \\
$(0, 0, 0, 8)$ & $1$   & $0$ & $(0, 0, 0, 8)$ & $4188834$   & $(8, 16, 24, 16)$ \\
$(2, 0, 0, 6)$ & $27$  & $0$ & $(2, 0, 0, 6)$ & $73465704$  & $(10, 18, 26, 16)$ \\
$(0, 1, 0, 6)$ & $21$  & $0$ & $(0, 1, 0, 6)$ & $64194312$  & $(9, 18, 26, 16)$ \\
$(1, 0, 0, 7)$ & $7$   & $0$ & $(1, 0, 0, 7)$ & $24488568$  & $(9, 17, 25, 16)$ \\
$(1, 0, 0, 7)$ & $7$   & $0$ & $(1, 0, 0, 7)$ & $24488568$  & $(9, 17, 25, 16)$ \\
$(1, 1, 0, 5)$ & $105$ & $0$ & $(1, 1, 0, 5)$ & $245188944$ & $(10, 19, 27, 16)$ \\
$(3, 0, 0, 5)$ & $77$  & $0$ & $(3, 0, 0, 5)$ & $149992479$ & $(11, 19, 27, 16)$ \\
\addlinespace
$(1, 0, 1, 6)$ & $48$  & $0$ & $(1, 0, 1, 6)$ & $247606632$ & $(10, 19, 28, 17)$ \\
$(0, 0, 1, 7)$ & $8$   & $0$ & $(0, 0, 1, 7)$ & $50692096$  & $(9, 18, 27, 17)$ \\
\bottomrule\end{tabular}
\caption{Associated graded for $\bigwedge^{10}\tangent_{G/P}$ for $(\mathrm{F}_4,\alpha_4)$}
\label{table:wedge-10-F4-P4}
\end{table}

\begin{table}
\centering
\begin{tabular}{crrcrcr}\toprule
weight & rank & degree & representation & dimension & sum of roots \\ \midrule
$(0, 0, 2, 4)$ & $35$  & $0$ & $(0, 0, 2, 4)$ & $27625000$  & $(8, 16, 24, 14)$ \\
\addlinespace
$(0, 1, 1, 4)$ & $112$ & $0$ & $(0, 1, 1, 4)$ & $139087676$ & $(9, 18, 26, 15)$ \\
$(1, 0, 1, 5)$ & $48$  & $0$ & $(1, 0, 1, 5)$ & $73322496$  & $(9, 17, 25, 15)$ \\
$(0, 0, 1, 6)$ & $8$   & $0$ & $(0, 0, 1, 6)$ & $15611882$  & $(8, 16, 24, 15)$ \\
$(0, 0, 3, 3)$ & $112$ & $0$ & $(0, 0, 3, 3)$ & $119488512$ & $(9, 18, 27, 15)$ \\
\addlinespace
$(1, 1, 0, 5)$ & $105$ & $0$ & $(1, 1, 0, 5)$ & $245188944$ & $(10, 19, 27, 16)$ \\
$(1, 0, 0, 7)$ & $7$   & $0$ & $(1, 0, 0, 7)$ & $24488568$  & $(9, 17, 25, 16)$ \\
$(0, 0, 2, 5)$ & $35$  & $0$ & $(0, 0, 2, 5)$ & $99243144$  & $(9, 18, 27, 16)$ \\
$(0, 0, 0, 8)$ & $1$   & $0$ & $(0, 0, 0, 8)$ & $4188834$   & $(8, 16, 24, 16)$ \\
$(0, 0, 2, 5)$ & $35$  & $0$ & $(0, 0, 2, 5)$ & $99243144$  & $(9, 18, 27, 16)$ \\
$(0, 2, 0, 4)$ & $168$ & $0$ & $(0, 2, 0, 4)$ & $307879299$ & $(10, 20, 28, 16)$ \\
$(2, 0, 0, 6)$ & $27$  & $0$ & $(2, 0, 0, 6)$ & $73465704$  & $(10, 18, 26, 16)$ \\
$(1, 0, 2, 4)$ & $189$ & $0$ & $(1, 0, 2, 4)$ & $404061372$ & $(10, 19, 28, 16)$ \\
$(0, 1, 0, 6)$ & $21$  & $0$ & $(0, 1, 0, 6)$ & $64194312$  & $(9, 18, 26, 16)$ \\
\addlinespace
$(2, 0, 1, 5)$ & $168$ & $0$ & $(2, 0, 1, 5)$ & $655589376$ & $(11, 20, 29, 17)$ \\
$(0, 1, 1, 5)$ & $112$ & $0$ & $(0, 1, 1, 5)$ & $487911424$ & $(10, 20, 29, 17)$ \\
$(1, 0, 1, 6)$ & $48$  & $0$ & $(1, 0, 1, 6)$ & $247606632$ & $(10, 19, 28, 17)$ \\
$(0, 0, 1, 7)$ & $8$   & $0$ & $(0, 0, 1, 7)$ & $50692096$  & $(9, 18, 27, 17)$ \\
$(1, 0, 1, 6)$ & $48$  & $0$ & $(1, 0, 1, 6)$ & $247606632$ & $(10, 19, 28, 17)$ \\
$(0, 0, 1, 7)$ & $8$   & $0$ & $(0, 0, 1, 7)$ & $50692096$  & $(9, 18, 27, 17)$ \\
\addlinespace
$(0, 0, 2, 6)$ & $35$  & $0$ & $(0, 0, 2, 6)$ & $318796478$ & $(10, 20, 30, 18)$ \\
$(1, 0, 0, 8)$ & $7$   & $0$ & $(1, 0, 0, 8)$ & $75985104$  & $(10, 19, 28, 18)$ \\
$(0, 0, 0, 9)$ & $1$   & $0$ & $(0, 0, 0, 9)$ & $12664184$  & $(9, 18, 27, 18)$ \\
$(2, 0, 0, 7)$ & $27$  & $0$ & $(2, 0, 0, 7)$ & $236722824$ & $(11, 20, 29, 18)$ \\
\bottomrule\end{tabular}
\caption{Associated graded for $\bigwedge^{11}\tangent_{G/P}$ for $(\mathrm{F}_4,\alpha_4)$}
\label{table:wedge-11-F4-P4}
\end{table}

\begin{table}
\centering
\begin{tabular}{crrcrcr}\toprule
weight & rank & degree & representation & dimension & sum of roots \\ \midrule
$(0, 0, 2, 5)$ & $35$  & $0$ & $(0, 0, 2, 5)$ & $99243144$  & $(9, 18, 27, 16)$ \\
\addlinespace
$(0, 0, 1, 7)$ & $8$   & $0$ & $(0, 0, 1, 7)$ & $50692096$  & $(9, 18, 27, 17)$ \\
$(1, 0, 1, 6)$ & $48$  & $0$ & $(1, 0, 1, 6)$ & $247606632$ & $(10, 19, 28, 17)$ \\
$(0, 1, 1, 5)$ & $112$ & $0$ & $(0, 1, 1, 5)$ & $487911424$ & $(10, 20, 29, 17)$ \\
\addlinespace
$(0, 0, 0, 9)$ & $1$   & $0$ & $(0, 0, 0, 9)$ & $12664184$  & $(9, 18, 27, 18)$ \\
$(2, 0, 0, 7)$ & $27$  & $0$ & $(2, 0, 0, 7)$ & $236722824$ & $(11, 20, 29, 18)$ \\
$(0, 1, 0, 7)$ & $21$  & $0$ & $(0, 1, 0, 7)$ & $203498568$ & $(10, 20, 29, 18)$ \\
$(1, 1, 0, 6)$ & $105$ & $0$ & $(1, 1, 0, 6)$ & $811481944$ & $(11, 21, 30, 18)$ \\
$(1, 0, 0, 8)$ & $7$   & $0$ & $(1, 0, 0, 8)$ & $75985104$  & $(10, 19, 28, 18)$ \\
$(0, 0, 2, 6)$ & $35$  & $0$ & $(0, 0, 2, 6)$ & $318796478$ & $(10, 20, 30, 18)$ \\
\addlinespace
$(1, 0, 1, 7)$ & $48$  & $0$ & $(1, 0, 1, 7)$ & $756760576$ & $(11, 21, 31, 19)$ \\
$(0, 0, 1, 8)$ & $8$   & $0$ & $(0, 0, 1, 8)$ & $150332598$ & $(10, 20, 30, 19)$ \\
\bottomrule\end{tabular}
\caption{Associated graded for $\bigwedge^{12}\tangent_{G/P}$ for $(\mathrm{F}_4,\alpha_4)$}
\label{table:wedge-12-F4-P4}
\end{table}

\begin{table}
\centering
\begin{tabular}{crrcrcr}\toprule
weight & rank & degree & representation & dimension & sum of roots \\ \midrule
$(0, 1, 0, 7)$ & $21$ & $0$ & $(0, 1, 0, 7)$ & $203498568$ & $(10, 20, 29, 18)$ \\
\addlinespace
$(1, 0, 1, 7)$ & $48$ & $0$ & $(1, 0, 1, 7)$ & $756760576$ & $(11, 21, 31, 19)$ \\
$(0, 0, 1, 8)$ & $8$  & $0$ & $(0, 0, 1, 8)$ & $150332598$ & $(10, 20, 30, 19)$ \\
\addlinespace
$(1, 0, 0, 9)$ & $7$  & $0$ & $(1, 0, 0, 9)$ & $216861554$ & $(11, 21, 31, 20)$ \\
$(0, 1, 0, 8)$ & $21$ & $0$ & $(0, 1, 0, 8)$ & $590446584$ & $(11, 22, 32, 20)$ \\
\bottomrule\end{tabular}
\caption{Associated graded for $\bigwedge^{13}\tangent_{G/P}$ for $(\mathrm{F}_4,\alpha_4)$}
\label{table:wedge-13-F4-P4}
\end{table}

\begin{table}
\centering
\begin{tabular}{crrcrcr}\toprule
weight & rank & degree & representation & dimension & sum of roots \\ \midrule
$(1, 0, 0, 9)$ & $7$ & $0$ & $(1, 0, 0, 9)$ & $216861554$ & $(11, 21, 31, 20)$ \\
\addlinespace
$(0, 0, 1, 9)$ & $8$ & $0$ & $(0, 0, 1, 9)$ & $412778496$ & $(11, 22, 33, 21)$ \\
\bottomrule\end{tabular}
\caption{Associated graded for $\bigwedge^{14}\tangent_{G/P}$ for $(\mathrm{F}_4,\alpha_4)$}
\label{table:wedge-14-F4-P4}
\end{table}

\begin{table}
\centering
\begin{tabular}{crrcrcr}\toprule
weight & rank & degree & representation & dimension & sum of roots \\ \midrule
$(0, 0, 0, 11)$ & $1$ & $0$ & $(0, 0, 0, 11)$ & $92512368$ & $(11, 22, 33, 22)$ \\
\bottomrule\end{tabular}
\caption{Associated graded for $\bigwedge^{15}\tangent_{G/P}$ for $(\mathrm{F}_4,\alpha_4)$}
\label{table:wedge-15-F4-P4}
\end{table}